\documentclass[10pt]{amsart}
\usepackage{amsmath,amssymb,amsthm}
\usepackage{float}
\usepackage{tikz}
\usetikzlibrary{matrix, fit}
\usetikzlibrary{backgrounds}
\usepackage{braket}
\usepackage{tikz-cd}
\usetikzlibrary{matrix}
\usetikzlibrary{cd}
\usepackage{mathtools}
\usepackage{comment}
\usepackage[colorlinks,hyperindex,linkcolor=blue]{hyperref}
\hypersetup{
pdfauthor = {L. Nakamura, A. del Pino Gomez},
 pdftitle= {Small energy isotopies of Legendrians in the complement of an overtwisted disc},
 pdfsubject = {}} 

\usepackage{pinlabel}
\usepackage[export]{adjustbox}

\usepackage{subcaption}
\usepackage{graphicx}
\newcommand{\D}{{\mathbb{D}}}
\newcommand{\R}{{\mathbb{R}}}
\newcommand{\N}{{\mathbb{N}}}

\newcommand{\NS}{{\mathbb{S}}}

\newcommand{\SC}{{\mathcal{C}}}

\renewcommand{\SS}{{\mathcal{S}}}

\newcommand{\eps}{\varepsilon}

\newcommand{\id}{{\operatorname{id}}}

\newcommand{\can}{{\operatorname{can}}}
\newcommand{\Op}{{\mathcal{O}p}}

\newcommand{\std}{{{\operatorname{std}}}}
\newcommand{\ev}{{{\operatorname{ev}}}}

\newcommand{\LagMon}{{\operatorname{Mon}_{Lag}}}

\newcommand{\LegEmb}{{\operatorname{Emb}_{Leg}}}

\newcommand{\OT}{{\operatorname{ot}}}

\newcommand{\Orange}{{\mathbb{O}}}

\newtheorem{proposition}{Proposition}[section]
\newtheorem{theorem}[proposition]{Theorem}
\newtheorem{definition}[proposition]{Definition}
\newtheorem{lemma}[proposition]{Lemma}

\newtheorem{corollary}[proposition]{Corollary}
\newtheorem{remark}[proposition]{Remark}
\newtheorem{question}[proposition]{Question}

\newcommand{\lu}[1]{{\color{orange} \small {#1}}}

\title{The size of an overtwisted disk}
\subjclass[2020]{Primary: 58A30. Secondary: 58A17, 57R57.}

\author{Lukas Nakamura}
\address{IMJ-PRG, CNRS, Sorbonne Université, Université Paris Cité}
\email{nakamura@imj-prg.fr}

\author{\'Alvaro del Pino}
\address{Utrecht University, Department of Mathematics, Budapestlaan 6, 3584 Utrecht, The Netherlands}
\email{a.delpinogomez@uu.nl}

\begin{document}

\begin{abstract}
We introduce a notion of size for an overtwisted disk, defined in terms of the Shelukhin-Hofer energy required to stabilise Legendrians in its vicinity. We compute this size for the standard $3$-dimensional overtwisted disk, which we use to prove that certain explicit product hyperdisks in higher-dimensional contact manifolds are overtwisted (and their overtwisted size can be computed explicitly as well).

We then use the overtwisted size to establish a number of quantitative results on contact flexibility. Concretely, we (1) obtain quantitative $h$--principles for Legendrians in the complement of an overtwisted disk, (2) construct overtwisted-at-infinity contact manifolds whose Legendrian Shelukhin-Hofer metric vanishes, (3) estimate overtwisted size under contact products and under products with Liouville manifolds, and (4) derive action bounds for contractible Reeb orbits in closed overtwisted contact manifolds.
\end{abstract}
\maketitle
\setcounter{tocdepth}{1}
\tableofcontents

\section{Introduction} \label{sec:introduction}

\subsection{Context}

In contact topology, the standard overtwisted disk is a certain germ of a contact structure along a codimension-$1$ disk. The precise definition of an overtwisted disk is not as important as its defining property: Contact structures containing it (which are said to be overtwisted) are classified by the $h$--principle \cite{El89,BEM}. This is one of the cornerstone results in the field.

That a concrete model is not so important was made precise by Eliashberg in \cite{El89}: he noted that any model that is formally trivial and implies the h-principle is equivalent to the overtwisted disk (in the sense that, for a closed contact manifold, the presence of one of them implies the presence of the other).

In dimension $3$, the standard overtwisted disk can be defined as follows. We take cylindrical coordinates $(\rho,\phi,z)$ in $\R^3$ and define the standard overtwisted-at-infinity contact form
\[ \alpha_\OT^\std := \cos(\rho) dz + \rho \sin(\rho) d \phi. \]
Then, the standard overtwisted disk is the germ of $\ker(\alpha_\OT^\std)$ along the smooth $2$-disk $\{z=0, \rho \leq \pi\}$.

In higher dimensions, the story is more subtle. The model introduced in \cite{BEM} is a germ, defined in terms of a certain universal constant, along a \emph{piecewise smooth} codimension-$1$ disk. It is a rather complicated object, but more tractable, equivalent, alternatives are known to exist \cite{Nie,Mur,MNPS,CMP,Hu17,CN25}. We single out a few:
\begin{itemize}
    \item[a. ] a Legendrian unknot with a loose chart,
    \item[b. ] a plastikstufe,
    \item[c. ] a sufficiently large ball in $\R^3 \times \R^{2n}$ endowed with the contact structure $\ker(\alpha_\OT^\std + \sum_i y_i dx_i)$, 
    \item[d. ] an overtwisted disk in $(\R^3,\alpha^\textnormal{std}_{\textnormal{ot}})$ times a sufficiently large ball in $(\R^{2n},\sum_i y_i dx_i)$, 
    \item[e. ] an overtwisted orange.
\end{itemize}

\subsection{An overtwisted hypersurface of explicit radius}

The object in item (d), as introduced in \cite{CN25}, is of particular importance in this article. One of our contributions is to make the size of the ball precise:
\begin{definition} \label{def:standardDisc}
Consider the action
\begin{equation*}
C^* := 2 \pi \max_{\rho \in [0,\pi]} \rho \sin(\rho) = 11.43\dots
\end{equation*}
of the unique simple closed Reeb orbit of $\alpha_\OT^\std$ contained in the standard $3$-dimensional overtwisted disk.

Let $n \geq 0$. We let $\Delta_\OT^\std$ be the germ of $\ker(\alpha_\OT^\std + \sum_i y_i dx_i)$ along the hyperdisk
\begin{equation*}
 \{z=0, \rho \leq \pi\} \times \left([-1,1] \times \left[-\frac{C^*}{2},\frac{C^*}{2}\right]\right)^n \subset \R^3 \times \R^{2n}.
\end{equation*}
\end{definition}

Our first theorem reads: 
\begin{theorem}\label{thm:NbhdOfModelOtDiskIsOT}
$\Delta_\OT^\textnormal{std}$ is overtwisted.
\end{theorem}
In particular, according to the $h$-principle, the presence of $\Delta_\OT^\textnormal{std}$ in a given contact structure is equivalent to the presence of a standard overtwisted disk in the sense of \cite{BEM}.

\subsubsection{Idea of the proof}

In order to establish Theorem \ref{thm:NbhdOfModelOtDiskIsOT} we will show that all neighbourhoods of $\Delta_\OT^\textnormal{std}$ contain a loose unknot \cite{CMP}. The starting point of the argument is the following result \cite{MNPS}: a Legendrian segment in a neighbourhood of the standard 3-dimensional overtwisted disk is Legendrian isotopic relative endpoints to its connected sum with $\partial \Delta_\OT^\std$ stabilised once. 

We show that one can construct such a Legendrian isotopy with precise quantitative control. To state it, we recall some terminology. Fix a strict contact manifold $(M,\alpha)$; we do not assume $M$ to be compact. We write $||-||_\alpha$ for the corresponding Shelukhin-Hofer norm for contact isotopies \cite{shelukhin2017hofer}. This induces a distance $d_\alpha$ between contactomorphisms (resp.\,Legendrians) by infimisation over all compactly-supported isotopies connecting them.

\begin{proposition} \label{prop:controlledIsotopy}
Let $f: L \rightarrow (\R^3,\alpha_\OT^\std)$ be a proper Legendrian embedding disjoint from the standard overtwisted disk $\Delta_\OT^\std$. Let $\eps >0$.

Then, there is a Legendrian $g: L \rightarrow (\R^3,\alpha_\OT^\std)$, obtained from $f$ by connect-summing with $\partial \Delta_\OT^\std$ stabilised once, and a compactly-supported contact isotopy between $f$ and $g$ of energy at most $\frac{C^*}{2}+\eps$.
\end{proposition}

\begin{remark} \label{rem:manystabilizations}
Proposition \ref{prop:controlledIsotopy} is a quantitative h-principle for loose Legendrians in dimension 3. A key observation is that, in order to establish flexibility in this setting, one requires arbitrarily many stabilizations, not just one. As we explain in Section \ref{sec:QuantitativeIsotopiesInDimension3}, this can be done by applying the proposition to many parallel copies of the overtwisted disk at once, so the cost is still bounded above by $\frac{C^*}{2}$.
\end{remark}

Crucially, the quantitatively controlled Legendrian isotopy of Proposition \ref{prop:controlledIsotopy} can be used to sweep out a loose unknot in $\Delta_\OT^\std \subset \R^{2n+1}$, $n>1$, establishing Theorem \ref{thm:NbhdOfModelOtDiskIsOT}. In fact, this loose unknot is contained in a neighbourhood of a certain stratified subset of $\Delta_\OT^\std \subset \R^{2n+1}$,  see Subsection  \ref{sssec:OTorange}, that turns out to be an overtwisted orange \cite{Hu17,Hu20}. In this manner, we recover the fact that a neighbourhood of an overtwisted orange is overtwisted. 

\subsubsection{Conventions in the paper} \label{sssec:conventions}

Before we continue, let us state some of the running assumptions throughout the paper. We will try to be explicit about these as we go along. The contact manifolds we study are assumed to be path-connected, but we do not ask that they are compact. We allow them to have boundary. The Legendrian submanifolds they contain need not be compact, but we ask that they are proper. We allow them to be disconnected, but they may only have finitely many components.

Whenever we work with diffeomorphisms and isotopies (including formal Legendrian isotopies) we ask that these are compactly supported. When the ambient manifold has boundary, this should be interpreted as being supported in the interior.

A $K$-family of maps $(f_k: X \to Y)_{k \in K}$ is compactly supported if there is some compact $A \subset K \times X$ such that $f_k(x) = f(x)$ for some fixed map $f$ whenever $(k,x) \notin A$. 

If $K$ is compact and $U \subseteq Y$ is an open subset, we will also say that $(f_k:X \to Y)_{k\in K}$ has compact support in $U$ if there exists a compact set $A \subseteq U$ such that for all $(k,x) \in K\times X$ with $f_k(x) \notin A$, it holds that $f_k(x) = f(x)$ for some fixed map $f$. 

\subsection{The size of an overtwisted disk}

The model overtwisted disk presented in Definition \ref{def:standardDisc} uses a very concrete choice of contact form. A large part of our work is dedicated to studying how overtwistedness is seen quantitatively as the contact form varies. We do this from a ``Legendrian perspective''.

First we define what it means for a Legendrian to have a ``visible'' stabilization of a certain size: 
\begin{definition}\label{def:stabilization}
Fix a strict contact manifold $(M,\alpha)$. Consider a proper and connected Legendrian embedding $g: L \rightarrow (M,\alpha)$.

Suppose $\dim M = 3$. If $g$ is oriented, we say that $g$ is an \textbf{$(N_+,N_-)$-fold $\eps$-stabilization} if it takes values in a strict Weinstein neighbourhood $W$ of height $\eps$ (Definition \ref{def:WeinsteinNeighbourhood}) of some other proper Legendrian embedding $h$ and, moreover, $g$ is obtained from $h$ by adding $N_+$ positive and $N_-$ negative stabilizations inside of $W$.

Suppose otherwise that $\dim M \geq 5$. Then, $g$ is an \textbf{$\eps$-stabilization} if $g$ takes values in a strict Weinstein neighbourhood $W$ of height $\eps$ of some other proper Legendrian embedding $h$ and, moreover, $g$ is a stabilization of $h$ in $W$.
\end{definition}
Recall that the theory of Legendrians is very different in dimension $3$ compared to higher dimensions, since stabilizations cannot ``self-replicate'' in dimension $3$, hence the distinction in the definition.

Following \cite{Nak24}, we can then discuss the cost of exhibiting such explicit stabilizations: 
\begin{definition}\label{def:stabilizationCost}
Fix a strict contact manifold $(M,\alpha)$ and an open subset $U \subseteq M$. Consider a proper and connected Legendrian embedding $f: L \rightarrow (M,\alpha)$.

Suppose $\dim M = 3$. The \textbf{stabilization cost} in $U$ of $f$ is
\[ \SC(f;M,U) := \limsup\limits_{N \to \infty} \SC^{N,N}(f;M,U) \]
where
\[ \SC^{N_+,N_-}(f;M,U) := \limsup\limits_{\eps \to 0} \inf\limits_{g} d_\alpha(f,g), \]
where $g: L \rightarrow (U,\alpha) \subseteq (M,\alpha)$ ranges over all $(N_+,N_-)$-fold $\eps$-stabilizations inside of Weinstein neighbourhoods contained in $U$.

Suppose $\dim M \geq 5$. The \textbf{stabilization cost} in $U$ of $f$ is
\[ \mathcal{C}(f;M,U) := \limsup\limits_{\eps \to 0} \inf\limits_{g} d_\alpha(f,g), \]
where $g: L \rightarrow (U,\alpha) \subseteq(M,\alpha)$ ranges over all $\eps$-stabilizations inside of Weinstein neighbourhoods within $U$.

If $U = M$, we simply write $\SC(f;M) \coloneqq \SC(f;M,M)$.
\end{definition}
We remark that the isotopies we infimise over, in order to define $d_\alpha$, are allowed to use the whole of $M$ and are compactly supported. They need not be the identity in the complement of $U$.

\begin{remark} \label{rem:Nakamura}
A notion of loose size for higher-dimensional Legendrians was defined in \cite{Nak24}. The results of \cite{Nak24} show that the loose size agrees with the stabilization cost we defined above, up to a factor of $2$. We discuss this in Lemma \ref{lem:Nakamura}.
\end{remark}

\begin{remark}\label{rmk:stabilizationCostDisconnected}
Suppose $L$ is disconnected and has finitely many components. Then we generalize Definitions \ref{def:stabilization} and \ref{def:stabilizationCost} as follows: an ($(N_+,N_-)$-fold) $\eps$-stabilization requires ($(N_+,N_-)$-fold) stabilizing every component of the given Legendrian $h$ within a strict Weinstein neighbourhood of height $\eps$.
\end{remark}

\begin{remark} \label{rmk:bothStabilizations}
The quantity $\mathcal{C}^{(N,N)}(f;M,U)$ does not depend on the chosen orientation of $L$ and is thus well-defined for $3$-dimensional, non-oriented Legendrians. This is why no orientation assumptions appear in Definition \ref{def:stabilizationCost}.
\end{remark}

We can then define the overtwisted size of a hyperdisk to be the supremum, over all Legendrians in the vicinity, of the stabilization/loose size. For the sake of generality (indeed, recall that some of the objects characterizing overtwistedness have more complicated topology), we allow more general subsets in the definition, not just hyperdisks.
\begin{definition}\label{def:OTsize}
Fix a strict contact manifold $(M,\alpha)$ and a closed subset $\Delta \subset M$ not equal to $M$. We require that $\Delta$ has a neighbourhood basis consisting of path-connected opens.

The \textbf{overtwisted size} of $\Delta$ is
\[ \SS(\Delta) := 2\sup\limits_{U} \sup\limits_{f} \mathcal{C}(f;U,U\setminus \Delta), \]
where $U$ ranges over all path-connected open neighbourhoods of $\Delta$ and $f : L \to (U,\alpha)$ ranges over all proper Legendrian embeddings, with finitely many components, that avoid $\Delta$.

Given $(M,\alpha)$ path-connected we define its overtwisted size as:
\[ \SS(M,\alpha) := \inf\limits_{\Delta} \SS(\Delta), \]
as $\Delta \subset M$ ranges over all codimension-1 disks.
\end{definition}
The factor of $2$ in the definition is introduced so that the overtwisted size compares more directly to various other quantities. The most obvious one is that the size then becomes an upper bound for the cost of Legendrian isotopies, see Theorem \ref{thm:BoundednessInComplementOfOtDiskOfSizeC}. Additionally, it compares well to spectral quantities, see Theorem \ref{thm:ClosedReebOrbitsInOTCtcMflds} and Remark \ref{rem:Nakamura} above.

In Subsection \ref{ssec:oneStabilization} we show that the overtwisted size is local, i.e.\,it depends solely on the germ of the contact form around $\Delta$. For now we can deduce the following related statement:
\begin{lemma}
If $\SS(\Delta)$ is finite, the germ of $\ker(\alpha)$ along $\Delta$ is overtwisted.
\end{lemma}
\begin{proof}
In higher dimensions, $\SS(\Delta)$ being finite implies that the unknot in a germ around $\Delta$ is loose, so overtwistedness follows from \cite{MNPS,CMP}. In dimension $3$, the assumption implies that the unknot can be destabilised arbitrarily many times, which contradicts the Thurston-Bennequin inequality, proving overtwistedness.
\end{proof}

Conversely, we can show that overtwisted disks have finite overtwisted size (showing thus that this is a meaningful notion to consider):
\begin{theorem} \label{thm:SizeOfOtDisk}
The overtwisted size of $\Delta^\std_\OT \subseteq (\R^{2n+1},\alpha^\textnormal{std}_\OT + \sum_i y_i dx_i)$ is exactly $C^*$.

In particular, the overtwisted size of the image of a contact embedding $\Delta_\OT^\textnormal{std} \rightarrow (M,\alpha)$ with conformal factor $\chi$ is bounded above by $C^* \max\limits_{\Delta_\OT^\std} |\chi|$.
\end{theorem}
The proof of this result relies on two ingredients. Showing that the size is bounded above by $C^*$ boils down to analysing the cost of the explicit Legendrian isotopy of Proposition \ref{prop:controlledIsotopy}. Proving that $C^*$ is a lower bound requires us to relate the overtwisted size to the Reeb spectrum via the energy-capacity inequality \cite{DRS24}. We discuss this further in the next Subsection \ref{ssec:SizeAndReeb}.

Observe that Theorem \ref{thm:SizeOfOtDisk} implies that $\Delta^\std_\OT$ is overtwisted, establishing Theorem \ref{thm:NbhdOfModelOtDiskIsOT}. In fact, we can also establish the following generalization:
\begin{theorem}\label{thm:OtDiskOfSizeCTimesBall}
Fix a strict contact manifold $(M,\alpha)$ of arbitrary dimension. Consider a compact subset $\Delta \subset (M,\alpha)$ of overtwisted size $C$. Then for any $n \geq 1$,
\[  \Delta^{(n)} := \Delta \times \left([-1,1] \times \left[-\frac{C}{2},\frac{C}{2}\right]\right)^n \subseteq \left(M \times \R^{2n}, \alpha + \sum_{i=1}^n y_i dx_i\right) \]
has overtwisted size at most $C$.
\end{theorem}
Theorem \ref{thm:OtDiskOfSizeCTimesBall} stands in contrast to the results of \cite{HCMMP20}, that show that the product of any closed contact manifold with a sufficiently small ball is tight; see also \cite[Theorem G]{Gir20} for a partial result in this direction. While it was known that the product of an overtwisted contact manifold with a sufficiently large ball is overtwisted \cite{NiederkrugerPresas2010,CMP,CN25}, Theorem \ref{thm:OtDiskOfSizeCTimesBall} provides an explicit size for the ball.

The stabilised model presented in Theorem \ref{thm:OtDiskOfSizeCTimesBall} can be found in products with Liouville manifolds, showing:
\begin{corollary}\label{lem:OTdiskInLiouvilleProd}
Suppose $(M,\alpha)$ is a strict contact manifold with complete Reeb flow and overtwisted size $C$. Let $(N, \omega = d\lambda)$ be an exact symplectic manifold with complete Liouville flow. Then $(M \times N, \alpha+\lambda)$ has overtwisted size $\leq C$.
\end{corollary}

Lastly, we pose:
\begin{question}
    What is the minimal value of $C$ such that
    \begin{equation*}
         \{z=0, \rho \leq \pi\} \times \left([-1,1] \times \left[-\frac{C}{2},\frac{C}{2}\right]\right)^n \subseteq \left(\R^{3+2n}, \alpha^\textnormal{std}_\OT + \sum_{i=1}^ny_i dx_i\right)
    \end{equation*}
    is overtwisted?
\end{question}
Ajij, Eliashberg, Mj, Pancholi, and Polterovich have announced work on a contact homology theory that can be used to provide a lower bound for $C$. It would be interesting to compare their bound with the upper bound provided by Theorem \ref{thm:OtDiskOfSizeCTimesBall}.

\subsection{Size and Reeb orbits} \label{ssec:SizeAndReeb}

According to Theorem \ref{thm:SizeOfOtDisk}, the model $\Delta^\std_\OT$, equipped with the standard overtwisted contact form, contains a simple closed Reeb orbit whose action is equal to the overtwisted size. Now we explore how the overtwisted size, for a general contact form, relates to the action of Reeb orbits.

While an open overtwisted contact manifold may contain no closed Reeb orbits (see \cite{PP15} for instance), it is known that any closed overtwisted contact manifold contains a closed contractible Reeb orbit. Indeed, it was shown in \cite{Ho} (see \cite{Nie,ah09,NR} for the higher dimensional case) that a Bishop family of pseudoholomorphic disks, born at the center of a $3$-dimensional overtwisted disk $(\Delta,\alpha)$, will explode and produce a finite energy plane bounding a contractible Reeb orbit. The action of this Reeb orbit is bounded above by the action of the boundaries of the Bishop disks, which we can bound above by:
\[ c_\textrm{Hofer} := \int_0^{2\pi} \max_{\rho \in [0,\pi]} |\alpha(\rho,\theta)(\partial_\theta)| d\theta,  \]
where $(\rho,\theta)$ are polar coordinates. Our argument in Theorem \ref{thm:SizeOfOtDisk} can readily be adapted to show:
\begin{proposition} \label{prop:SizeOfOTDisk}
Let $\Delta \subset (M,\alpha)$ be an overtwisted disk in a $3$-dimensional contact manifold. Its overtwisted size is bounded above by $c_\textrm{Hofer}$. 
\end{proposition}
In higher dimensions, the arguments in \cite{ah09} also yield an upper bound for the action of the closed Reeb orbit, but it is harder to compare to the overtwisted size, since the model used in \cite{ah09} is that of a plastikstufe.

Whereas this proposition provides an upper bound for the overtwisted size, the striking feature of Theorem \ref{thm:SizeOfOtDisk} is that we are able to compute the size exactly. This relies on the following result:
\begin{theorem}\label{thm:ClosedReebOrbitsInOTCtcMflds}
    For any $C>0$ and $n \geq 1$, there exists a number $R > 0$ such that any closed strict contact manifold $(M,\alpha)$ of dimension $2n+1$, containing a hyperdisk $\Delta$ of overtwisted size $C$ and a strict Darboux ball $B(R)$ of radius $R$ with $\Delta \cap B(R) = \emptyset$, contains a contractible closed Reeb orbit of action at most $C$. Moreover, for fixed $n$, the constant $R(C)$ can be assumed to go to zero as $C$ goes to zero.
\end{theorem}

Theorem \ref{thm:ClosedReebOrbitsInOTCtcMflds} can be used to prove the existence of closed Reeb orbits for contact manifolds with boundary if we assume that the contact form is large away from the overtwisted disk.
\begin{corollary}\label{cor:closedReebOrbit}
    Let $(M,\alpha)$ be a strict contact manifold, possibly with boundary, $\Delta \subseteq M$ a hyperdisk of overtwisted size $C$, and $U \subseteq M$ an open neighbourhood of $\Delta$ with $\overline U \subseteq M^\circ$.
    
    Then, there exists $\eps > 0$ such that for any positive function $\chi: M \to \R$ satisfying $\chi|_{M \setminus U} = 1/\eps$ and $\chi|_\Delta = 1$, the contact form $\chi \alpha$ admits a contractible closed Reeb orbit in $M$ of action at most $C$.
\end{corollary}
In particular, if we apply Corollary \ref{cor:closedReebOrbit} to the case where $M$ is a small closed neighbourhood of $\Delta$, then we can conclude the existence of closed Reeb orbits close to $\Delta$ for contact forms that grow sufficiently fast near $\Delta$.

\begin{question}
    Does any closed strict contact manifold $(M,\alpha)$ of overtwisted size $C$ contain a contractible closed Reeb orbit of action at most $C$?
\end{question}

 

\subsection{Consequences regarding Legendrian isotopies}

In overtwisted contact manifolds, Legendrians in the complement of an overtwisted disk are loose, and they satisfy the h-principle \cite{Dy,etn13,BEM,Cha22,cf24}; see Theorem \ref{thm:h-principleLooseLegDim3}. By using that the cost of stabilising Legendrians in the complement of an overtwisted disk is bounded by the overtwisted size, we can quantify this h-principle.

\begin{theorem}\label{thm:BoundednessInComplementOfOtDiskOfSizeC}
Consider a hyperdisk $\Delta \subset (M,\alpha)$ of overtwisted size $C$ with $M$ path-connected. Then, any two proper Legendrian embeddings $f_0,f_1:L \to M$ that avoid $\Delta$ and that are formally Legendrian isotopic satisfy
\[ d_\alpha(f_0,f_1) \leq C \,. \]
\end{theorem}
The reader should compare this to the main result of \cite{Nak24}, which is a quantitative version of Murphy's h-principle for loose Legendrians \cite{Mur}. For disks modeled on $\Delta_\OT^\std$ it then follows that: 
\begin{corollary} \label{thm:small energy isotopies in the complement of an overtwisted disk}
Consider $\Delta_\OT^\std \rightarrow (M,\alpha)$ with conformal factor $\chi$ and $M$ path-connected. Then any two formally isotopic proper Legendrian embeddings $f_0,f_1: L \to M$ avoiding $\Delta_\OT^\std$ satisfy
\[ d_\alpha(f_0,f_1) \leq C^*\max\limits_{\Delta_\OT^\std} |\chi|. \]
\end{corollary}

By \cite{DRS24}, $d_\alpha$ defines non-degenerate metric on the space of Legendrian submanifolds, for any closed contact manifold. In contrast, Corollary \ref{thm:small energy isotopies in the complement of an overtwisted disk} gives a general construction of Legendrian isotopy classes with vanishing Shelukhin-Chekanov-Hofer metric on open contact manifolds:
\begin{corollary}\label{cor:LegegendriansWithVanishingd_alpha}
Let $(M,\xi)$ be an open path-connected contact manifold that is overtwisted at infinity. Then $(M,\xi)$ admits a contact form $\alpha$ for which the Shelukhin-Chekanov-Hofer metric vanishes on all isotopy classes of closed Legendrian embeddings into $M$.
\end{corollary}

To the authors' knowledge, all previously known examples of Legendrian isotopy classes with vanishing $d_\alpha$ arise as lifts of exact, displaceable Lagrangians to the contactization \cite{Can23,Nak24a}. In symplectizations of overtwisted contact manifolds, such Lagrangians always exist \cite{Mur13}. In this case, Corollary \ref{cor:LegegendriansWithVanishingd_alpha} states that not only these lifts, but in fact all closed Legendrians have vanishing $d_\alpha$. 

In higher dimensions, two Legendrians that are formally isotopic and lie in the complement of two different overtwisted disks are Legendrian isotopic, since they are loose. In dimension 3, this is no longer true in general \cite{Vo3}, and we need to require that the two overtwisted disks are contact isotopic. Correspondingly, the following generalization of Theorem \ref{thm:BoundednessInComplementOfOtDiskOfSizeC} holds.
\begin{theorem}\label{thm:BoundMovingOtDisk}
    Let $(M,\alpha)$ be a path-connected strict contact manifold. Let $\Delta_0, \Delta_1 \subseteq (M,\alpha)$ be hyperdisks of overtwisted sizes $C_0$ and $C_1$. Let $f_0,f_1:L \to M$ be two proper Legendrian embeddings such that $f_i$ avoids $\Delta_i$ and assume that $f_0$ is formally isotopic to $f_1$. If $\dim M =3$, we also assume that $\Delta_0$ and $\Delta_1$ are contact isotopic. Then we have that
    \begin{equation*}
        d_\alpha(f_0,f_1) \leq \frac{C_0}{2}+\frac{C_1}{2}.
    \end{equation*}
\end{theorem}

\begin{remark}\label{rmk:BoundMovingOtDiskFixedHmtpClass}
    For $\eps > 0$, the Legendrian isotopy from $f_0$ to $f_1$ of energy at most $\frac{C_0}{2} + \frac{C_1}{2} + \eps$ guaranteed by Theorem \ref{thm:BoundMovingOtDisk} can be chosen to have any prescribed (formal) Legendrian isotopy class, see Remark \ref{rmk:SmallEnergyIsotopiesWithFixedHmtpClassMovingOtDisk}. In other words, Theorem \ref{thm:BoundMovingOtDisk} extends to the universal cover of the space of Legendrians in $M$.
\end{remark}

It is known that bounds on the distances of Legendrians related by the Reeb flow are directly related to the existence of positive loops \cite{HS24}. Thus, Theorem \ref{thm:BoundMovingOtDisk}  and Remark \ref{rmk:BoundMovingOtDiskFixedHmtpClass} imply the following.
\begin{corollary}\label{cor:PosLoops}
    Let $\Delta \subseteq (M,\alpha)$ be a hyperdisk of overtwisted size $C$ in a path-connected strict contact manifold $(M,\alpha)$ with complete Reeb flow. Then for any $\eps > 0$, any closed Legendrian embedding $f:L \to M$ avoiding $\Delta$ is part of a contractible positive loop of Legendrian embeddings of energy at most $2C + \eps$.
\end{corollary}

\subsection{Consequences for contact products and the Legendrian graph}

Given two strict contact manifolds $(M_1,\alpha_1)$ and $(M_2,\alpha_2)$, their contact product is defined as $(M_1 \times M_2 \times (0,\infty), \alpha_2-s \alpha_1)$, where s denotes the coordinate on $(0,\infty)$.
\begin{theorem}\label{thm:vanihsingOTsizeInCtcProd}
    Assume that $(M_1,\alpha_1)$ is overtwisted with complete Reeb flow. Then the complement of any compact subset in $(M_1 \times M_2 \times (0,\infty), \alpha_2-s \alpha_1)$ has vanishing overtwisted size.
\end{theorem}

By applying Theorem \ref{thm:vanihsingOTsizeInCtcProd} to $M_2$ replaced by a pre-compact open ball $U$, we see that also the complement of any compact subset of $M_1 \times \overline{U} \times \R$ contains disks of arbitrarily small overtwisted size. We can conclude from this observation that the geometry of the contactomorphism group of $(M,\alpha)$ and geometry of the isotopy classes of Legendrian graphs of contactomorphisms of $(M,\alpha)$ can be very different. Namely, recall that to any contactomorphism $\phi: (M,\alpha) \to (M,\alpha)$ we can associate the proper Legendrian embedding
\begin{align*}
    \Gamma_\phi: M &\to (M\times M \times (0,\infty), \pi_2^*\alpha- s \pi_1^*\alpha )\,,\\
    x &\mapsto (x,\phi(x),\chi(x))
\end{align*}
where $\pi_1$ and $\pi_2$ denote the projections onto the first and second factor of $M$, $s$ denotes the coordinate on $(0,\infty)$, and $\chi$ denotes the conformal factor of $\phi$. It then immediately follows from Theorem \ref{thm:BoundednessInComplementOfOtDiskOfSizeC} that:
\begin{corollary}\label{cor:ContactomorphismGroupVSLegendrianGraph}
Let $(M,\alpha)$ be a path-connected overtwisted strict contact manifold with complete Reeb flow. Then:
\begin{itemize}
    \item the pseudo-metric $d_{\pi_2^*\alpha - s \pi_1^*\alpha}$ vanishes on the isotopy class of the graph $\Gamma_\textnormal{id}$ of the identity,
    \item but $d_\alpha$ is non-degenerate on the identity component of the group of compactly supported contactomorphisms of $(M,\alpha)$.
\end{itemize} 
\end{corollary}
The second item is established in \cite{shelukhin2017hofer}. Corollary \ref{cor:ContactomorphismGroupVSLegendrianGraph} should be compared to a result of Allais and Arlove \cite{AA25a}, who showed that if $M$ is compact and $\Gamma_\textnormal{id}$ is orderable (and thus non-loose \cite{liu2020positive}), then for any two contactomorphisms $\phi_0, \phi_1:M \to M$, the distance $d_\alpha(\phi_0,\phi_1)$ is equal to $d_{\pi_2^*\alpha - s \pi_1^*\alpha}(\Gamma_{\phi_0},\Gamma_{\phi_1})$, provided that $\phi_0$ and $\phi_1$ are sufficiently $C^1$-close to the identity.



\subsection{Outline of the paper}

We recall the relevant contact topological background in Section \ref{sec:Preliminaries}.

The first milestone in the paper is Proposition \ref{prop:controlledIsotopy}, which tells us that there is a universal bound for the cost of isotoping a loose Legendrian in dimension $3$ so that it exhibits a stabilization; we establish this in Section \ref{sec:QuantitativeIsotopiesInDimension3}. To get there we first need some preliminary results about isotoping Legendrians so part of them is close to the overtwisted disk; this is technical but rather standard and is done in  Section \ref{sec:bringing Legendrian in standard position}.

With Proposition \ref{prop:controlledIsotopy} at hand, we can establish Theorem \ref{thm:NbhdOfModelOtDiskIsOT} (i.e.\,show that our hypersurface $\Delta_\OT^\std$ is overtwisted) and its generalization Theorem \ref{thm:OtDiskOfSizeCTimesBall} (i.e.\,overtwisted size behaves well if we multiply with a Liouville ball of controlled size). These are established in Section \ref{sec:ProofStabOtSize}.

We then tackle the most technical part of the paper, which is showing that the overtwisted size provides universal control for the cost of isotopies between loose Legendrians. We do this in Sections \ref{sec:smallEnergyIsotopies} and \ref{sec:energyBoundsLegendrianIsotopies}. This includes the proofs of Theorem \ref{thm:BoundednessInComplementOfOtDiskOfSizeC} and Theorem \ref{thm:BoundMovingOtDisk}.

We then study the relationship between overtwisted size and Reeb spectrum (Section \ref{sec:ClosedReebOrbitsInOTCtcMflds}). This interplay is what allows us to compute the exact size of the standard $\Delta_\std^\OT$ (Theorem \ref{thm:SizeOfOtDisk}).

We close the article (Section \ref{sec:products}) discussing the results on contact products and products with Liouville manifolds.

\subsection*{Disclosure of use of AI}

This paper was developed and written fully by its authors. Once the paper was complete, had been polished by us, and we were ready to upload to arXiv, we asked Chat-GPT Astra to produce a referee report. Said report found typos, small mistakes in computations and estimates, and missing hypotheses in statements. We iterated this a couple of times until we were satisfied. All mistakes that may remain in the paper are due to us.

\subsection*{Acknowledgements} The second author is funded by the NWO grant Vidi.223.118 “Flexibility and rigidity of tangent distributions”. The second author is supported by the Fondation Sciences Mathématiques de Paris.

\section{Preliminaries}\label{sec:Preliminaries}

We now recall some of the fundamentals of contact topology. We do so to set notation. We follow the conventions introduced in Subsection \ref{sssec:conventions}.

\subsection{Contact manifolds}

A strict contact manifold is a pair $(M^{2n+1},\alpha)$ consisting of a $(2n+1)$-dimensional manifold $M$ and a contact form $\alpha$, i.e.\,$\alpha \wedge d \alpha^n$ is a volume form. Its kernel $\xi = \ker(\alpha)$ is thus a contact structure, that is to say, a hyperplane field that is maximally non-integrable.

\subsubsection{Contactomorphisms}

A contactomorphism $\phi: (M,\alpha) \rightarrow (N,\beta)$ is a diffeomorphism such that $\phi_*\ker(\alpha) = \ker(\beta)$. Equivalently, $\phi^*\beta = \chi\alpha$, for some nowhere vanishing function $\chi: M \to \R$, which we call \emph{the conformal factor} of $\phi$. If $\chi$ is positive, it means that $\phi$ preserves coorientations. We say that $\phi$ is \emph{strict} if its conformal factor is $1$.

\subsubsection{Contact isotopies}

The contact form $\alpha$ induces a one-to-one correspondence between compactly supported time-dependent functions $(H_t)_{t \in [0,1]}: M \to \R$ and compactly supported contact isotopies $\psi_t: (M,\alpha) \to (M,\alpha)$ starting at the identity. Namely, to a contact isotopy $\phi^H_t$ we associate the function $H_t$ defined via $H_t \circ \phi^H_t = \alpha(\partial_t \phi^H_t)$. The conformal factor of $\phi^H_t$ will be denoted as $\chi^H_t$.

Fix two contact isotopies $\phi_t$ and $\psi_t$ with Hamiltonians $H_t$ and $G_t$. We write $\phi_t * \psi_t$ for their concatenation, with the convention that $\phi_t$ is performed first (i.e.\,we concatenate left to right). The corresponding Hamiltonian satisfies
\[ (H * G)_t \circ (\phi_t * \psi_t) = 2H_{2t} \circ\phi_{2t} \quad\text{ if } \quad t \in [0,1/2] \]
and
\[ (H * G)_t \circ (\phi_t * \psi_t) = 2G_{2t-1} \circ \psi_{2t-1} \circ \phi_1 \quad\text{ for} \quad t \in [1/2,1]. \]
In practice, we often assume that the isotopies are constant whenever $t$ is close to $0$ and $1$, so the concatenation is also smooth in the parameter at $t=1/2$.

\subsection{Legendrian embeddings}

A Legendrian embedding $f:L \to (M^{2n+1},\alpha)$ is an embedding of an $n$-manifold $L$ such that $f^*\alpha = 0$. A Legendrian isotopy is a smooth family of Legendrian embeddings $(f_t)_{t \in [0,1]}: L \to (M,\alpha)$.

\subsubsection{Weinstein neighbourhoods}

Fix an $n$-dimensional manifold $L$. Its $1$-jet space $(J^1L = T^*L \times \R, dz + \lambda_\can)$ is a strict contact manifold in which the inclusion of the zero-section $\iota: L \rightarrow J^1L$ is Legendrian. Moreover, it is a universal model for neighbourhoods of Legendrian embeddings:
\begin{definition} \label{def:WeinsteinNeighbourhood}
Fix a manifold $K$ (possibly with boundary and corners) and a Riemannian metric on $L$. Consider a compactly-supported family of proper Legendrian embeddings $(f_t)_{t \in K}: L \to (M,\alpha)$.

A (parametric) \textbf{strict Weinstein neighbourhood} of $f_t$ is a  compactly supported family
\[ (\phi_t)_{t \in K}: \{|z| < \delta^z_t, |p| < \delta^p_t \} \subseteq (J^1L,dz + \lambda_\can) \to (M,\alpha) \]
of strict contact embeddings satisfying $f_t = \phi_t|_L$, for some continuous families of functions $(\delta^z_t)_{t \in K},(\delta^p_t)_{t \in K}: L \to (0,\infty)$. We say that $\phi_t$ has \textbf{height} $2 \sup_L \delta^z_t$.

We will sometimes also omit the embeddings $\phi_t$ from the notation and call their images $(\phi_t(\{|z| < \delta^z_t, |p| < \delta^p_t \}))_{t \in K}$ a (parametric) strict Weinstein neighbourhood. 
\end{definition}

Then:
\begin{proposition}\label{thm:StrictWeinsteinNeighbrouhoods}
Any compactly-supported family of proper Legendrian embeddings $(f_t)_{t \in K}: L \to (M,\alpha)$ admits a strict Weinstein neighbourhood.
\end{proposition}
\begin{proof}
The case when $K = \{*\}$ and $L$ is closed is proven in \cite[Theorem 6.2.2]{Gei2}. It is easy to verify that all choices in said construction are contractible so that the proof also works when $K$ and $L$ are compact. To deal with non-compactness, recall that the proof relies on Moser's method. The vector fields involved in the proof vanish along $L$, so by restricting to sufficiently small neighbourhoods of the zero-section in $J^1L$, we can ensure that their flows exist on a given compact parameter space. As $f_t$ does not depend on $t$ outside of a compact subset of $L$, we can fix such a neighbourhood near infinity, and on each compact subset of $K$, we can find functions $\delta^z_t, \delta^p_t,\phi_t$, which depend only on contractible choices. Using a compact exhaustion of $K$, the construction extends to the general case.
\end{proof}

\subsubsection{Contact isotopy extension}

According to the contact isotopy extension theorem, a Legendrian isotopy $f_t$ is always induced by an ambient contact isotopy $\phi_t$, meaning that $\phi_t \circ f_0 = f_t$. This is also true, more generally, for isotropic isotopies. As in the smooth case, these statements are established by extending the vectors tangent (in the time direction) to the isotopy to global contact vector fields and integrating. As such, the contact isotopy extension theorem is parametric in nature. 

The parametric nature of contact isotopy has the following useful consequence:
\begin{lemma} \label{lem:Serre}
Let $\LegEmb(L;M,\alpha)$ be the space of Legendrian embeddings of $L$ in $(M,\alpha)$ endowed with the weak $C^\infty$ topology. Consider the evaluation map
\[ \ev: \LegEmb(\D^n;M,\alpha) \rightarrow M \]
sending an embedding $\psi$ to $\psi(0) \in M$. This is a Serre fibration.
\end{lemma}

Let $\LagMon(\R^n,\ker(\alpha))$ be the space of Lagrangian frames; i.e.\,monomorphisms $\R^n \rightarrow \ker(\alpha)$ from the trivial $n$-bundle to the contact structure. The following stronger statement is true:
\begin{lemma} \label{lem:Serre2}
The evaluation map
\[ \ev: \LegEmb(\D^n;M,\alpha) \rightarrow \LagMon(\R^n,\ker(\alpha)) \]
sending an embedding $\psi$ to $d_0\psi$ is a Serre fibration and a weak homotopy equivalence.
\end{lemma}
The statement follows from the $h$-principle, although it can also be proven via the following elementary ``Legendrian Alexander-trick '' argument: use dilations in $\D^n$ in order to ``restrict'' any finite dimensional family of embeddings to an arbitrarily small ball; then, use linear interpolation in a Weinstein neighbourhood to match any two maps that share the same first order data at the origin.

We will use these results to carry out various constructions. Applying these produces families of Legendrian embeddings that are continuous for the Whitney topology, but not necessarily smooth. To address this, we will invoke:
\begin{lemma} \label{lem:smoothingFamilies}
A finite-dimensional family of Legendrian embeddings, continuous for the weak $C^\infty$-topology, can be smoothed. This can be done relatively in parameter and domain.
\end{lemma}
The proof is standard: given any Legendrian in the family we consider a Weinstein neighbourhood, where it is seen as the zero section. Nearby Legendrians in the family are then described as $1$-jets of functions, meaning that they can be smoothed by smoothing the underlying  Whitney-continuous family of functions.

\subsection{Flexibility for Legendrians}

\subsubsection{Formal Legendrians}

Flexibility results in contact topology are often stated in terms of h-principles. Roughly, these assert that some formal data already determines the genuine data we are interested in. We will need the notion of formal Legendrians: A \emph{formal Legendrian embedding} into a contact manifold $(M,\xi)$ is a pair $(f,F_s)$ consisting of an embedding $f:L \to M$ and a homotopy of injective tangent maps $F_s: TL \to f^*TM$ with $s \in [0,1]$ covering $f$ such that $F_0 = df$ is the differential of $f$ and $F_1(T_xL) \subseteq T_{f(x)}M$ is a Lagrangian subspace of $\xi_{f(x)}$ with respect to its conformal symplectic structure for all $x \in L$. A Legendrian embedding $f:L\to M$ is identified with the formal Legendrian embedding $(f,df)$.

\subsubsection{Stabilized Legendrians in dimension 3} 

Given a Legendrian knot in a contact 3-manifold, there is a local operation that changes the Legendrian knot type, called \emph{stabilization}.
\begin{definition}
    Let $f:L \to (M^3,\xi)$ be a Legendrian embedding. Let $x \in L$, and choose a Weinstein neighbourhood for $f$ near the point $x$. The Legendrian embedding obtained from $f$ by replacing the front of $f$ in this Weinstein neighbourhood by a zig-zag is called a \textbf{stabilization of $f$}.
\end{definition}

\begin{remark}
    We can also define stabilizations of \emph{formal} Legendrians as follows: let $(f,F_s)$ be a formal Legendrian embedding. After a $C^0$-small isotopy, we may assume that $(f,F_s)$ is Legendrian near a point, where we can now perform the stabilization as before. By using parametric Weinstein neighbourhoods of a family of (formal) Legendrians, stabilizations can also be defined parametrically.
\end{remark}

There are two non-equivalent types of stabilizations. Given an orientation of the Legendrian, these two types are called \textbf{positive} and \textbf{negative}, see Figure \ref{fig:Legendrian stabilizations in dimension 3.}. Note that these are interchanged if the orientation of $L$ is flipped.

\begin{figure}
\centering
\begin{subfigure}{0.30\textwidth}
    \includegraphics[width=\textwidth]{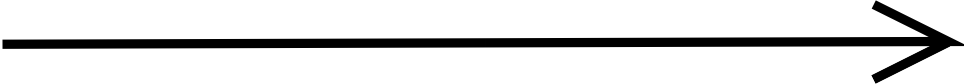}
    \caption{An oriented Legendrian arc $\gamma$.}
    \label{fig:arc}
\end{subfigure}
\hfill
\begin{subfigure}{0.30\textwidth}
    \includegraphics[width=\textwidth]{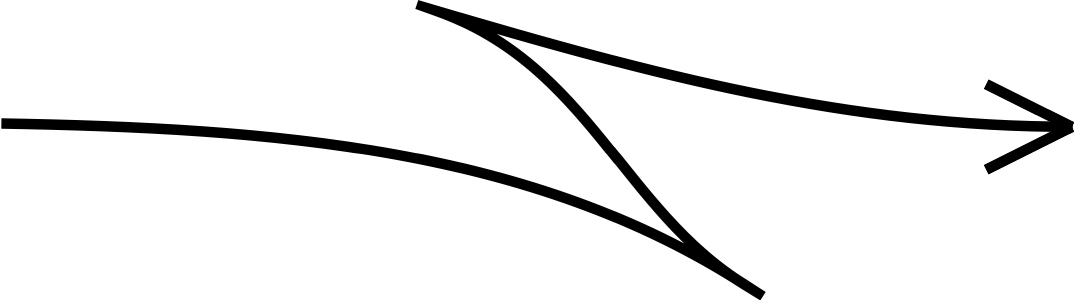}
    \caption{A positive stabilization of $\gamma$.}
    \label{fig:positive_stabilization}
\end{subfigure}
\hfill
\begin{subfigure}{0.30\textwidth}
    \includegraphics[width=\textwidth]{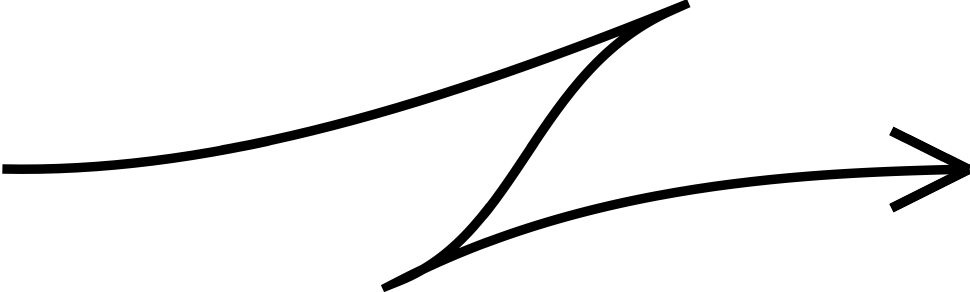}
    \caption{A negative stabilization of $\gamma$.}
    \label{fig:negative_stabilization}
\end{subfigure}

\caption{The two types of stabilizations of an oriented Legendrian knot in the front projection.}
\label{fig:Legendrian stabilizations in dimension 3.}
\end{figure}

\subsubsection{The h-principle for Legendrian knots}

Fuchs and Tabachnikov \cite{FT} showed that two Legendrian knots in standard contact $\R^3$ that are isotopic as smooth knots become Legendrian isotopic after sufficiently many stabilizations. We will need a stronger version of this h-principle, dealing with families of Legendrian links. A proof can be found in \cite[Theorem 7]{CP} for the case of knots, which generalizes in a straightforward fashion:
\begin{theorem}\label{thm:hprincipleLegLinks}
    Let $(M,\xi)$ be a $3$-dimensional contact manifold and $(f_t,F_{s,t})$ be a family of proper formal Legendrian embeddings with $t \in I$, where $I$ is a compact manifold. Assume that $(f_t,F_{s,t})$ is Legendrian in a neighbourhood of a closed subset $A \subseteq I$. For any $\eps > 0$, there exists, after stabilizing $(f_t,F_{s,t})$ sufficiently often, a homotopy of formal isotopies from $(f_t,F_{s,t})$, fixed near $A$, to a family of Legendrian embeddings $g_t$ whose underlying family of isotopies is $\eps$-close to $f_t$ in the $C^0$-norm.
\end{theorem}


\subsubsection{Loose Legendrians in dimension 3}

Suppose $(M,\xi)$ is an overtwisted contact 3-manifold. Then, a Legendrian embedding $L \rightarrow (M,\alpha)$ is \emph{loose} if its complement is overtwisted (as defined in the introduction). Such loose Legendrians satisfy the h-principle \cite{EliFras,Dy}:

\begin{theorem}\label{thm:h-principleLooseLegDim3}
    Let $\Delta \subseteq (M,\xi)$ be an overtwisted disk and $f_0,f_1: L \rightarrow M \setminus \Delta$ two Legendrian embeddings. Any formal Legendrian isotopy from $f_0$ to $f_1$ is homotopic with fixed end-points to a genuine Legendrian isotopy.
\end{theorem}

\subsubsection{Loose Legendrians in higher dimensions}

In higher dimensional contact manifolds, loose Legendrians were introduced by Murphy \cite{Mur}. 

\begin{definition}
    Let $(M,\xi)$ be a contact manifold of  dimension $3 + 2n \geq 5$, and let $L \subseteq M$ be a connected Legendrian submanifold. A \textbf{loose chart} for $L$ is a chart in $M$ contactomorphic to 
    \begin{equation*}
        [0,1]^3 \times \{|q| < \rho, |p| < \rho\} \subseteq \left(\R^3 \times T^* \R^{n}, dz + ydx + \sum_ip_i dq_i\right)
    \end{equation*}
    for some $\rho > 1$, in which $L$ is equal to $L_0 \times Z_\rho$, where $L_0 \subseteq \R^3$ is a Legendrian arc that agrees with $\{z=y=0\}$ near the boundary and whose front is a zig-zag , and $Z_\rho = \{p = 0\} \subseteq \{|q| < \rho, |p| < \rho\}$ is the zero section. 

    If $L$ is disconnected, a loose chart for $L$ is a collection of mutually disjoint loose charts, one for each component of $L$. If $f:L \to M$ is a Legendrian embedding, a loose chart for $f$ is a loose chart for its image $f(L)$. A Legendrian (embedding) is called \textbf{loose} if it admits a loose chart.
\end{definition}

Murphy showed that loose Legendrians satisfy an h-principle:

\begin{theorem}[\cite{Mur}]\label{thm:LooseLegh-principle}
    Let $f_0,f_1:L \to (M,\xi)$ be two proper loose Legendrian embeddings that are compactly supported formally Legendrian isotopic. Then there exists a compactly supported Legendrian isotopy from $f_0$ to $f_1$.
\end{theorem}

\begin{remark}
    In \cite{Mur}, only the case of closed and connected Legendrians is considered. In \cite[Theorem 3.1]{Nak24}, it is shown that the h-principle also applies to proper Legendrian embeddings.  Furthermore, the proof of the h-principle involves only the loose charts and a $C^0$-small neighbourhood of the formal Legendrian isotopy; see the proof of \cite[Theorem 3.1]{Nak24}. Therefore, it also applies to possibly disconnected Legendrian embeddings as long as we have disjoint loose charts for all components of $L$.
\end{remark}

In higher dimensions there exists an operation, called stabilization, that can change the Legendrian isotopy class of a Legendrian embedding. We do not need the precise construction, but only the following result of Murphy.
\begin{proposition}[\cite{Mur}]
    A Legendrian is stabilized if and only if it is loose.
\end{proposition}

As in dimension three, Legendrians in the complement of an overtwisted disk are loose. Thus, the h-principle holds.
\begin{theorem}[\cite{BEM}]\label{thm:h-principleLooseLegInComplementOfOTDisk}
    Let $(M,\xi)$ be an overtwisted contact manifold with overtwisted disk $\Delta$.  For any two Legendrian embeddings $f_0,f_1: L \rightarrow M \setminus \Delta$, any formal Legendrian isotopy from $f_0$ to $f_1$ is homotopic, relative endpoints, to a genuinely Legendrian isotopy.
\end{theorem}

\subsection{The energy of isotopies}

Given a compactly supported contact isotopy $(\phi^H_t)_{t \in [0,1]}$ with Hamiltonian $H_t$, we define its energy as 
\begin{equation*}
    \Vert \phi^H_t \Vert_\alpha \coloneqq \int_0^1 \max_M |H_t| dt.
\end{equation*}
This notion of energy induces a pseudo-norm $\Vert \cdot \Vert_\alpha$ and a pseudo-distance $d_\alpha$ on the identity component of the (compactly-supported) contactomorphism group by taking the infimum over the energies of all paths either from the identity to a given contactomorphism or from one contactomorphism to another. These are called the \emph{Shelukhin-Hofer norm} and the \emph{Shelukhin-Hofer distance}. By a theorem of Shelukhin \cite{shelukhin2017hofer}, $\Vert \cdot \Vert_\alpha$ and $d_\alpha$ are non-degenerate.

If the contact form $\alpha$ is replaced by $\chi \alpha$ for some positive function $\chi$, then the Hamiltonian $H_t$ (with respect to $\alpha$) of a contact isotopy $\psi_t$ is changed to $\chi H_t$. Consequently, we have that
\begin{equation}\label{eq:DistanceAndConformalFactor}
    \Vert \cdot \Vert_{\chi \alpha} \leq \left(\sup_M \chi \right)\Vert \cdot \Vert_\alpha, \quad d_{\chi \alpha} \leq \left( \sup_M \chi \right) d_\alpha. 
\end{equation}

\subsubsection{Legendrian isotopies}

We define the energy of a Legendrian isotopy $f_t$ similarly, as the infimum of the energies of all contact isotopies that induce $f_t$. This is equivalent to 
\begin{equation*}
    \Vert f_t \Vert_\alpha = \int_0^1 \max_L |\alpha(\partial_t f_t)| dt.
\end{equation*}
As before, the notion of energy induces a pseudo-metric $d_\alpha$ on each isotopy class of Legendrian embeddings, called the \emph{Shelukhin-Chekanov-Hofer distance}.

Observe that the distance $d_\alpha$ between two Legendrian embeddings that are isotopic via domain reparametrisations is zero; this can be shown directly or invoking Corollary \ref{cor:boudnednessOfStrictWeinsteinNbhdOfHeightC}. The question of non-degeneracy for $d_\alpha$ (as studied in Corollary \ref{cor:LegegendriansWithVanishingd_alpha}) becomes meaningful once one quotients out said reparametrisations (or, even further one mods out by all diffeomorphisms of the domain, not just isotopies, yielding thus a distance in the space of Legendrian submanifolds).

\subsubsection{Reparametrizations of contact isotopies}

The Shelukhin-Chekanov-Hofer metric was defined by minimizing the integral of the $L_\infty$-norm of the generating Hamiltonian. It turns out that is enough to minimize the $L_\infty$-norm of the Hamiltonians.

\begin{lemma}\label{lem:EquiparametrizedHamiltonian}
    For two contactomorphisms $\phi_0$ and $\phi_1$ isotopic to the identity, we have that
    \begin{equation*}
        d_\alpha(\phi_0,\phi_1) = \inf \left\{ \max_{(t,x) \in [0,1]\times M} |H_t|  \Big|\, \phi^H_1 \circ \phi_0 = \phi_1 \right\}.
    \end{equation*}
    An analogous formula holds for Legendrian embeddings.
\end{lemma}
\begin{proof}
    This follows immediately from \cite[Lemma 3.1]{Nak23} applied to $t \mapsto \max_{x \in M} |H_t|$.
\end{proof}

\subsubsection{The energy-capacity inequality}

On some classes of contact manifolds, including all closed contact manifolds and contactizations of geometrically bounded symplectic manifolds \cite{DRS20,DRS24}, and for orderable Legendrian isotopy classes \cite{AA25,Nak23} this distance is non-degenerate, but in general it is not \cite{Can23,Nak24a}, see Corollary \ref{cor:LegegendriansWithVanishingd_alpha}. In fact, the following special case of the main result of \cite{DRS24} can provide lower bounds on the distance of two Legendrians.
\begin{theorem}[\cite{DRS24}]\label{thm:energy-capacity inequality}
    Let $(M,\alpha)$ be a closed strict contact manifold and $L \subseteq M$ be a closed Legendrian submanifold. We let $\hbar$ be the infimum of the actions of all non-trivial Reeb chords of $L$ and all closed, contractible Reeb orbits. Then, for any Hamiltonian $H_t:M \to \R$, $t \in [0,1]$, such that 
    \begin{equation*}
        \Vert H_t \Vert_{osc} \coloneqq \int_0^1 \left(\max_{x \in M} H_t(x) - \min_{x \in M} H_t(x) \right) dt < \hbar,
    \end{equation*}
    there exists a Reeb chord between $L$ and each $\phi^H_t(L)$ of action not larger than $\Vert \phi_t^H\Vert_\alpha$.
\end{theorem}

\begin{proof}
    The only part that is not explicitly stated in \cite[Theorem 1.4]{DRS24} is the bound on the action. For this, recall the proof of \cite[Theorem 1.4]{DRS24}. One constructs barcodes associated to the filtered Rabinowitz-Floer complex with respect to certain action windows, and the existence of a Reeb chord follows from the existence of an infinite bar. By \cite[Lemma 5.1]{DRS24}, the starting value $l(t)$ of that infinite bar (and thus the action of the Reeb chord) satisfies $|\partial_t l(t)| \leq \max\limits_{x \in M} |H_t(x)|$, and $l(0)$ can be chosen to be arbitrarily close to zero. The construction required the use of a generic Hamiltonian but, by compactness of $L$, if we take a sequence of Hamiltonians that $C^\infty$-converge to $H_t$ and if we let $l(0)$ go to zero, then we will find a Reeb chord from $L$ to $\phi^H_t(L)$ with the required action bound.
\end{proof}

\subsubsection{Contact manifolds of bounded cost}

\begin{definition} \label{def:boundedCost}
We say that a strict contact manifold $(M^{2n+1},\alpha)$ is of \textbf{bounded cost $C$} if any compactly supported contact isotopy $\psi_t:M \to M$ with $t \in [0,1]$ can be homotoped with fixed endpoints to a contact isotopy of energy at most $C$.
\end{definition}

The following statement is proven in \cite{Nak24}, but we briefly recall the proof:
\begin{lemma} \label{lem:boundedCost}
Consider a Liouville domain $(W,\lambda)$ and a finite interval $(-C,C)$. Then, the (piece of) contactization $(W \times (-C,C), ds + \lambda)$ has bounded cost equal to $2C$.
\end{lemma}
\begin{proof}
The Liouville vector field $X$ lifts to a contact vector field $s\partial_s + X$ in the contactisation $W \times \R$. Its Hamiltonian function is $s$. Its flow $\psi_t$ is defined for $t \in (-\infty,0]$ and satisfies $\psi_t(W \times (-C,C)) \subset W \times e^t(-C,C)$. It is not compactly-supported.

At time $t$, we can cut-off the time-independent Hamiltonian $s$ outside of $W \times e^t[-C,C]$, producing a time-dependent Hamiltonian $H_t$ in $W \times \R$ that is compactly-supported. It still compresses $W \times (-C,C)$ into $W \times e^t(-C,C)$. The energy of this Hamiltonian can be assumed to be arbitrarily close to $\int_{-\infty}^0 e^tC dt = C$.

According to our conventions (Subsection \ref{sssec:conventions}), a compactly-supported isotopy of $W \times (-C,C)$ is supported in the interior. As such, it extends to the full contactization as the identity. We can thus conjugate it by the aforementioned compression to produce another compactly-supported isotopy, in the same isotopy class with fixed endpoints, and energy arbitrarily close to $2C$.

Since the support of $\psi_t$ is, in fact, contained in $W\times  (-C',C')$ for some $C' < C$, the lemma follows by applying the above argument to $C'$.
\end{proof}
Concrete examples are bounded tubular neighbourhoods of the zero section in $1$-jet spaces $J^1(L,\R)$ endowed with the standard contact form. Darboux balls (and thus loose charts) are special cases. In particular, Lemma \ref{lem:boundedCost} applies to strict Weinstein neighbourhoods of Legendrian embeddings of height $C$:

\begin{corollary}\label{cor:boudnednessOfStrictWeinsteinNbhdOfHeightC}
    A strict Weinstein neighbourhood of a proper Legendrian embedding of height $C$ has bounded cost equal to $C$.
\end{corollary}

\subsection{Quantitative flexibility of loose Legendrians}

In Section \ref{sec:introduction}, we defined the stabilization cost of a connected, higher-dimensional, loose Legendrian as the energy that is needed to produce a loose chart inside of a strict Weinstein neighbourhood of arbitrarily small height. If the Legendrian embedding $f:L \to M$ is disconnected, we instead require that we can produce a loose chart inside of a Weinstein neighbourhood of arbitrarily small height for all components of $L$ at the same time, see Remark \ref{rmk:stabilizationCostDisconnected}.

We now relate the stabilization cost to the notion of \emph{loose size} introduced in \cite{Nak24}. It is defined as follows:
\[ \mathcal{S}(f;M) := 2\limsup\limits_{\eps \to 0} \inf\limits_{g} d_\alpha(f,g), \]
where $(M,\alpha)$ is a strict contact manifold of dimension at least $5$, $f$ is a proper Legendrian embedding into $M$, and $g$ ranges over all Legendrians that have a loose chart contained in a strict Darboux ball of radius $\eps$.

\begin{lemma} \label{lem:Nakamura}
$\mathcal{S}(f;M) = 2\mathcal{C}(f;M)$. 
\end{lemma}
\begin{proof}
Suppose we have an $\eps$-stabilization $g$. Then, the loose size of $g$ is bounded from above by $2\eps$, according to Corollary \ref{cor:boudnednessOfStrictWeinsteinNbhdOfHeightC}. It follows that, if $f$ is isotopic with cost $c$ to $g$, then the cost of isotoping $f$ to some $g'$ with a small loose chart is bounded above by $c+2\eps$. As $\eps$ approaches zero we can let $c$ approach the stabilization cost of $f$, showing that $c$ is an upper bound for half the loose size.

Conversely, assume that $g$ has loose size $\eps$. Let $g'$ be an $\eps'$-stabilization of $h$ in the same Legendrian isotopy class as $g$, then $d_\alpha(g,g') \leq \eps/2+\eps'$, by Theorem \ref{thm:LooseLegQuantitativeFlexibility} below. Since $\eps'$ can be assumed arbitrarily small, the stabilization cost of $g$ is bounded above by $\eps/2$. If now $f$ is isotopic with cost $c$ to such a $g$, we have that the stabilization cost of $f$ is at most $c+\eps/2$. As we let $\eps \to 0$, $2c$ can be assumed to approach loose size of $f$ which is thus double of the stabilization cost.
\end{proof}

The following result was established in \cite{Nak24} for the special case where $L$ is closed and connected. Using the previous lemma one could rephrase it using the stabilization cost.
\begin{theorem}\label{thm:LooseLegQuantitativeFlexibility}
    Let $(M,\alpha)$ be a strict contact manifold of dimension $\geq5$. Let  $f_0,f_1:L \to M$ be proper Legendrian embeddings that are formally Legendrian isotopic and have loose sizes $C_0$ and $C_1$, respectively. We have that:
    \begin{equation*}
        d_\alpha(f_0,f_1) \leq C_0/2 + C_1/2.
    \end{equation*}
\end{theorem}
\begin{proof}
     We explain how to generalize the arguments in \cite{Nak24} to the present case. The two ingredients in the proof are Murphy's h-principle for loose Legendrians \cite{Mur} (see Theorem \ref{thm:LooseLegh-principle} above) and the vanishing of $d_\alpha$ for non-Legendrians \cite{RoZhChekanovsDichotomy,DRSC0Limits}. The latter holds for properly embedded and possibly disconnected non-Legendrians; see \cite[Theorem 2.4, Remark 2.5]{Nak24}. Note further that \cite[Corollary 2.6]{Nak24} also applies in the case that $f$ is Legendrian outside of a compact set and non-Legendrian almost everywhere in a neighbourhood of the support of $\Phi_t$, i.e.\,given such an embedding $f:L \to M$ and a contact isotopy $\Phi_t:M \to M$, we can find for any $\delta  > 0$ another contact isotopy $\phi_t$ of arbitrarily small energy such that $\Phi_1 \circ f = \phi_1 \circ f$ and $\phi_t \circ f$ is $\delta$-close to $\Phi_t \circ f$ in the $C^0$-norm.

    The proof of the main result of \cite{Nak24} follows the following steps:
    \begin{itemize}
        \item Perturb $f_0$ to a formal Legendrian isotopy $(g,G_s)$ so that $g$ is non-Legendrian almost everywhere.

        \item Pick a compactly supported small energy contact isotopy $\phi_t$ so that $\phi_1 \circ g = \Phi_1 \circ g$ and $\phi_t \circ g$ is $C^0$-close to $\Phi_t\circ g$.

        \item $C^0$-approximate $g$ by a loose Legendrian $\chi$.

        \item Apply Murphy's h-principle inside of sets of bounded cost to find Legendrian isotopies of bounded energy from $f_0$ to $\chi$ and from $\phi_1 \circ \chi$ to $f_1$.
    \end{itemize}

    We now modify the first step so that $(g,G_s)$ agrees with $(f_0,Df_0)$ outside of a compact subset and the intersection of $g(L)$ with a neighbourhood of the support of $\Phi_t$ is non-Legendrian almost everywhere. After this modification, the rest of the argument works as before. 
\end{proof}


\section{(Formal) Legendrians in standard position}\label{sec:bringing Legendrian in standard position}

In this section we discuss standard models for Legendrian arcs in the vicinity of the standard overtwisted disk (Subsection \ref{ssec:standardPosition}). We then explain how a (formal) Legendrian can be isotoped at one point to create a Legendrian arc in such a standard position (Subsections \ref{ssec:standardisingPoint} and \ref{ssec:standardisingOTDisc}). This is then used to prove that it is sufficient to consider a single stabilization when computing (an upper bound for) the overtwisted size (Subsection \ref{ssec:oneStabilization}). All of this boils down to drawing Legendrian isotopies, which we do using a cylindrical front projection; we introduce it in Subsection \ref{ssec:cylindricalProjection}.

\subsection{The cylindrical projection in overtwisted $\R^3$} \label{ssec:cylindricalProjection}

Consider $\R^3$ endowed with the standard overtwisted contact form $\alpha_\OT^\std = \cos(\rho) dz + \rho \sin(\rho) d \phi$, expressed in cylindrical coordinates.

\begin{definition}
The map $\R^3 \setminus \{\rho=0\} \rightarrow \NS^1 \times \R$ given by $(\rho,\phi,z) \mapsto (\phi,z)$ is called the \textbf{cylindrical projection}.
\end{definition}

Observe that the cylindrical projection is a Legendrian fibration (just like the usual front projection). Standard transversality arguments show that:
\begin{lemma}
For a Legendrian $L \subset (\R^3,\alpha_\OT^\std)$, the following property is open (in the $C^2$-topology) and dense (in the $C^\infty$-topology): $L$ is disjoint from the $z$-axis and its cylindrical projection is a curve whose only singularities are cusps and transverse self-intersections of multiplicity $2$.
\end{lemma}

\begin{definition}
    The \textbf{front} of a Legendrian $L \subseteq (\R^3,\alpha_\OT^\std)$ is its image under the cylindrical projection.
\end{definition}

\begin{remark}
    In contrast to the usual front projection of Legendrians in standard contact $\R^3$, vertical tangencies may appear.
\end{remark}

If $L$ is a generic Legendrian with cylindrical projection $\pi(L)$, we readily see that the slope of $\pi(L)$ can be computed from the missing radial coordinate $\rho$. Conversely, one can recover $\rho$ from the slope of $\pi(L)$, but not uniquely. Concretely, zero slope corresponds to the values $\rho = k\pi$ with $k\geq 1$. All other slopes correspond similarly to countably many choices of $\rho$, which are roughly spaced an amount $\pi$. To remove this ambiguity in pictures, we will indicate the regions $\{(k-1)\pi \leq \rho \leq k \pi\}$ if there is a risk of confusion.

Similarly, we can use parametric transversality to yield a list of Legendrian Reidemeister moves. The list contains the three standard moves, plus a fourth move that describes how a generic isotopy crosses the $z$-axis, and a fifth move that describes embedded Legendrians whose fronts have self-intersections with tangencies (so the underlying $\rho$ coordinates are necessarily distinct).

\begin{lemma} \label{lem:Reidemeister}
Let $(L_s)_{s \in [0,1]} \subset (\R^3,\alpha_\OT^\std)$ be a generic, compactly-supported Legendrian isotopy. Then, the cylindrical projections of the $L_s$, up to ambient isotopy in the cylinder $\NS^1 \times \R$, differ by a finite sequence of moves. These moves are of five types and are depicted in Figure \ref{fig:Reidemeister}.
\end{lemma}

\begin{figure}
\centering
\begin{subfigure}{0.51\textwidth}
    \includegraphics[width=0.4\textwidth]{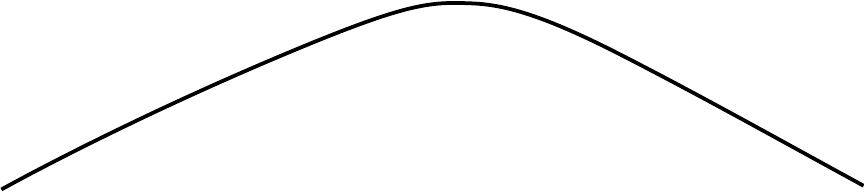}
    \hfill
    \raisebox{15pt}{$\leftrightarrow$}
    \hfill
    \includegraphics[width=0.4\textwidth]{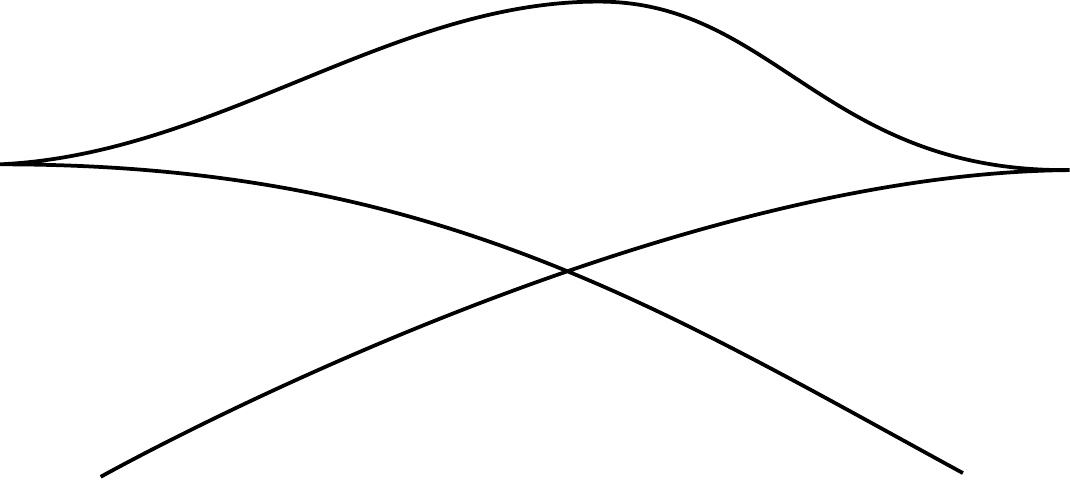}
    \caption{RI.}
    \label{fig:R1}
\end{subfigure}
\hfill
\begin{subfigure}{0.51\textwidth}
    \includegraphics[width=0.4\textwidth]{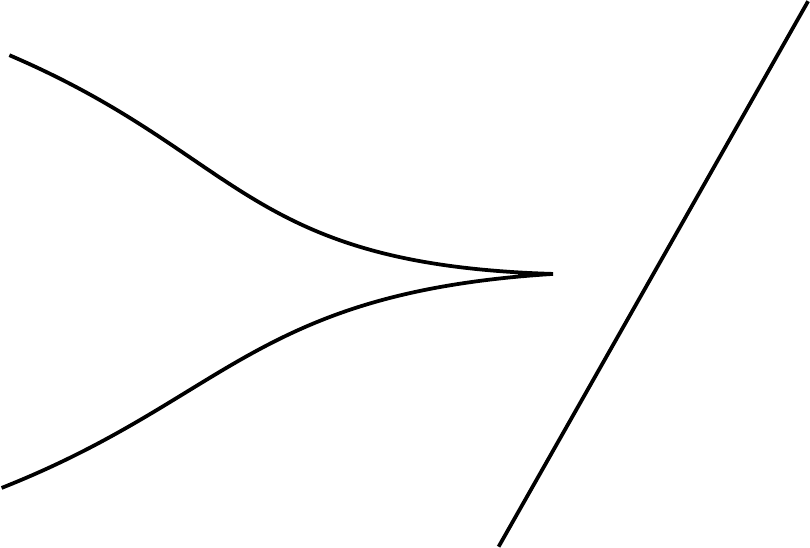}
    \hfill
    \raisebox{25pt}{$\leftrightarrow$}
    \hfill
    \includegraphics[width=0.3\textwidth]{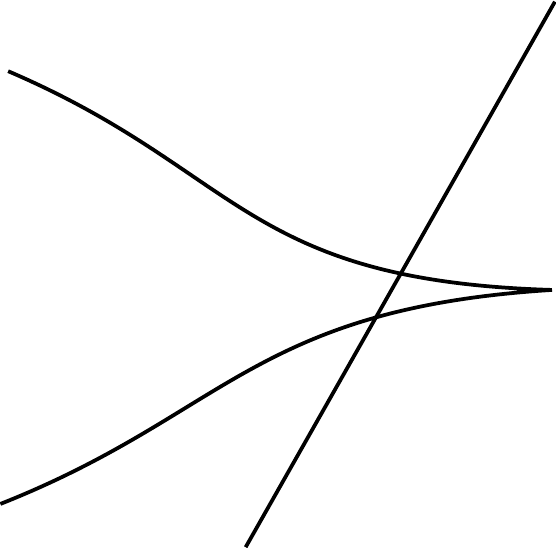}
    \caption{RII.}
    \label{fig:R2}
\end{subfigure}
\hfill\begin{subfigure}{0.51\textwidth}
    \includegraphics[width=0.4\textwidth]{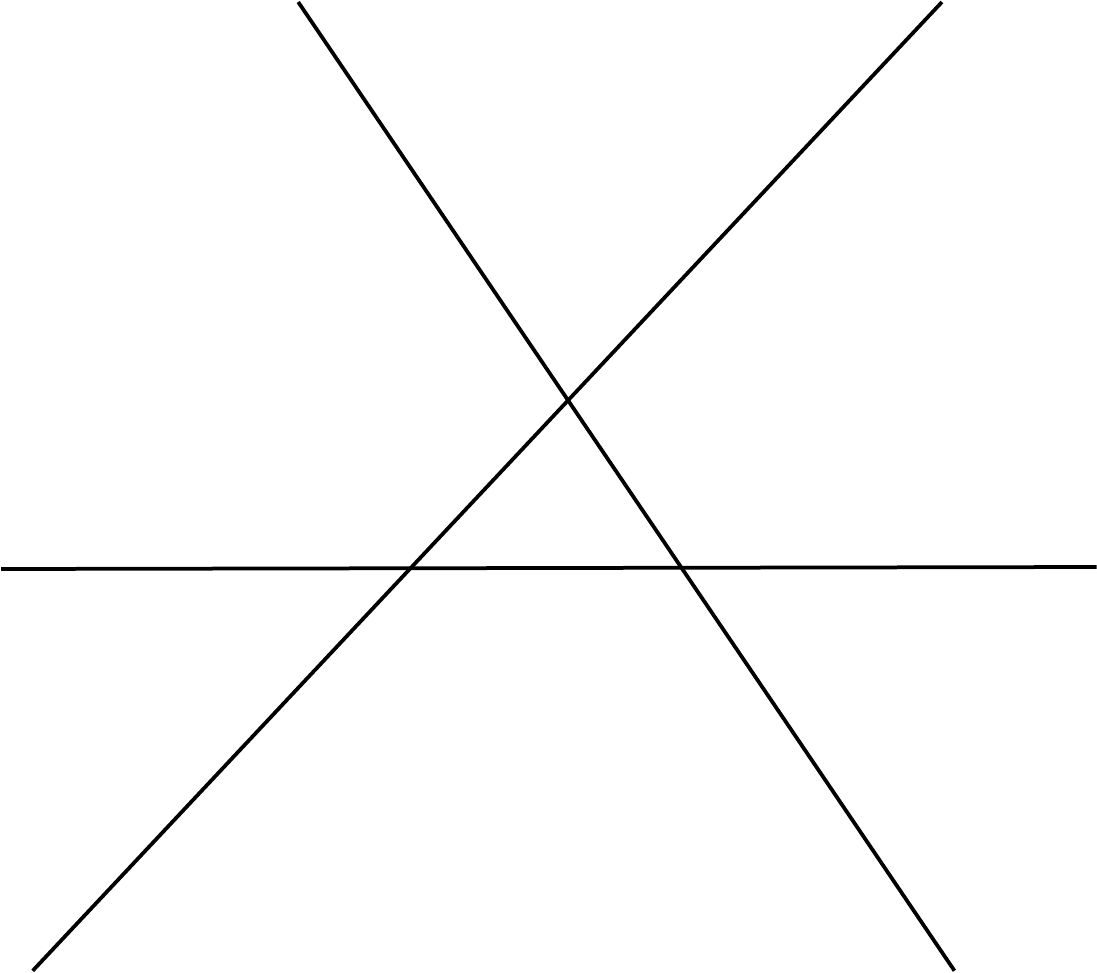}
    \hfill
    \raisebox{35pt}{$\leftrightarrow$}
    \hfill
    \includegraphics[width=0.4\textwidth]{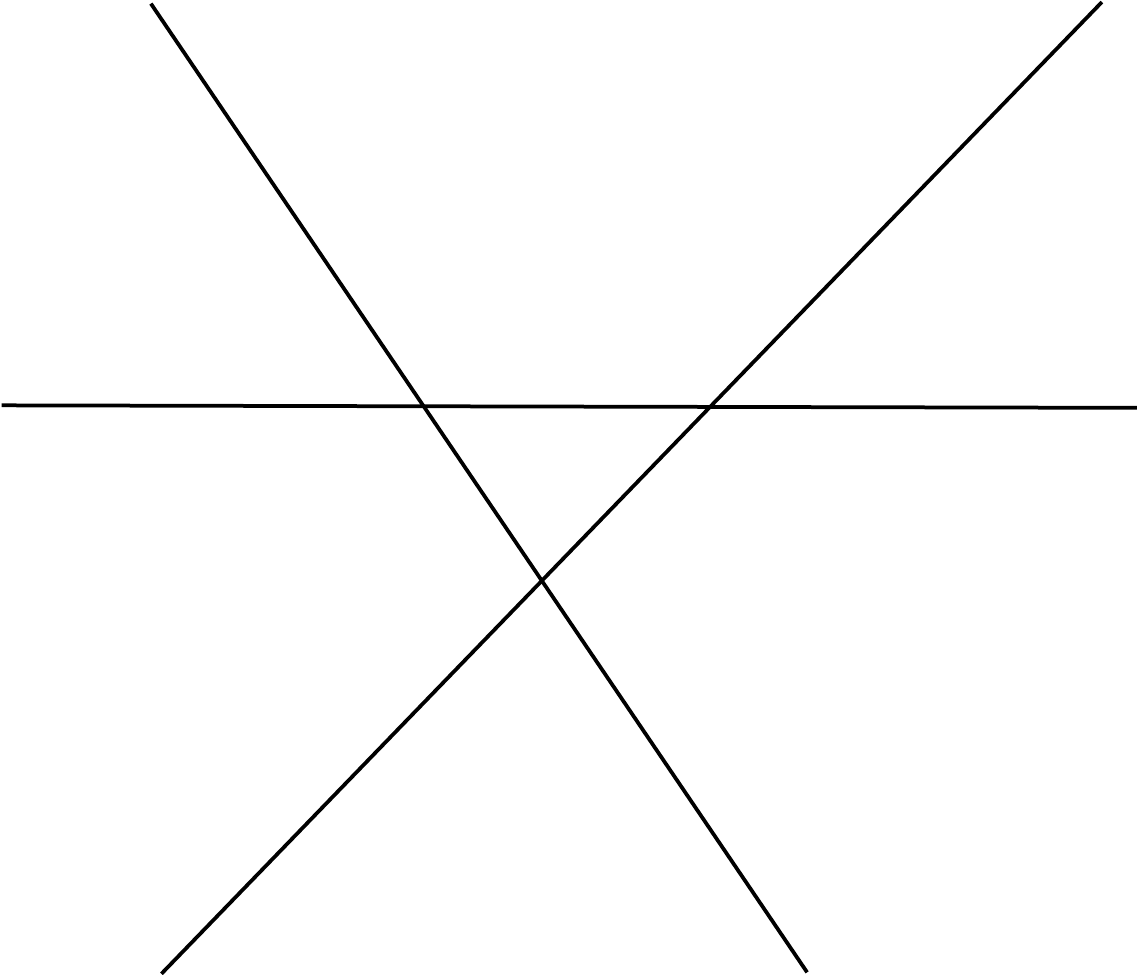}
    \caption{RIII.}
    \label{fig:R3}
\end{subfigure}
\hfill\begin{subfigure}{0.51\textwidth}
    \includegraphics[width=0.4\textwidth]{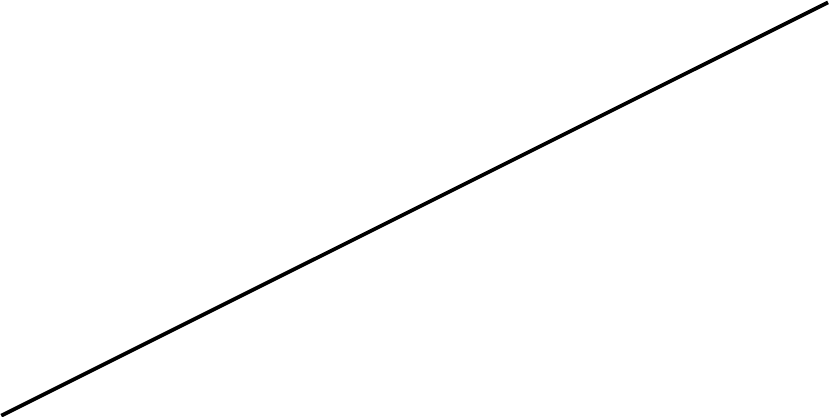}
    \hfill
    \raisebox{15pt}{$\leftrightarrow$}
    \hfill
    \includegraphics[width=0.45\textwidth]{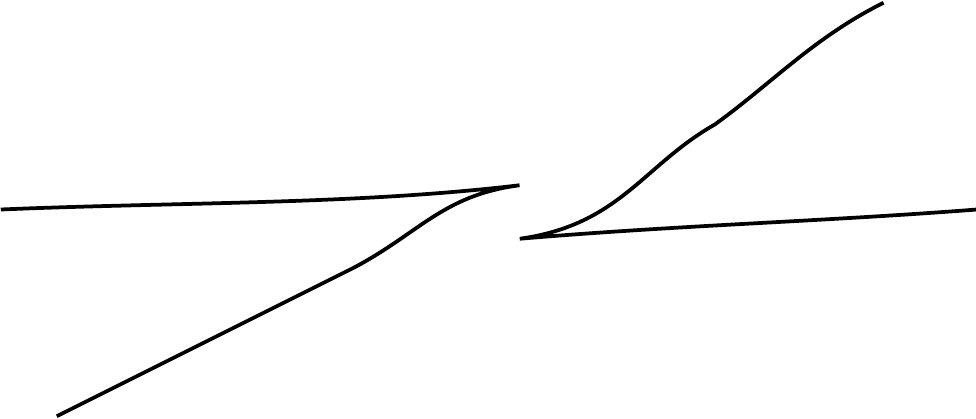}
    \caption{RIV. Only applies if the Legendrian lifts are contained in the region $\{0< \rho < \pi\}$.}
    \label{fig:R4}
\end{subfigure}
\hfill\begin{subfigure}{0.51\textwidth}
    \includegraphics[width=0.4\textwidth]{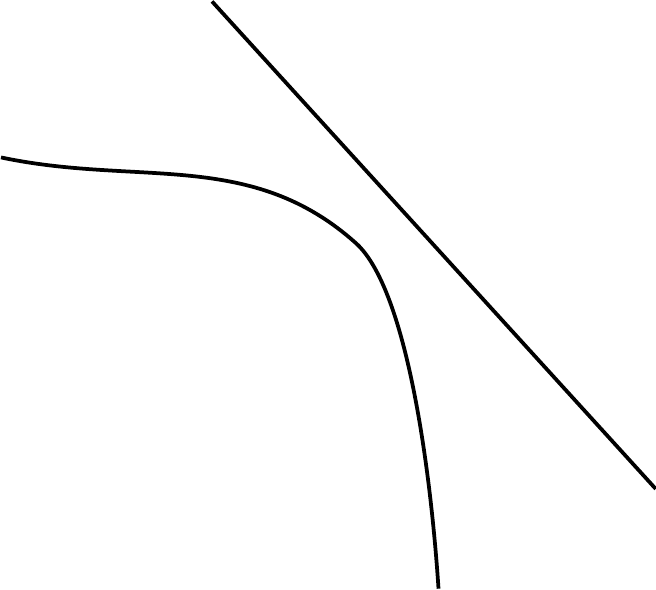}
    \hfill
    \raisebox{35pt}{$\leftrightarrow$}
    \hfill
    \includegraphics[width=0.35\textwidth]{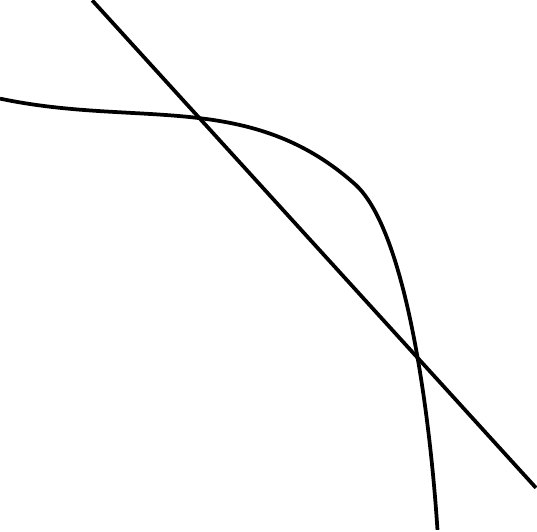}
    \caption{RV. Only applies if the Legendrian lifts of the two Legendrian arcs are contained in different regions $\{(k-1)\pi < \rho < k\pi\}$.}
    \label{fig:R5}
\end{subfigure}

\caption{The Reidemeister moves for Legendrian fronts under the cylindrical projection. The first three moves are the standard Reidemeister moves for Legendrian fronts and are local in $z$ and $\theta$. The fourth move lifts to a Legendrian that passes through the $z$-axis. It is no longer local in $\theta$, since the almost horizontal part of the front goes once around the $S^1$-factor. The fifth move describes a Legendrian whose front has a self-tangency at a point $p$ such that the two lifts of $p$ differ.}
\label{fig:Reidemeister}
\end{figure}

\subsection{Loose Legendrians in standard position} \label{ssec:standardPosition}

Let us begin in dimension $3$. Consider, for some $\rho_0 > \pi$ and $z_0 > 0$, the neighbourhood $U = \{\rho \leq \rho_0, |z|\leq z_0\} \subseteq \R^3$ of the standard overtwisted disk $\Delta^\textnormal{std}_\textnormal{ot}\subseteq \R^3$. We then let $f_\std:[0,1] \to U$
be a Legendrian arc whose front is shown in Figure \ref{fig:f_st}. Concretely: its boundary lies on $\partial U$, it intersects $\Delta_{\textnormal{ot}}^\textnormal{std}$ solely in the point $(x,y,z) = (\pi,0,0)$ where it is tangent to $\partial \Delta_{\textnormal{ot}}^\textnormal{std}$, and everywhere else satisfies $\pi < \rho \circ f_\std < 2\rho_0$.

For each larger dimension $2n+1$, $n \geq 1$, we define 
\[ U := U \times \left([-(1+\eps),1+\eps] \times [-C_0/2,C_0/2] \right)^{n-1} \]
recursively; we apologise for the overload in notation. Similarly, we define the Legendrian embedding 
\begin{equation*}
\begin{split}
f_\std: [0,1] \times [-(1+\eps),1+\eps]^{n-1} \quad& \longrightarrow \quad U \\
(s,x_1,...,x_{n-1}) \quad & \mapsto \quad (f_\std(s),x_1,0,x_2,0,...,x_{n-1},0)
\end{split}
\end{equation*}
for some $C_0 > C^*$ and $\eps > 0$.

\begin{figure}
\centering
\begin{subfigure}{0.30\textwidth}
    \includegraphics[width=0.7\textwidth]{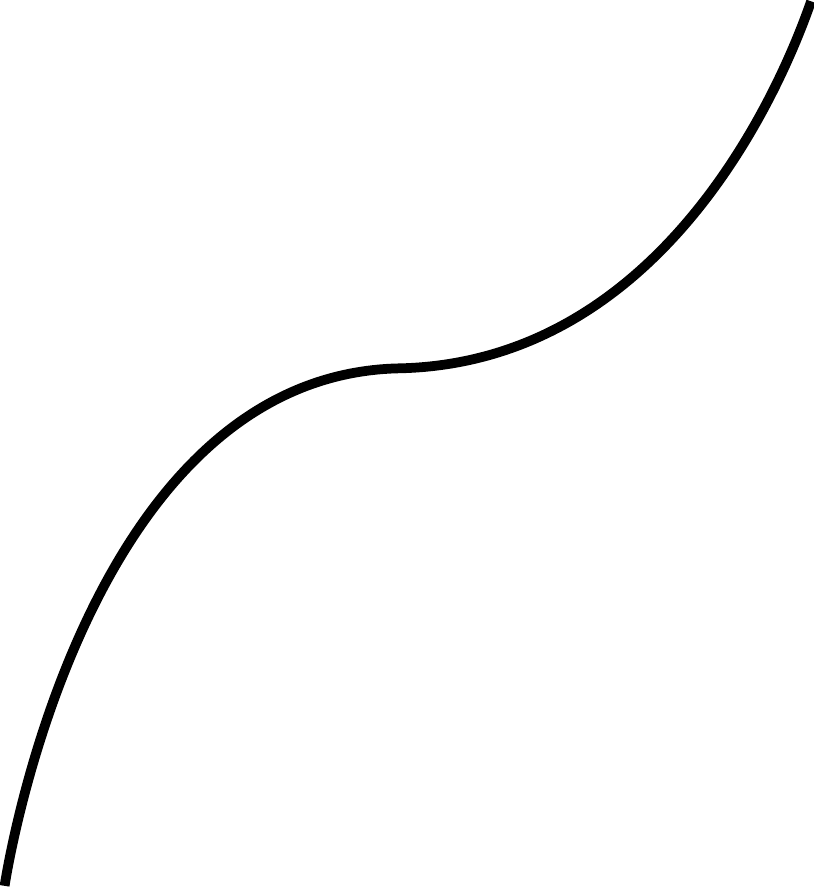}
    \caption{The front of $f_\std$.}
    \label{fig:f_st}
\end{subfigure}
\hfill
\begin{subfigure}{0.30\textwidth}
    \includegraphics[width=0.7\textwidth]{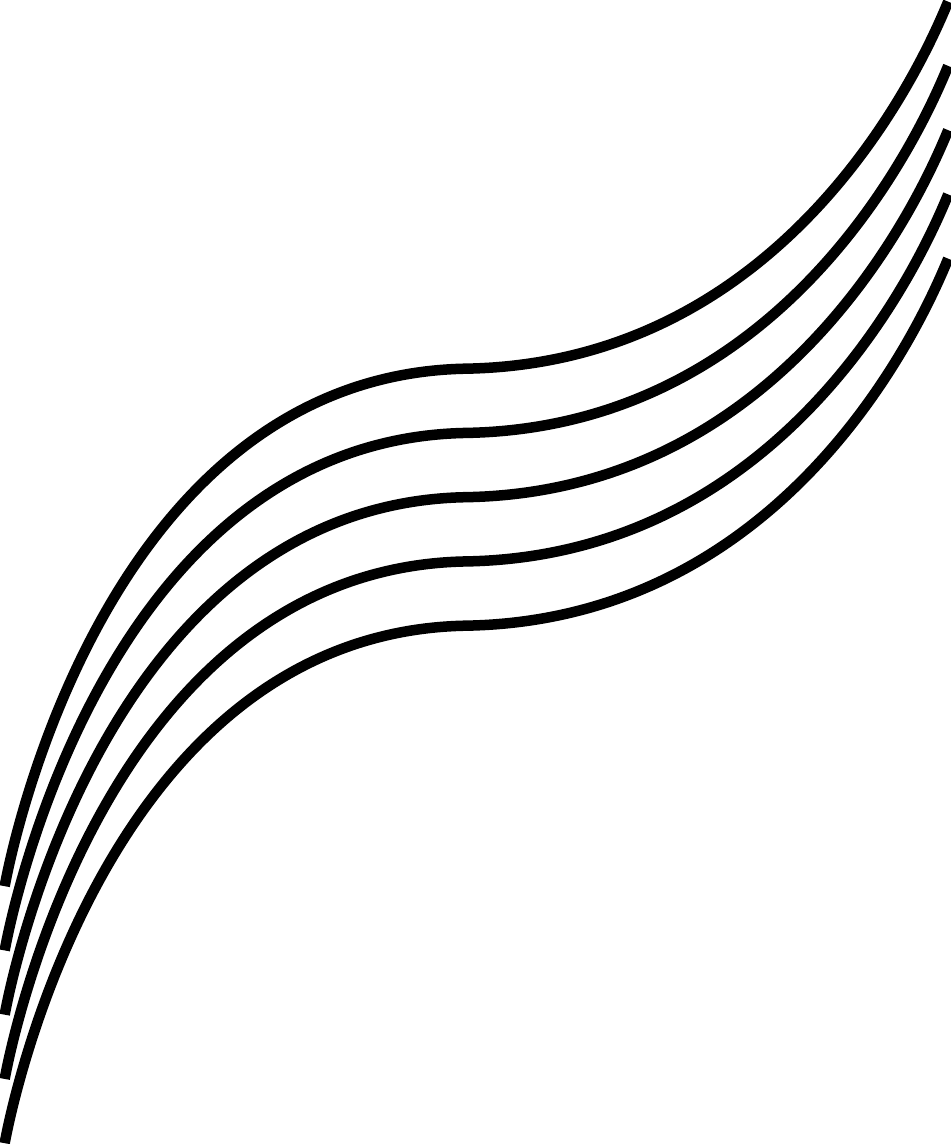}
    \caption{The front of $f^{(k)}_\std$.}
    \label{fig:f_st^k}
\end{subfigure}
\hfill
\begin{subfigure}{0.30\textwidth}
    \includegraphics[width=\textwidth]{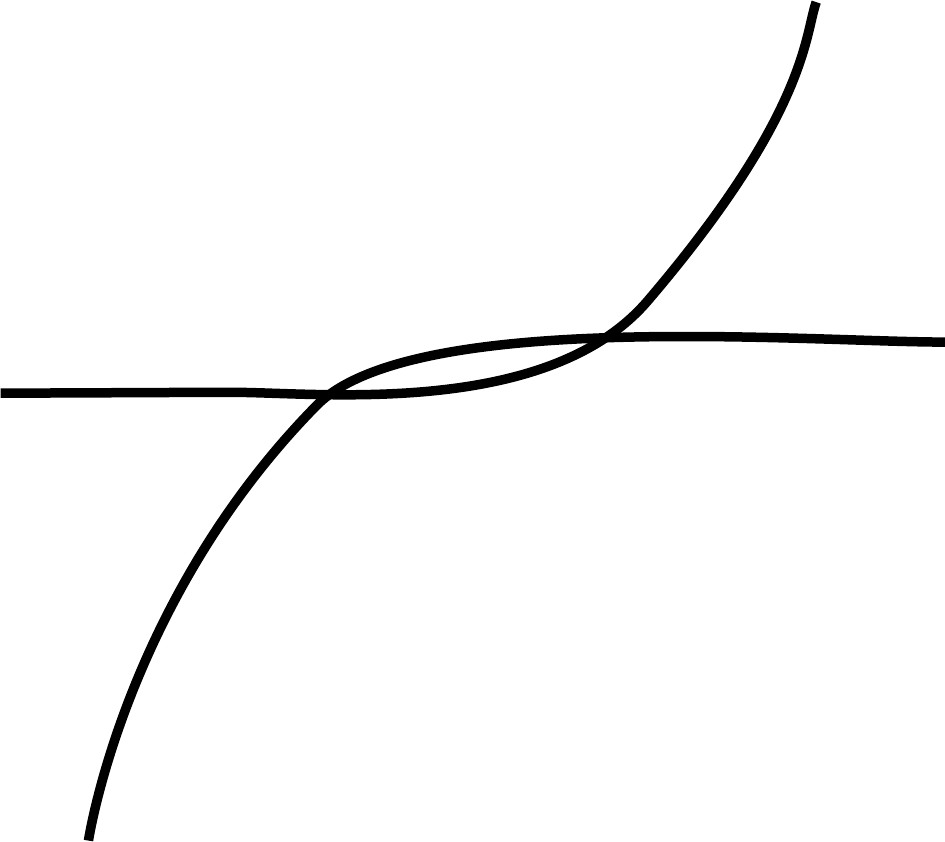}
    \caption{The front of $f_\std'^{(1)}$.}
    \label{fig:f'_st}
\end{subfigure}

\caption{The Legendrian standard models near $\Delta_\OT^\textnormal{std}$. The image of $f_\std$ is contained in the region $\{\pi \leq \rho < 2 \pi \}$ and intersects the overtwisted disk in a single point at which the front is horizontal. The Legendrian $f^{(k)}_\std$ consists of $k$ parallel copies of $f_\std$. The Legendrian $f_\std'^{(1)}$ is the connected sum of $f_\std$ and the boundary of the overtwisted disk; in Subfigure (c), the left side is identified with the right side. The Legendrian $f_\std'^{(k)}$ consists of $k$ parallel copies of $f_\std'^{(1)}$.}
\label{fig:Legendrian standard model}
\end{figure}

In our argument, the pair $(U,f_\std)$ will play the role of a standard model for a Legendrian near an overtwisted disk. Up to legendrian isotopy (see the upcoming subsections) this can be achieved for any Legendrian in the complement of an overtwisted disk. We allow $\rho_0 > \pi$, $z_0 > 0$, $C_0 > C^*$, and $\eps >0$ to be replaced by a smaller number if needed.

In dimension $3$, we will also need $k \in \N$ parallel copies of this standard model as follows:
\begin{equation}\label{eq:U^(k)}
    U^{(k)} \coloneqq \left\{\left(\rho,\phi,z + \frac{iz_0}{k}\right)|\, \rho \leq \rho_0, |z| \leq \frac{z_0}{3k}, i \in \{0,1,..., k-1\}\right\} \subseteq U\,,
\end{equation}
which contains the Legendrian embedding $f_\std^{(k)}$ with boundary on $\partial U^{(k)}$, which is a union of parallel copies of $f_\std$, one in each of the $k$ components of $U^{(k)}$, see Figure \ref{fig:f_st^k}.

In this case, denote by $f_\std'^{(k)} \subseteq U^{(k)}$ the Legendrian obtained by gluing to each of the $k$ components of $f_\std$ the boundary of the parallel copy of the overtwisted disk inside of the corresponding component of $U^{(k)}$, see Figure \ref{fig:f'_st} for the case $k=1$. It is well-known that $f_\std^{(k)}$ is Legendrian isotopic to a $k$-fold stabilization (with one stabilization on each connected component) of $f_\std'^{(k)}$ \cite{MNPS}, see Section \ref{sec:QuantitativeIsotopiesInDimension3} for an explicit description of such an isotopy.

\subsection{Standardising at one point} \label{ssec:standardisingPoint}

The following notation will be convenient: we say that $J^1(L,[0,\infty))$ is the upper hemisphere of $J^1(L,\R)$. We similarly call $J^1(L,(-\infty,0])$ the lower hemisphere. We can see a Darboux ball $B$ as a neighbourhood of the zero section in $J^1(\D^n,\R)$ and thus speak of its hemispheres $B^+$ and $B^-$ accordingly.

\begin{lemma} \label{lem:FormalLegendrianStandardPosition}
There is a function $C(r)>0$ going to zero as $r$ goes to zero, such that:

Fix a strict contact manifold $(M^{2n+1},\alpha)$, an open strict Darboux ball $(B,\alpha_\std) \subset (M,\alpha)$ of radius $r$, and a family $(f_t,F_{s,t}): L \to M \setminus B, t \in [0,1],$ of formal Legendrian embeddings which is Legendrian for $t \in \{0,1\}$. Assume that $M \setminus B^-$ is path-connected.

Then, there exist families of Legendrian embeddings $a_t,b_t:L \to M$ and a family of formal Legendrian embeddings $(g_t,G_{s,t}):L \to M$ with $t \in [0,1],$ such that
\begin{enumerate}
    \item  $a_0 = f_0$, $a_1=g_0$, $b_0 = g_1$, and $b_1 = f_1$.
    \item the family $(f_t,F_{s,t})$ is homotopic, relative endpoints and through formal Legendrian embeddings, to the concatenation $a_t * (g_t,G_{s,t}) * b_t$,
    \item $|\alpha(\partial_t a_t)|, |\alpha(\partial_t b_t)| < C(r)$,
    \item $g_t^{-1}(B)$ and $(g_t,G_{s,t})|_{g_t^{-1}(B)}$ do not depend on $t$, and $g_t(L)$ intersects $B$ only in its upper hemisphere $B^+$ and contains the disk $\D^n_{r/2} \subset B$.
\end{enumerate}
In particular, $(g_t,G_{s,t})|_{g_t^{-1}(B)}$ is Legendrian.
\end{lemma}
\begin{proof}
A first homotopy, relative to endpoints in the parameter, allows us to assume that $(f_t,F_{s,t})$ is Legendrian near a point $p \in L$, for all $t$. Then, by the existence h-principle for isotropic paths, we can find a family of embedded isotropic paths $\gamma_t: [0,1] \to M$ with $\gamma_t(0) = f_t(p)$ and $\gamma_t(1)$ in the interior of $B^+$ such that $\gamma_t^{-1}(B)$ and $\gamma_t|_{\gamma_t^{-1}(B)}$ do not depend on $t$. The path $\gamma_t$ can be assumed to be disjoint from $B^-$, due to the path-connectedness of $M \setminus B^-$, and disjoint from $f_t(L)$ away from $0$. 

We think of each path $\gamma_t$ as a homotopy of the point $\gamma_t(0) = f_t(p)$. Then, the contact isotopy extension theorem produces a $t$-dependent family of contact isotopies that pushes $\gamma_t(0)$ to $\gamma_t(1)$ and is supported in a neighbourhood of $\gamma_t$. However, we can be more explicit and provide an isotopy that has arbitrarily small energy. This is where the assumption that $\gamma_t$ is isotropic will be used.

First, we invoke Lemma \ref{lem:Serre} and the homotopy lifting property. This implies that the family of points $(\gamma_t(u))_{t,u \in [0,1]}$ can be thickened to a family of Legendrian embeddings $w_{t,u}: U \rightarrow (M,\alpha)$ such that $w_{t,0} = f_t|_U$, where $U \cong \D^n \subset L$ is a sufficiently small neighbourhood of $p$.

Moreover, we can assume that $w_{t,u}$ does not depend on $t$ in the region $\gamma_t^{-1}(B)$. Indeed, homotopy lift $\gamma_0$ first to yield the family $w_{0,u}$. Since the paths $\gamma_t$ do not depend on $t$ over $\gamma_t^{-1}(B)$, we can thicken $\gamma_t|_{\gamma_t^{-1}(B)}$ to $w_{0,u}$, for all $t$. Lastly, apply Lemma \ref{lem:Serre} and homotopy lifting to $(\gamma_t(u))_{t,u \in [0,1]}$ but now relative to $\{u=0\} \cup \{t=0\} \cup \gamma_t^{-1}(B)$. We apply Lemma \ref{lem:smoothingFamilies} to smooth out the families obtained in this manner.

Apply contact isotopy extension to obtain a family of contact isotopies $\psi_{t,u}$ generating $w_{t,u}$. Note that $u$ is the time parameter of the isotopy and $t$ serves as an additional parameter. We can cut off the corresponding Hamiltonian so their support is an arbitrarily small neighbourhood of $\gamma_t(u)$. In doing so, their value becomes arbitrarily close to $\alpha(\gamma_t'(u)) = 0$, since we chose the paths to be isotropic. It follows that these isotopies have arbitrarily small cost.

We obtain thus a family $f_{t,u} := \psi_{t,u} \circ f_t$ of formal Legendrian embeddings, all of them Legendrian over $U$. For $u$ close to $1$, they meet $B$ over some small ball $U'$ contained in $U$, where they do not depend on $t$.

We conclude using Lemma \ref{lem:Serre2}. We apply it to $f_{t,1}|_{U'}$ in order to further contact isotope, within $B$, to a family $f_{t,2}$ satisfying Property (4). This is a consequence of the fact that the space of Lagrangian frames in $B$ is connected. A priori, this isotopy has unknown energy, but we can now use the fact that $B$ is of bounded cost (Lemma \ref{lem:boundedCost}) to replace it by an isotopy whose cost is universally bounded (and in fact going to zero with the radius $r$). This concludes the proof.
\end{proof}

\begin{remark}\label{rmk:StdPositionDisconnected}
    As is clear from its proof, Lemma \ref{lem:FormalLegendrianStandardPosition} admits the following refinement. Assume that $L$ has finitely many connected components $L_\iota$ with $\iota \in I$, and we have chosen a corresponding collection of mutually disjoint Darboux balls $(B_\iota, \alpha_{\textnormal{std}}) \subseteq (M,\alpha)$ disjoint from the image of $f_t$. Then we may assume that Property (4) holds for every $\iota$ with $L$ and $B$ replaced by $L_\iota$ and $B_\iota$ with the additional condition that $g_t(L_{\iota'})$ does not intersect $B_\iota$ for $\iota \neq \iota'$.
\end{remark}

\subsection{Standardizing near an overtwisted disk}  \label{ssec:standardisingOTDisc}

If we apply Lemma \ref{lem:FormalLegendrianStandardPosition} just to a Legendrian, we see that any Legendrian embedding can be brought into a standard position in $B$ by an isotopy of arbitrarily small energy. This observation can be used to also bring Legendrians in different types of standard positions. In particular, we have the following:

\begin{lemma}\label{lem:IntoStdPosNearOTDisk}
    Let $(M^{2n+1},\alpha)$ be a connected, overtwisted, strict  contact manifold. Let $\psi: V \subseteq M \to U \subseteq \R^3 \times \R^{2n-2}$ be a contactomorphism from a compact neighbourhood of an overtwisted disk in $M$ to the standard model. Let $f:L \to M \setminus V$ be a proper Legendrian embedding. Let $\eps > 0$ and, if $n=1$, let $k \in \N$. Then there exists a proper Legendrian isotopy $g_t$ such that
    \begin{itemize}
        \item[] 1) $g_0 = f$,
        \item[] 2) $\Vert g_t \Vert_\alpha < \eps$,
        \item[] 3) If $n > 1$, $g_1|_{g_1^{-1}(V)}$ agrees with a reparametrization of the standard model $f_\std$ under the identification given by $\psi$.
        \item[] 4) If $n=1$, $g_1|_{g_1^{-1}(\psi^{-1}(U^{(k)}))}$ agrees with a reparametrization of $f^{(k)}_\std$ under the identification given by $\psi$. 
    \end{itemize}
\end{lemma}

\begin{proof}
    On $V$, we use the coordinates induced by $\psi$. Let $D$ be a Darboux ball in $M$ about the point $(x,y,z) = (\pi,0,0)$. We may assume that $D$ does not intersect the image of $f$, has bounded cost equal to $\eps$, and contains the image of $f_\std$. After possibly shrinking $V$, we may assume that $D \setminus V \neq \emptyset$. Now pick a tiny Darboux ball $B \subseteq D \setminus V$, and connect it to $M \setminus (V \cup D)$ via a thin tube $T$ around an arc in $D \setminus V$. By Lemma \ref{lem:FormalLegendrianStandardPosition} applied to $B \subseteq (M \setminus (D \cup V)) \cup B \cup T$, there exists a Legendrian isotopy $g_t$ satisfying 1) and 2) such that it intersects $B^-$ in the disk $\D_{r/2}^n$, where $r$ denotes the radius of $B$. Note that $B$ can be chosen so that the intersection of the lower and upper hemisphere of $B$ is almost horizontal in the Darboux chart for $D$ with the upper hemisphere being above the lower one. Then we can use Reidemeister moves in the front projection plus an isotopy of the front with compact support in $D$ to isotope the front of $g_t$ until it intersects $V$ as the desired standard model for the Legendrians (see Figure \ref{fig:standardizing near ot disk}). As $D$ has bounded cost equal to $\eps$, the energy of this isotopy can be assumed to be smaller than $\eps$. As $\eps>0$ was arbitrary, this finishes the proof.
\end{proof}

\begin{figure}
\centering
\begin{subfigure}{0.450\textwidth}
    \includegraphics[width=1\textwidth]{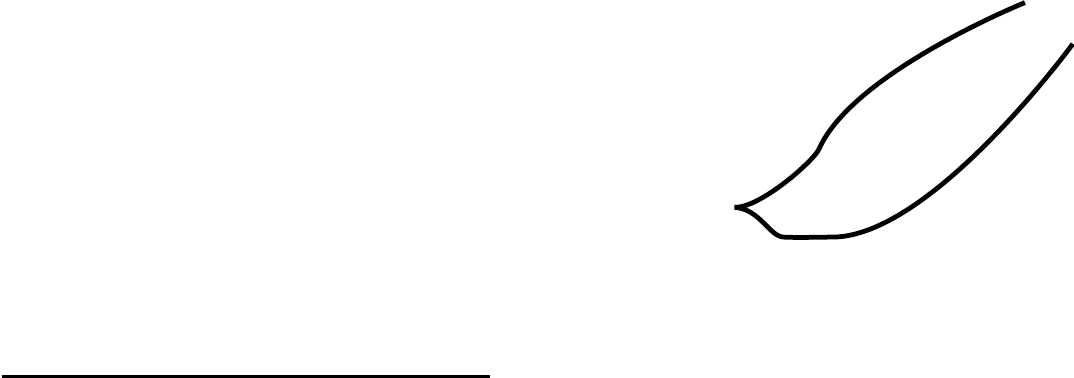}
    \caption{A Legendrian arc near the overtwisted disk.}
    \label{fig:arc near ot disk}
\end{subfigure}
\hfill
\begin{subfigure}{0.45\textwidth}
    \includegraphics[width=1\textwidth]{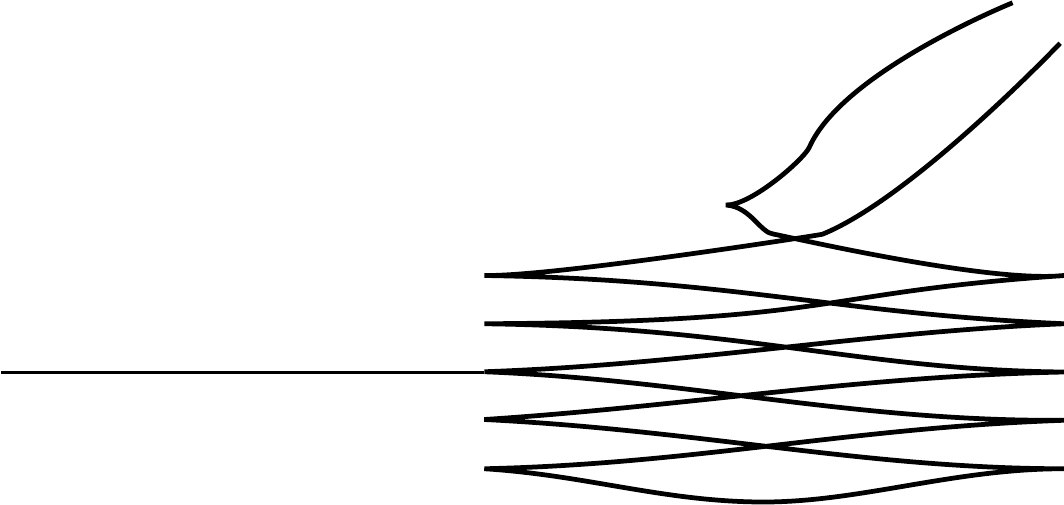}
    \caption{After a Legendrian isotopy, we may assume that the Legendrian arc intersects a small neighbourhood of the overtwisted disk as $f^{(k)}_\std$.}
    \label{fig:arc near ot disk as f_st}
\end{subfigure}

\caption{This figure shows how a Legendrian arc near an overtwisted disk can be isotoped to intersect a smaller neighbourhood as the standard model $f^{(k)}_\std$. We use the front projection in a Darboux chart with contact form $dz - y dx$, where the overtwisted disk has coordinates $\{z=0, x \leq 0\}$.}
\label{fig:standardizing near ot disk}
\end{figure}

\begin{remark}\label{rmk:IntoStdPosNearOTDiskDisconnected}
    If $L$ has finitely many connected components $L_\iota$, then we can use Remark \ref{rmk:StdPositionDisconnected} to show the following refinement of Lemma \ref{lem:IntoStdPosNearOTDisk}. Namely, for mutually disjoint $V_\iota$ and contactomorphisms $\psi_\iota:V_\iota \to U$ to the standard overtwisted model, we may assume that 3) and 4) hold with $V$ and $\psi$ replaced by $\bigcup_\iota V_\iota$ and $\psi_\iota$ with the additional condition that $g(L_{\iota'})$ does not intersect $V_\iota$ for $\iota' \neq \iota$. 
\end{remark}

\begin{remark}\label{rmk:StdPosOrientation}
    If $n = 1$ and we pick orientations of $f_\std^{(k)}$ and $f$, then for any tuple of signs $\sigma = (\sigma_1,...,\sigma_k) \in \{-1,+1\}^k$ we can choose $g_t$ in such a way that the orientation of $g_1$ induced by the orientation of $f$ agrees with $\sigma_i$ times the orientation of the $i$-th component of $f_\std^{(k)}$ under the identification given by $\psi$. Indeed, if we add one Reidemeister I move to the Legendrian shown in Figure \ref{fig:arc near ot disk as f_st}, then the local orientation changes. Thus, by adding more Reidemeister I moves in between the ones shown in Figure \ref{fig:arc near ot disk as f_st}, we can attain any prescribed orientation.
\end{remark}

\subsection{The overtwisted size is local in nature}  \label{ssec:oneStabilization}

A direct consequence of Lemma \ref{lem:IntoStdPosNearOTDisk} and Remark \ref{rmk:IntoStdPosNearOTDiskDisconnected} is that the overtwisted size is actually a property of the germ of contact form around the overtwisted disk. 
\begin{lemma} \label{lem:OTsizeIsLocal}
Let $\Delta \subseteq (M,\alpha)$ be a proper closed subset admitting a neighbourhood basis consisting of path-connected opens. Then
\begin{equation} \label{eq:sizeIsLocal}
\SS(\Delta) = 2\limsup_U \sup_f \SC(f;U,U \setminus \Delta)
\end{equation}
where the limit is taken over the poset of all path-connected open neighbourhoods $U$ of $\Delta$ (i.e.\,the limit as we take them smaller) and $f$ ranges over all proper Legendrian embeddings into $U$ that avoid $\Delta$, with finitely many components.
\end{lemma}
\begin{proof}
It is clear from the definitions that the left-hand side in \eqref{eq:sizeIsLocal} is at least as large as the right-hand side. It remains to prove the opposite inequality. 

Consider then an arbitrary Legendrian $f$ with image in some path-connected $U$ and disjoint from $\Delta$. Fix some other $V$, a path-connected neighbourhood of $\Delta$ with smooth boundary  $\partial V$, disjoint from $f$. Place a Darboux ball $B \subseteq U \setminus \Delta$ centered at a point in $\partial V$. We can isotope $f$ so it intersects $V$ only within $B$. According to Lemma  \ref{lem:FormalLegendrianStandardPosition}, this can be done with arbitrarily small cost if we allow ourselves to shrink $B$.

Write $g$ for the restriction of $f$ to $f^{-1}(V)$. It follows that the stabilization cost of $f$ is bounded above, up to an arbitrarily small error, by the stabilization cost of $g$ within $V$. This proves the claim.
\end{proof}

Moreover, for the standard overtwisted disk, we claim that the overtwisted size depends solely on being able to stabilize the standard model $f_\std$ (as introduced in Subsection \ref{ssec:standardPosition}). Even further, in dimension $3$ it suffices to consider the cost of stabilizing a single time. Let us introduce some terminology.

In dimension 3, we defined $\mathcal{C}^{N_+,N_-}(f;M,U)$ as the energy that it takes to realize an oriented Legendrian embedding $f:L \to (M,\alpha)$ as an arbitrarily small $(N_+,N_-)$-fold stabilization inside of $U \subset M$. If $L$ is connected, let us define
\begin{equation*}
    \mathcal{C}^1(f;M,U) \coloneqq \min(\mathcal{C}^{1,0}(f;M,U),\mathcal{C}^{0,1}(f;M,U)),
\end{equation*} 
the energy that it takes to realize $f$ as an arbitrarily small 1-fold stabilization. If $L$ is disconnected, we extend this definition by also taking the maximum over all components. Then:
\begin{lemma}\label{lem:OTsizeViaStdModel}
    Let $\Delta_\OT \subseteq (M,\alpha)$ be contactomorphic to the standard overtwisted disk $\Delta_\OT^\std$. If $\dim M = 3$, the overtwisted size of $\Delta_\OT$ is given by
    \begin{equation}\label{eq:OTSizeviaStdModel}
        \SS(\Delta_\OT) = 2\limsup_U \SC^1(\phi \circ f_\std;U)
    \end{equation}
If $\dim M \geq 5$, we have instead:
     \begin{equation}\label{eq:OTSizeviaStdModelgeq5}
        \SS(\Delta_\OT) = 2\limsup_U \SC(\phi \circ f_\std; U)
    \end{equation}
In both cases, the limit is taken over the poset of all path-connected open neighbourhoods $U$ of $\Delta_\OT$ (i.e.\,it is the limit as we take them smaller), and $\phi$ denotes the corresponding identification of an open neighbourhood of $\Delta_\OT^\std$ with $U$.
\end{lemma}
\begin{proof}
    The left-hand side in \eqref{eq:OTSizeviaStdModel} and \eqref{eq:OTSizeviaStdModelgeq5} is at least as large as the right-hand side. We prove the opposite inequality. 
    
    Take arbitrarily many disjoint copies $\Delta_\OT^{i}$ of $\Delta_\OT \subseteq (M,\alpha)$. In coordinates coming from the identification with $\Delta_\OT^\std$, these are simply small translations in the $z$-direction. If these copies are sufficiently close to each other, the contact forms on their open neighbourhoods differ by multiplication by a function that is close to the constant function $1$. 
    
    Let us assume that $\dim M\geq 5$. Let $U$ be some path-connected open neighbourhood of $\Delta_\OT$. We fix a contact isotopy $\phi_t$ supported in $U$, of energy $\SC(\phi \circ f_\std, U) + \eps$ for some small $\eps >0$, such that $\phi_1 \circ f_\std$ is an $\eps$-stabilization.
    
    Consider then $f:L \to U \setminus\Delta_\OT \subseteq (M ,\alpha)$, a Legendrian embedding with $l$ connected components. Take $l$ nearby copies $\Delta_\OT^i$ of $\Delta_\OT$. After a Legendrian isotopy of $f$ of arbitrarily small energy, we may assume by Lemma \ref{lem:IntoStdPosNearOTDisk} and Remark \ref{rmk:IntoStdPosNearOTDiskDisconnected} that each component of the image of $f$ intersects a neighbourhood of one of the $\Delta_\OT^i$ in a copy of the standard model $f_\std$. We can now use translations in $z$ to consider $l$ copies of $\phi_t$, that take place simultaneously. We see them as a single contact isotopy $\Phi_t$, supported on small disjoint neighbourhoods of the $\Delta_\OT^i$, all of them contained in $U$. Simultaneity implies that its energy is $\SC(\phi \circ f_\std;U^i)$, approximately, where $U^i$ is a neighbourhood of $\Delta_\OT^i$ disjoint from the other copies and disjoint from $\Delta_\OT$. Invoking the cofinality Lemma \ref{lem:OTsizeIsLocal} concludes the proof.

    If $\dim M = 3$, we argue in the same way, except that now we assume that $f$ intersects neighbourhoods of the copies of $\Delta_\OT$ as $f_\std^{(k)}$. It again follows that for any $k \in \N$, we can realize $f$ as a $k$-fold $\eps$-stabilization in the complement of $\Delta_\OT$ after an isotopy of energy $\SC^1(\phi \circ f_\std;U)$, approximately. By Remark \ref{rmk:StdPosOrientation}, we can, after fixing an orientation of $f$, prescribe the type of each of the $k$-stabilizations (positive or negative).
\end{proof}

\begin{remark}\label{rmk:OTSizeViaStdModelConvex}
    The proof relied on the fact that we can find arbitrarily many nearby copies of $\Delta_\OT$ in any small neighbourhood. It follows that the statement applies as well whenever we have a subset $\Delta \subseteq (M,\alpha)$ that is \emph{convex}, meaning that there is a contact vector field that immediately displaces it from itself, and has connected complement in $M$. We leave the details to the reader.
\end{remark}

\section{The model quantitative isotopy of loose Legendrians in dimension $3$}\label{sec:QuantitativeIsotopiesInDimension3}

This whole section is dedicated to the proof of Proposition \ref{prop:controlledIsotopy} and its generalization Proposition \ref{prop:SizeOfOTDisk}. This will also establish the upper bound claimed in Theorem \ref{thm:SizeOfOtDisk}.

\subsection{Proof of Proposition \ref{prop:controlledIsotopy}} \label{ssec:constructionIsotopy}

Let $f: [0,2] \to (\R^3,\alpha_\OT^\std)$ be a Legendrian embedding intersecting $\Delta_\OT^\std$ only in the point $(\rho,\theta,z) = (\pi,0,0) = f(1)$, where it is tangent to $\Delta_\OT^\std$. We may take $f$ to be a reparametrization of the standard Legendrian $f_\std$ defined in Section \ref{sec:bringing Legendrian in standard position}.

We will now construct an isotopy of Legendrian embeddings $(f_t)_{t \in [0,3]}$ with $f_0 = f$ satisfying the conclusions of Proposition \ref{prop:controlledIsotopy}.

\subsubsection{A piecewise object} \label{sssec:OTorange}

We will construct a family of piecewise smooth curves $(g_t)_{t \in [0,3]}: [0,6] \rightarrow \R^3$ that we will then desingularise to yield the $f_t$. It is convenient to first define the images $\Orange_t := g_t([1,5])$. We use cylindrical coordinates $(\rho,\theta,z)$ throughout.
\begin{itemize}
    \item For $t \in [0,1]$, $\Orange_t$ is the radial segment $\{\rho \in \pi[1-t,1];\, \theta,z=0\}$.
    \item For $t \in [1,2]$, $\Orange_t$ is the union of the circular segment $\{\rho =\pi;\, \theta \in [-\pi(t-1),\pi(t-1)] ;\, z=0\}$ and the radial segments $\{\rho \in [0,\pi];\, \theta = \pm \pi(t-1);\, z=0\}$.
    \item For $t \in [2,3]$, $\Orange_t$ is the union of the circle $\{\rho =\pi;\, z=0\}$ and the radial segment $\{\rho \in \pi[t-2,1];\, \theta = \pi,z=0\}$.
\end{itemize}
In particular, $\Orange_0$ consists of the point $(\pi,0,0)$ and $\Orange_3$ is the boundary of the overtwisted disk. We write $\Orange$ for the union of all the $\Orange_t$, which we regard as a subspace of $\Delta_\OT^\std \times [0,3] \subset \R^3 \times [0,3]$.

\subsubsection{A family of piecewise Legendrians}

We now define the promised family $(g_t)_{t \in [0,3]}: [0,6] \rightarrow \R^3$. Each individual curve will be piecewise smooth and tangent to $\xi_\OT^\std$, with image in $\Orange_t \cup f([0,2])$. Moreover, all of them will be given by
\begin{equation*}
    s \mapsto \begin{cases}
        f(s) \quad & s \in [0,1]\,, \\
        f(s-4) \quad & s \in [5,6]\,.
    \end{cases}
\end{equation*}
at the ends of the domain. The initial curve $g_0$ is the reparametrisation of $f_0 = f$ satisfying $g_0(s) := (\pi,0,0)$ if $s \in [1,5]$.

For $t \in [0,1]$, $g_t$ is the map:
\begin{equation*}
    s \mapsto \begin{cases}
        \big(\pi,0,0\big) \quad & s \in [1,2]\,,\\
        \big((1-(s-2)t)\pi,0,0\big) \quad & s \in [2,3]\,,\\
        \big((1-(4-s)t)\pi,0,0\big) \quad & s \in [3,4]\,,\\
        \big(\pi,0,0\big) \quad & s \in  [4,5]\,, \\
    \end{cases}
\end{equation*}
For $t \in [1,2]$, $g_t$ is given by the expression:
\begin{equation*}
    s \mapsto \begin{cases}
        \Big(\pi, \pi(t-1)(s-1),0\Big) \quad & s \in [1,2]\,,\\
        \Big(\pi(3-s),\pi(t-1),0\Big) \quad & s \in  [2,3]\,, \\
        \Big(\pi(s-3),-\pi(t-1),0\Big) \quad & s \in  [3,4]\,, \\
        \Big(\pi, \pi(t-1)(s-5),0\Big) \quad & s \in  [4,5]\,.
    \end{cases}
\end{equation*}
Lastly, for $t \in [2,3]$, $g_t$ is described by:
\begin{equation*}
    s \mapsto \begin{cases}
        \Big(\pi,\pi(s-1),0\Big) \quad & s \in  [1,2]\,, \\
        \big(\pi - (s-2)(3-t)\pi,\pi,0\big) \quad & s \in  [2,3]\,, \\
        \big(\pi - (4-s)(3-t)\pi,\pi,0\big) \quad & s \in  [3,4]\,, \\
        \Big(\pi,\pi(s-5),0\Big) \quad & s \in [4,5]\,.
    \end{cases}
\end{equation*}

\begin{figure}
\centering
\begin{subfigure}{0.30\textwidth}
    \includegraphics[width=0.6\textwidth]{figures/h1_cropped.pdf}
    \caption{$t=0$.}
    \label{fig:h1}
\end{subfigure}
\hfill
\begin{subfigure}{0.30\textwidth}
    \includegraphics[width=0.6\textwidth]{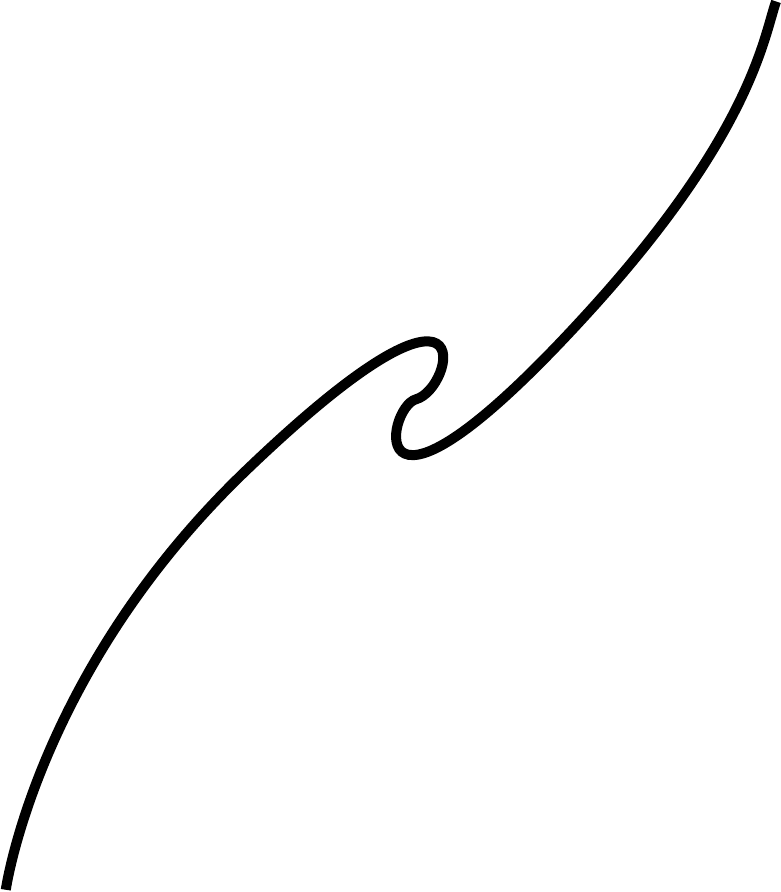}
    \caption{$0 < t < 1$.}
    \label{fig:h2}
\end{subfigure}
\hfill
\begin{subfigure}{0.30\textwidth}
    \includegraphics[width=\textwidth]{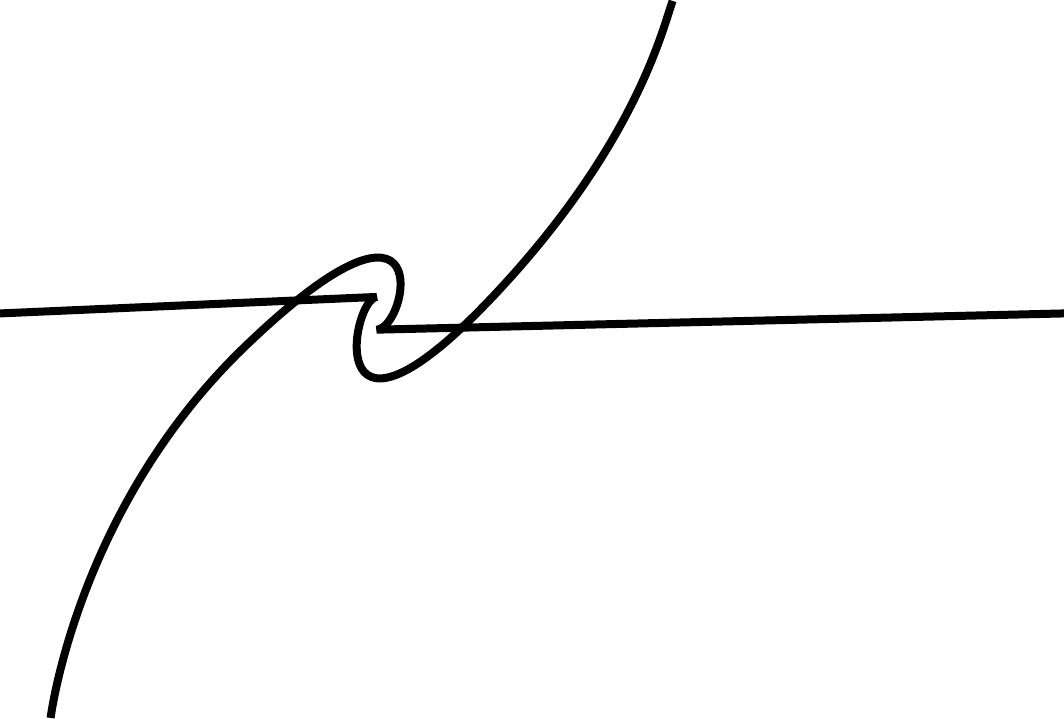}
    \caption{$1 < t \approx 1$.}
    \label{fig:h3}
\end{subfigure}
\hfill
\begin{subfigure}{0.30\textwidth}
    \includegraphics[width=\textwidth]{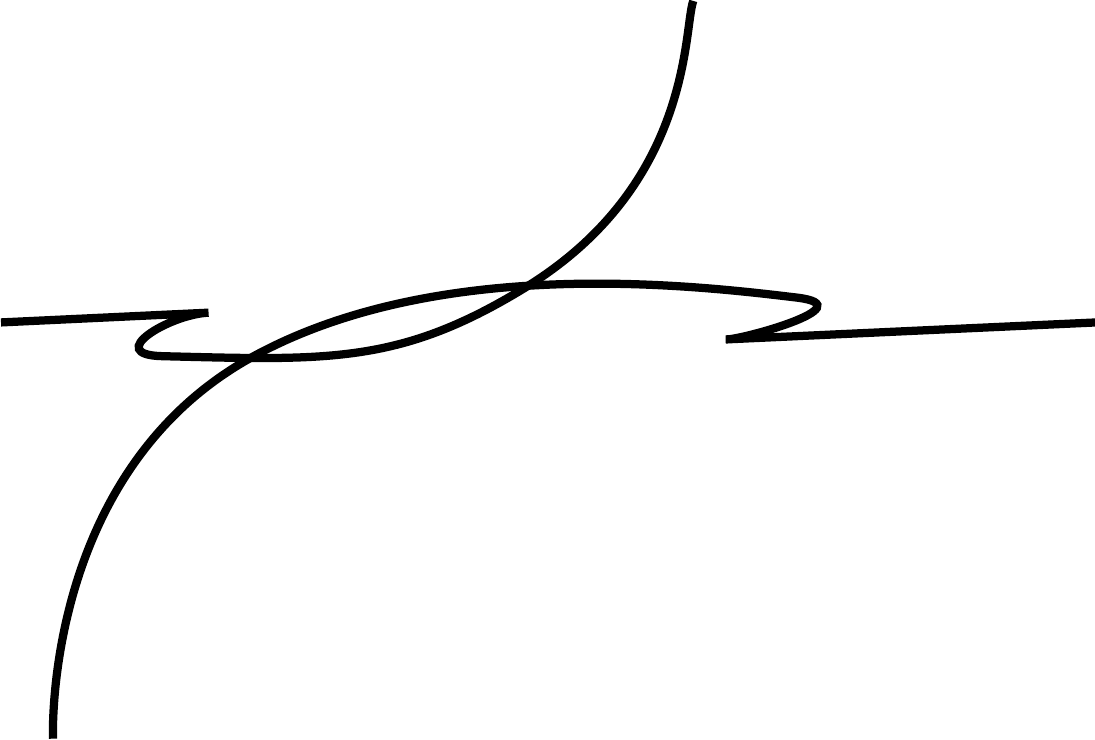}
    \caption{$1< t < 2$.}
    \label{fig:h4}
\end{subfigure}
\hfill
\begin{subfigure}{0.30\textwidth}
    \includegraphics[width=\textwidth]{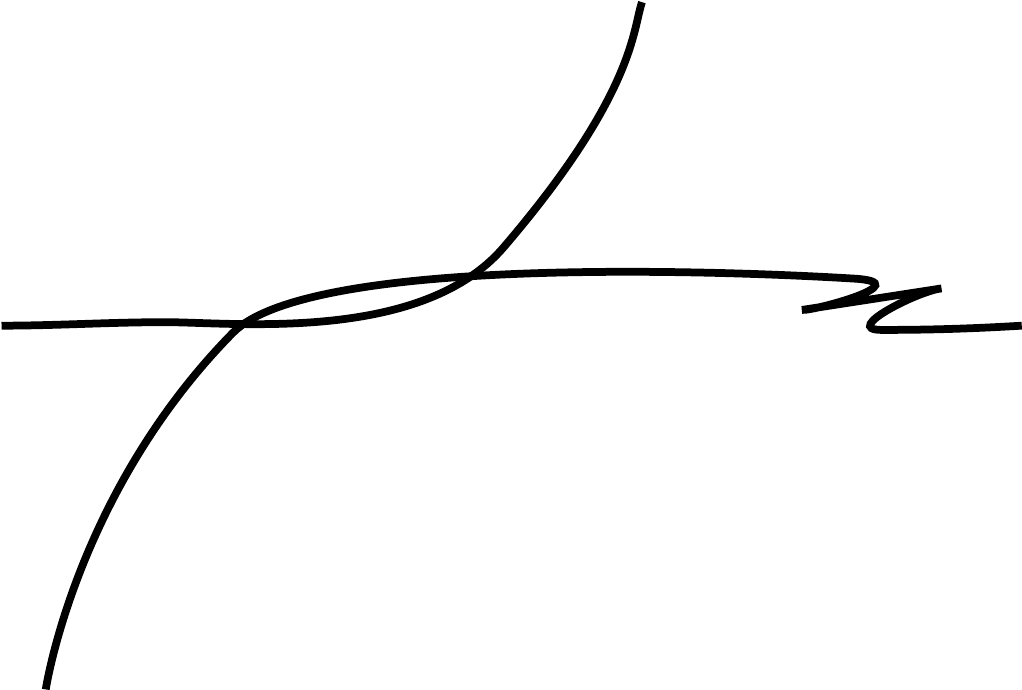}
    \caption{$t = 2$.}
    \label{fig:h5}
\end{subfigure}
\hfill
\begin{subfigure}{0.30\textwidth}
    \includegraphics[width=\textwidth]{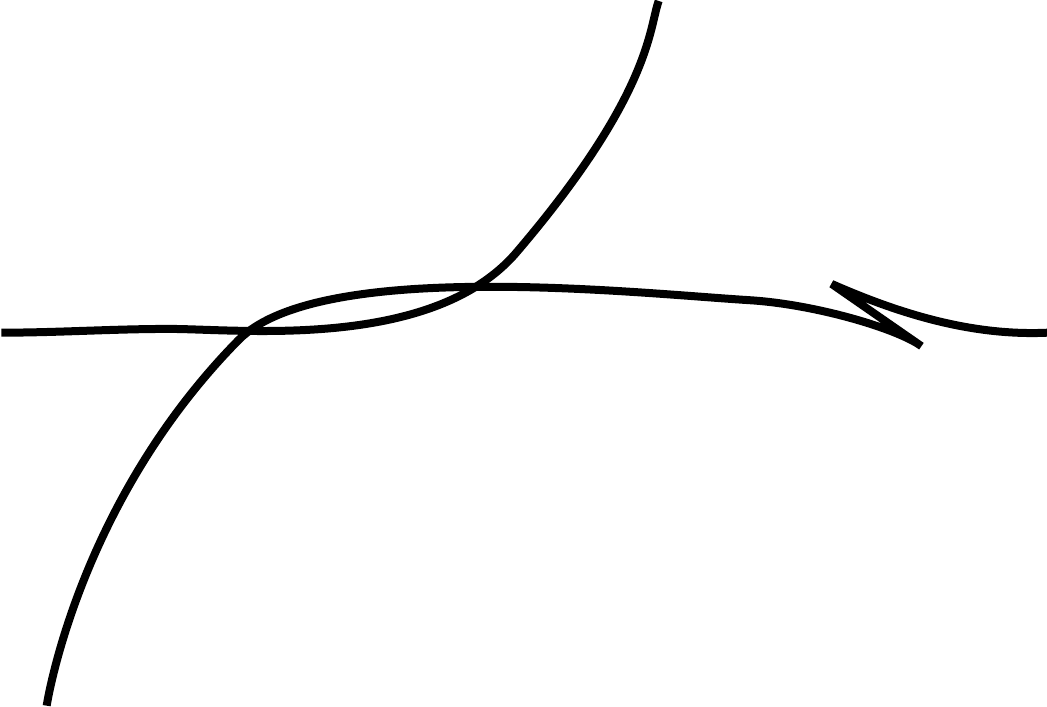}
    \caption{$t = 3$.}
    \label{fig:h6}
\end{subfigure}

\caption{The front of $h_t$ for different values of $t$. At $t = 0$, the image of $h_t$ is contained in the region $\{\pi \leq \rho < 2 \pi\}$ and intersects the overtwisted disk in a single point. For $t \in (0,1)$, $h_t$ moves a small part of the arc towards the origin; the front is twisting. Subfigure (c) shows the front of $h_t$ right after crossing the $z$-axis. To go from (b) to (c), we applied the Reidemeister move RIV. For $t \in (1,2)$, the sub-arcs connecting the set $\{\rho = \pi\}$ to $\{\rho \approx 0\}$ rotate in the $\pm \theta$-direction. To go from (c) to (d), we applied RV (twice) and then RII (also twice). For $t \in (2,3)$, the sub-arcs connecting  $\{\rho = \pi\}$ to $\{\rho \approx 0\}$ contract to $\{\rho = \pi\}$; the front (in particular, the stabilization it contains) untwists. At $t=3$, we find a Legendrian embedding that is a stabilization of a connected sum of $h_0$ and the boundary of the overtwisted disk.}
\label{fig:Stabilizing_isotopy_3d}
\end{figure}

\subsubsection{Smoothing and end of the proof}

Using the family $g_t$ we now construct a family of smooth Legendrians $h_t$. This is done in two steps: (1) smoothing the corners, (2) perturbing to achieve the embedding condition. These modifications take place in an arbitrarily small neighbourhood of $\Orange_t$.

Providing an explicit formula for $(h_t)_{t \in [0,3]}$ is not particularly instructive, nor practical or necessary. Instead, we have opted to depict it in Figure \ref{fig:Stabilizing_isotopy_3d}, using the cylindrical projection. Moreover, the knot diagrams shown in the figure can be assumed to satisfy the following properties:
\begin{itemize}
    \item $h_t|_{[1,5]}$ is arbitrarily close to  $\Orange_t$, and in particular arbitrarily close to $\Delta_\OT^\std$.
    \item $h_3$ is obtained from $h_0$ by surgery with an arbitrarily small stabilization of $\partial\Delta_\OT^\std$.
    \item Away from the singularities of $g_t$, the curves $h_t$ and $g_t$ are $C^1$-close.
    \item $h_t$ agrees with $f$ at the ends of its domain.
    \item $h_0$ is a reparametrisation of $f$.
\end{itemize}
The last item implies that there is a Legendrian isotopy connecting the two with arbitrarily small cost. To conclude the proof we have to establish the claimed quantitative bound for the isotopy just described.

First we claim that for any given $\eps > 0$ we can choose $h_t$ satisfying $|\alpha_\OT^\std(\partial_t h_t)| < \eps$ for every $t \not\in [1,2]$. Indeed, over this time region the isotopy $h_t$ is supported in an arbitrarily small neighbourhood of a radial arc, which is Legendrian. Corollary \ref{cor:boudnednessOfStrictWeinsteinNbhdOfHeightC} then implies that the isotopy can be assumed to have arbitrarily small energy.

We thus focus on $t \in [1,2]$. Note that the family $g_t$, restricted to $s \in [1,2] \cup [4,5]$, spans the boundary of the overtwisted disk. As such, we can apply  Corollary \ref{cor:boudnednessOfStrictWeinsteinNbhdOfHeightC} again and bound the energy contribution by $\eps$. It follows that the interesting interval is $s \in [2,4]$, where $g_t$ consists of two rotating radial arcs. Over each arc we obtain the estimate:
\begin{equation*}
    |\alpha_\OT^\std(\partial_t g_t)| = \pi |\rho \sin(\rho) d\phi(\partial_\phi)| = \pi |\rho \sin(\rho)|.
\end{equation*}
We claim that this applies to $h_t$ as well, up to an error that can be made arbitrarily small by shrinking the smoothing area: Indeed, the smoothing points are contained in neighbourhoods of legendrian arcs, so the energy contribution there can be assumed to be small by shrinking said neighbourhoods.

In summary, the energy of this isotopy, up to small error, is given by
    \begin{equation} \label{eq:CostStandardDisc}
       \int_0^3 \max |\alpha_\OT^\std(\partial_t h_t)| dt \approx \pi \max_{\rho \in [0,\pi]} \rho \sin(\rho) = \frac{C^*}{2}.
    \end{equation}
This concludes the proof of Proposition \ref{prop:controlledIsotopy}. \hfill \qedsymbol


This implies, by Lemma \ref{lem:OTsizeViaStdModel}, that the overtwisted size of the standard $3$-dimensional overtwisted disk is bounded from above by $C^*$, as claimed in Theorem \ref{thm:SizeOfOtDisk}.

\subsection{Overtwisted disks with arbitrary contact form}

\begin{proof}[Proof of Proposition \ref{prop:SizeOfOTDisk}]
Recall the key quantity
\[  c_\textrm{Hofer} = \int_0^{2\pi} c(\theta) d\theta, \qquad c(\theta) := \max_{\rho \in [0,\pi]} |\alpha(\rho,\theta)(\partial_\theta)|. \]

The argument is exactly the same as for the standard disk. The isotopy we constructed still has energy concentrated in the time interval $t \in [1,2]$ and domain interval $s \in [2,4]$, where it is given by two rotating radial arcs. It follows, reasoning as above, and up to small error, that the cost of our isotopy is given by the following analogue of Equation \ref{eq:CostStandardDisc}:
\[  c_\textrm{stab} := \int_0^{\pi} \max\{c(\theta),c(-\theta)\} d\theta, \]
but with this we only get the estimate:
\[ c_\textrm{stab} \leq \int_0^{\pi} c(\theta) + c(-\theta) d\theta = 
\int_0^{\pi} c(\theta) + \int_\pi^{2\pi} c(\theta) d\theta = c_\textrm{Hofer}. \]
By definition, the overtwisted size is bounded above by twice the cost of the isotopy we constructed, so we need $c_\textrm{stab} \leq \frac{c_\textrm{Hofer}}{2}$. In order to obtain the improved estimate, we modify the construction of the isotopy as follows: Let $\theta_0 \in [0,2\pi]$ be such that 
\begin{equation*}
    \int_0^{\theta_0}c(\theta) d\theta = \int_{\theta_0}^{2\pi} c(\theta)d\theta = \frac{1}{2} \int_0^{2\pi}c(\theta) d\theta.
\end{equation*}
We let further $\phi_+:[1,2] \to [0,\phi_0]$ and $\phi_-:[1,2] \to [\theta_0,2\pi]$ be reparametrizations, where $\phi_+$ is orientation preserving and $\phi_-$ is orientation reversing, such that 
\begin{equation*} 
    c(\phi_+(t))\phi_+'(t) = -c(\phi_-(t))\phi_-'(t) = \frac{1}{2} \int_0^{2\pi}c(\theta) d\theta.
\end{equation*}
Then we modify the definition of $g_t$, in the above construction of the isotopy, such that (a) during the time interval $t \in [1,2]$, the two radial segments are at angle $\phi_+(t)$ and $\phi_-(t)$ instead of $\pi(t-1)$ and $-\pi(t-1)$, (b) during the time interval $[2,3]$, the isotopy is supported near $\{\theta = \theta_0\}$.

The cost of the resulting isotopy is then given by $\frac{1}{2}\int_0^{2\pi} c(\theta)d\theta$, as required. This finishes the proof.
\end{proof}

\section{Creating a small loose chart and the model overtwisted disk in higher dimensions}\label{sec:ProofStabOtSize}

We now study the trace of the isotopy produced in Proposition \ref{prop:controlledIsotopy}. This will establish Theorem \ref{thm:OtDiskOfSizeCTimesBall}.

\begin{proof}[Proof of Theorem \ref{thm:OtDiskOfSizeCTimesBall}]
Let $\eps > 0$. Consider
\[ M^{(n)}_\eps := M \times \left([-1-\eps,1+\eps] \times \left[-(1+\eps)\left(\frac{C}{2}+\eps\right),(1+\eps)\left(\frac{C}{2}+\eps\right) \right]\right)^n; \]
we take coordinates $(x_i,y_i)$ in the second factor and set
\[ \alpha^{(n)} := \alpha + \sum_{i=1}^n y_i dx_i. \]
For each open subset $U \subset M$ we can define $U^{(n)}_\eps$ analogously. Lastly, we also introduce:
\[ \Delta^{(n)} := \Delta \times \left([-1,1] \times \left[-\frac{C}{2},\frac{C}{2}\right]\right)^n \subset M^{(n)}_\eps. \]

We must show that every proper Legendrian $F$ in some $U^{(n)}_\eps$, disjoint from $\Delta^{(n)}$, can be isotoped, with cost at most $\frac{C}{2} + \eps$, to exhibit a explicit small stabilization.

\textbf{Step 1: the model.} We first study the case of a certain model family of Legendrians.

Since $\Delta \subseteq (M,\alpha)$ has overtwisted size $C$, it follows that for any neighbourhood $V \subseteq M$ of $\Delta_\OT$ and any proper Legendrian embedding $f:L\to V \setminus \Delta$, there exists a Legendrian isotopy $f_t$, with compact support and energy less than $C/2 + \eps$, from $f$ to a stabilization contained in an arbitrarily small Weinstein neighbourhood of some other Legendrian embedding. It follows that, up to time reparametrization, we may assume $|\alpha(\partial_t f_t)| \leq \frac{C}{2} + \eps$.

Fix one such $f$. Fix $\zeta:[-1-\eps,1+\eps] \to [0,1]$, a smooth function that is equal to $0$ near $[-1-\eps,-1]$ and $[1,1+\eps]$, equal to $1$ near $0$, and satisfies $|\zeta'| < 1 + \eps$. Then we construct Legendrian isotopies for $n\geq1$
\begin{equation*}
\begin{split}
    f^{(n,\eps)}_t: L \times [-1-\eps,1+\eps]^n \quad&\longrightarrow\quad \left(V^{(n)}_\eps,\alpha^{(n)} \right) \\
    (q,x_1,...,x_n) \quad &\mapsto\quad \left(f^{(n-1,\eps)}_{t \zeta(x_n)}(q,x_1,...,x_{n-1}), x_n, -\alpha^{(n-1)}\left(\frac{d}{dx_n}f^{(n-1,\eps)}_{t \zeta(x_n)}(q,x_1,...,x_{n-1})\right) \right),
\end{split}
\end{equation*}
with $f^{(0,\eps)}_t := f_t$.

\textbf{The meaning of this formula.} Recall that the Legendrian trace of a Legendrian isotopy is defined as follows: we see $t$ as an extra position variable $x_1$, meaning that we introduce a new dual momentum variable $y_1$ as well. Setting $y_1$ to be minus the Hamiltonian generating the isotopy yields then a higher-dimensional Legendrian.

The formula provided above amounts to taking the Legendrian trace in multiple directions. Moreover, we cut off with the function $\zeta$ in order to yield a Legendrian that is a product in certain regions. We also retain $t$ as a time parameter, since our goal is to produce a stabilizing isotopy in higher dimensions.

\textbf{Step 2. Properties of the model.} Induction on $n$ shows that $f^{(n,\eps)}_t$ is indeed Legendrian for all $t$. We also have to verify that this definition is well-posed. Namely, we must show that the image is indeed in $V^{(n)}_\eps$ (meaning that the $y_n$-coordinate of $f^{(n,\eps)}_t$ lies in $[-(1+\eps)(C/2+\eps),(1+\eps)(C/2+\eps)]$). We do this last. First we compute the generating Hamiltonian of $f^{(n,\eps)}_t$:
\begin{equation*}
    \alpha^{(n)}(\partial_t f^{(n,\eps)}_t(q,x_1,...,x_n)) = \zeta(x_n) \alpha^{(n-1)}(\partial_s f_s^{(n-1,\eps)}(q,x_1,...,x_{n-1}))|_{s = t \zeta(x_n)}
\end{equation*}
which, by induction, shows that $|\alpha^{(n)}(\partial_t f^{(n,\eps)}_t(q,x_1,...,x_n))| < C/2 + \eps$. In particular, the energy of the isotopy $f_t^{(n,\eps)}$ is bounded by $\Vert f_t^{(n)} \Vert_{\alpha^{(n)}} < \frac{C}{2}+\eps$. This implies the promised bound: 
\begin{align*}
    \left|\alpha^{(n-1)}\left(\frac{d}{dx_n}f^{(n-1,\eps)}_{t \zeta(x_n)}(q,x_1,...,x_{n-1})\right)\right| &     = \left|t \zeta'(x_n) \alpha^{(n-1)}\left(\frac{d}{ds}f^{(n-1,\eps)}_s|_{s=t \zeta(x_n)}(q,x_1,...,x_{n-1})\right)\right| \\
    & \leq (1+\eps)\left(\frac{C}{2}+\eps\right).
\end{align*}

Lastly, $f^{(n)}_1$ has a loose chart if $n \geq 1$ whose size can be made arbitrarily small by choosing the stabilization of $f_1$ to be sufficiently small.

\textbf{Step 3. General case.} Consider then an arbitrary Legendrian embedding $F$ mapping into $U^{(n)}_\eps$ and disjoint from $\Delta^{(n)}$. Fix a smaller closed neighbourhood $V \subset U$ of $\Delta$ and some $0 < \delta < \eps$ such that $F$ is disjoint from $V^{(n)}_{2\delta}$. We claim that there is a proper Legendrian embedding $f$ of the compact disk into $V$ such that, after a contact isotopy of arbitrarily small cost, $F$ meets $V^{(n)}_{\delta}$ exactly in $f^{(n,\delta)}_0$ for some Legendrian embedding $f$. Then the result will follow from the model case.

For the standard overtwisted disk, this claim is Lemma \ref{lem:IntoStdPosNearOTDisk}. In this more general setting, the proof is completely analogous. Indeed, the model Legendrian embedding $f_\std$ introduced in Subsection \ref{ssec:standardPosition} can be assume to have image in a small Darboux ball, in a way that the trace $(f_\std)^{(n,\eps)}_0$ of the model Legendrian in higher dimensions is again equal to the model Legendrian $f_\std$ in higher dimensions. Thus the same construction that was used in Section \ref{sec:bringing Legendrian in standard position} applies here.
\end{proof}

This result immediately implies the upper bounds for the overtwisted size of the standard overtwisted disk claimed in Theorem \ref{thm:SizeOfOtDisk}. The lower bound requires other methods, which we will introduce in Section \ref{sec:ClosedReebOrbitsInOTCtcMflds}. We postpone the proof until then.

\section{Small energy isotopies of stabilized Legendrians} \label{sec:smallEnergyIsotopies}

In this section we explain how stabilizations allow us to homotope Legendrian isotopies so that they have small energy. We address the problem first assuming that the Hamiltonian generating the isotopy is non-negative; this is done in Subsection \ref{ssec:nonNegativeIsotopies}. The general case is then tackled in Subsection \ref{ssec:generalIsotopies}.

This section is technical. Its consequences are spelled out in the next section.

\subsection{Preliminaries: Approximating by piecewise positive/negative Hamiltonians} \label{sec:positiveNegative}

In order to reduce the general case to the setting of non-negative isotopies, we now prove that Hamiltonian isotopies can be discretised: i.e.\,we can approximate them by concatenations of non-negative and non-positive isotopies.
\begin{lemma}\label{lem:ApproxByPWPosNegHamiltonian} 
Fix constants $C, s_0 > 0$. Then, there is a natural number $n$, depending on $C$ and $s_0$, with the following property:

Let $(M,\alpha)$ be a strict contact manifold and $N \subset M$ a compact subdomain. Suppose that the Reeb flow $\rho_s$ of $\alpha$, applied to $N$, exists over the interval $[0,s_0]$. Then, for any time-dependent Hamiltonian $(H_t)_{t \in [0,1]}: (M,\alpha) \rightarrow (-C,C)$, supported on $N$, there is a  piecewise smooth Hamiltonian $G_t: (M,\alpha) \rightarrow (-2C,2C)$ such that:
\begin{itemize}
    \item Over each $[i/n,(i+1)/n] \subset [0,1]$, $G_t \circ \phi^G_t$ is smooth and non-negative or non-positive.
    \item $\phi^G_{2i/n} = \phi^H_{2i/n}$ for all $i$.
\end{itemize}
\end{lemma}
\begin{proof}
Consider the subset $A \subset M$ swept by $N$ as we flow it using the Reeb flow for times in $[0,s_0]$. Fix $F_+:M \to [0,C]$, a compactly supported autonomous Hamiltonian which is equal to $C$ over $A$. 

Any even integer $n$ greater than $C/s_0$ works; we fix such an $n$ now. We divide $[0,1]$ into $n$ intervals of size $1/n$. We define the piecewise autonomous Hamiltonian $F_t:M \to [-C,C]$ as $F_+$ for $t \in [2i/n,(2i+1)/n]$ and as $-F_+$ for $t \in [(2i+1)/n,(2i+2)/n]$. Then we have that $\phi^F_{2i/n} = \textnormal{id}_M$ for all $i$. The isotopy $\phi^F_t$ shifts $N$ within $A$, along the Reeb flow.

We set $\phi^G_t \coloneqq \phi^F_t \circ \phi^H_t$ which is generated by the Hamiltonian
\begin{equation*}
    G_t = F_t + H_t \circ \left(\phi^F_t\right)^{-1};
\end{equation*}
observe that conformal factors do not appear in the equation because the Reeb flow is strict. The isotopy $G_t$ has the required properties, since $H_t\circ \left(\phi^F_t\right)^{-1}$ is supported  on $A$.
\end{proof}
One can think of this argument as a form of convex integration. As we increase $n$ (keeping $C$ and $s_0$ fixed) the isotopy $\phi^G_t$ will converge to $\phi^H_t$ in $C^0$.

\subsection{The case of non-negative isotopies} \label{ssec:nonNegativeIsotopies}

 Let $\alpha_\std = dz + p dq$ denote the standard contact from on $T^*S^1$. The reader should think of the upcoming lemma as allowing us to create Legendrian isotopies whose Hamiltonian we can prescribe up to a $C^0$-error. The price to pay is that the Legendrian being isotoped is not chosen by us (if the statement was about the zero section in jet space it would apply to any Legendrian, thanks to the Weinstein neighbourhood theorem), but is rather an $N$-fold stabilization of it.
\begin{lemma}\label{lem:ApproximatePositiveHamiltonians}
Given any $C>0$ and $\eps > 0$, there exists a natural number $N = N(C,\eps) \in \N$ with the following property.

Fix a family of functions $H_t: J^1S^1 \to [0,C)$, $t \in [0,1]$, and a neighbourhood $W$ of the zero section in $J^1S^1$. Then, there exists a Legendrian isotopy $f_t:S^1 \to W$ such that $\alpha_\std(\partial_t f_t)$ is $\eps$-close to $H_t \circ f_t$, for all $t \in [0,1]$.

Moreover, each Legendrian $f_t$ is an $N$-fold stabilization of the zero section within $W$. All stabilizations are of the same sign (either all positive or all negative for the standard orientation on $S^1$).
\end{lemma}
Observe that we cannot simply let $f_t$ be induced by the contact isotopy generated by $H_t$. Indeed, said isotopy may not restrict to a contact isotopy of $W$. 

\begin{proof}
    \textbf{Step 1:} We claim that it is enough to consider the case $H_t > 0$.

    Assume that Lemma \ref{lem:ApproximatePositiveHamiltonians} holds for positive Hamiltonians. Then for any sufficiently small $\eps'> 0$ we can apply Lemma \ref{lem:ApproximatePositiveHamiltonians} to $C + \eps'$ and $\eps - \eps'$ to construct $f_t$ such that  $\alpha_\std(\partial_t f_t)$ is $\eps - \eps'$-close to $\eps' + H_t \circ f_t$ for all $t \in [0,1]$. It follows that  $\alpha_\std(\partial_t f_t)$ is $\eps$-close to $H_t \circ f_t$.
    
    \textbf{Step 2:} From now on assume that $H_t > 0$. We claim that it is enough to consider the case of the autonomous Hamiltonian $H = 1$.

    Let $\eps' > 0$ be any positive number smaller than $\eps$. Assuming that the assertion holds for $H=1$, there exists a Legendrian, which is an $N(1,\frac{\eps'}{C})$-fold stabilization of the zero section, and a corresponding Legendrian isotopy $f_t:S^1 \to \widetilde{W}$ such that $\alpha_\std(\partial_t f_t) \in [1-\frac{\eps'}{C},1+\frac{\eps'}{C}]$. Here $\widetilde{W}$ is some small neighbourhood of the zero section, which we will fix below.

    Suppose then that we have some arbitrary family $H_t$. Consider the time-dependent contact form $\alpha_t \coloneqq \frac{1}{H_t} \alpha_\std$ on $J^1S^1$. Recall that its Reeb vector field is precisely the Hamiltonian vector field associated to $H_t$ with respect to $\alpha_\std$. This is precisely where positivity of $H_t$ enters the argument.
    
    According to the Weinstein neighbourhood theorem (Proposition \ref{thm:StrictWeinsteinNeighbrouhoods}), the zero section admits a time-dependent neighbourhood $(W_t,\alpha_t) \subset W$ that is strictly contactomorphic via $\phi_t: W_t \to \widetilde{W}$ to a neighbourhood $\widetilde{W}$ of the zero section in $(J^1S^1,\alpha_\std)$. As such, the Hamiltonian vector field of $H_t$ in $W_t$ is identified, via $\phi_t$, with the Reeb vector field in $\widetilde{W}$, which is generated by the Hamiltonian $H=1$. It follows that the Legendrian isotopy $\phi_t^{-1} \circ f_t$ satisfies the conclusions of the lemma. Indeed, using that $\phi_t$ is a strict contactomorphism, we see that 
    \begin{equation*}
    \begin{split}
        \alpha_\std\left(\partial_t \left(\phi^{-1}_t \circ f_t \right) \right) &= H_t \circ \left(\phi^{-1}_t \circ f_t \right) \cdot \alpha_t\left(\partial_t \left(\phi^{-1}_t \circ f_t \right) \right) \\
        &= \left(H_t \circ \phi^{-1}_t \circ f_t \right) \cdot \left(\alpha_\std(\partial_t f_t) + \alpha_t((\partial_t \phi^{-1}_t) \circ f_t) \right)\\
        &\in \left[\left(1-\frac{\eps}{C}\right)H_t \circ \left(\phi^{-1}_t \circ f_t \right), \left(1+\frac{\eps}{C}\right)H_t \circ \left(\phi^{-1}_t \circ f_t \right) \right] \\
        &\subseteq [H_t \circ \left(\phi^{-1}_t \circ f_t \right) -\eps,H_t \circ \left(\phi^{-1}_t \circ f_t \right) +\eps],
    \end{split}
    \end{equation*}
    where the $``\in"$ and $``\subseteq"$ are understood pointwise in $S^1$. Here we use that $\eps' < \eps$ and the fact that $\alpha_t(\partial_t \phi^{-1}_t)$ is zero along the zero section so we can assume $\alpha_t((\partial_t \phi^{-1}_t) \circ f_t)$ to be arbitrarily small by compressing $\widetilde W$. This proves the claim.

    \begin{figure}
    \centering	\includegraphics[width=0.5\textwidth,height=0.9\textheight,keepaspectratio]{"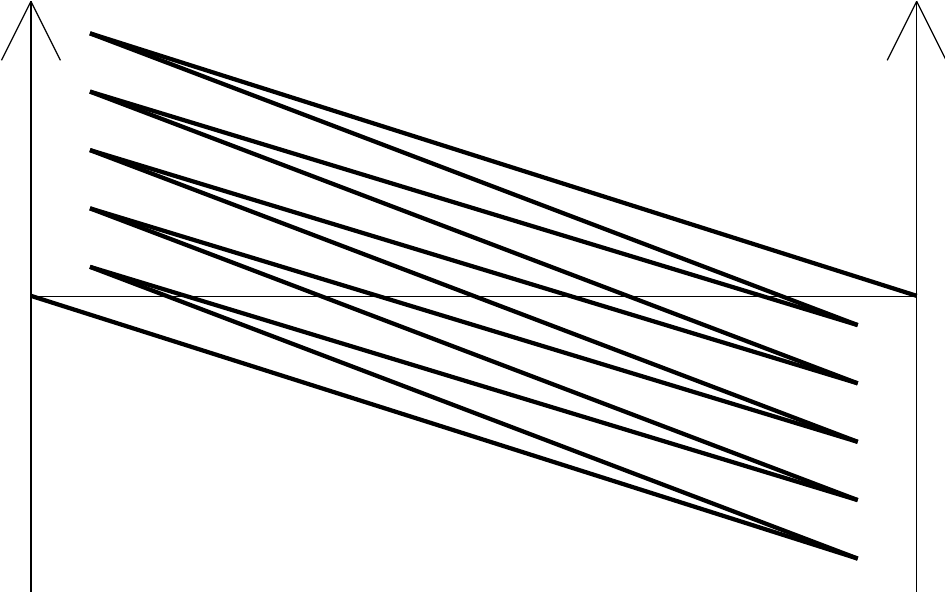"}
	\caption{The front of the stabilized Legendrian $f:S^1 \to J^1S^1$ used in the Lemma \ref{lem:ApproximatePositiveHamiltonians}, Step 3.}
	\label{fig:S1-stabilized}
    \end{figure}

    \textbf{Step 3:} It remains to establish the case of $H_t=1$ autonomous. We assume $\eps < 1$. Choose $N > \frac{1}{\eps}$. Choose $\delta>0$ so small that $\{|p|, |z|\leq (1+\eps)\delta \} \subseteq W$, where $(q,p,z)$ are the coordinates on $J^1S^1$.
    
    Then let $f:S^1 \to W$ be an $N$-fold stabilization of the zero section in $J^1S^1$ such that minus the slope of the front of $f$ is contained in $[(1-\eps)\delta,(1+\eps)\delta]$, see Figure \ref{fig:S1-stabilized}. Let $F$ be the Hamiltonian $F(q,p,z) = \frac{p}{\delta}$ whose associated contact isotopy is given by $\phi^F_t(q,p,z)=(q+\frac{t}{\delta},p,z)$. In particular, $\phi^F_t(f(S^1)) \subseteq W$ for all times, and after cutting off $F$ outside a neighbourhood of the image of $\cup_t \phi^F_t(f(S^1))$, we may assume that $F$ has compact support in $W$. Since the $p$-coordinate of $f$ is contained in $\subseteq [(1-\eps)\delta,(1+\eps)\delta]$, we obtain 
    \begin{equation*}
        \alpha_\std(\partial_t (\phi^F_t \circ f)) = F \circ \phi_t \circ f \in [1-\eps,1+\eps],
    \end{equation*}
    which finishes the proof.
\end{proof}

One can also establish an analogue of Lemma \ref{lem:ApproximatePositiveHamiltonians} for $J^1I$, with $I$ an arbitrary $1$-dimensional manifold. This is Lemma \ref{lem:ApproximatePositiveHamiltoniansOnR} below. It will follow from the following auxiliary result, which deals with the particular case in which $I = \R$, the Hamiltonian is the indicator function of $[0,1]$, and we allow some (big but controlled) error.
\begin{lemma} \label{lem:ApproximatePositiveHamiltoniansOnRAux}
Fix $\eps >0$. There exists $N = N(\eps) \in \N$ such that for any neighbourhood $W\subseteq J^1\R$ of $\R$ there exists a compactly supported family of Legendrian embeddings $f_t: \R \to W$ with $t \in [0,1]$ satisfying 
\begin{align*}
\alpha_\std(\partial_t f_t)(x) \in [1-\eps,1+\eps] &\quad\text{ if } \pi_q(f_t(x)) \in [0,1] \\
\alpha_\std(\partial_t f_t)(x) \in [-10,10] &\quad\text{ if } \pi_q(f_t(x)) \in \R \setminus (0,1),
\end{align*}
where $\pi_q$ denotes the projection onto the $q$-coordinate. Moreover, each $f_t$ is an $N$-fold stabilization of $\R \subset W$. All stabilizations have the same sign.
\end{lemma}
\begin{proof}
As in the proof of Lemma \ref{lem:ApproximatePositiveHamiltonians}, we take a number of stabilizations stacked on top of each other, and we move those stabilizations in one direction (determined by the stabilization sign). The issue is that this sideways pushing will take the stabilizations outside of $J^1[0,1]$. This forces us to stop the pushing as we reach its boundary. Moreover, one may worry that once a stabilization has been moved across $J^1[0,1]$ (and has stopped) it is no longer of use to us, so we may need more stabilizations the smaller $W$ is.

To address these issues we use two ``stabilization packages''. A package is shifted sideways to approximate $H=1$. As we stop it, and the other one takes its place, we also shrink it. Once it is very small, it can be shifted in the other direction, with arbitrarily small cost, to be reused. This shifting happens within a small Weinstein neighbourhood of the second package. With the strategy in mind, let us proceed to the proof.

\textbf{Step 1}: We first construct the model ``stabilization package'' $\widetilde{g}$ in $J^1\R$. It is sufficient that we describe the front of $\widetilde{g}$; it is depicted in Figure \ref{fig:g_tilde}.  Indeed, parameterizations are not very important, since reparameterizations can be performed with zero cost.

We use coordinates $(q,p,z)$ on $J^1\R$, which is endowed with the contact form $\alpha_\std = dz +pdq$. We write $\pi_q$, $\pi_z$, and $\pi_p$ for each of the coordinate projections.

The construction of $\widetilde{g}$ is qualitative. We fix $\delta > 0$ so small that for every $q \in [-3,4]$, the fiber of $W$ over $q$ contains $\{|p| < 2\delta, |z| < 2\delta\}$. We then fix $p_+,p_- > 0$ satisfying $\delta > p_+ > p_-$ and $\frac{p_+}{p_-} < 1 + \eps$. These are meant to be the slopes of the different branches of the front; in particular, traversing two consecutive branches of length $l$ yields a vertical displacement of size $l(p_+-p_-)$. We also leave some buffer space at each side of the package; its size $\eps_2 > 0$ satisfies
\[ (2-2 \eps_2)((K+1)p_- - Kp_+) = -2\eps_2 p_- \]
for some $K \in \N$. Equivalently, the quantity
\[ K = \dfrac{p_-}{p_+-p_-}\left(1+\dfrac{2 \eps_2}{2-2\eps_2}\right), \]
can be chosen to be any integer larger than $\frac{1}{\eps}$. It represents the amount of stabilizations used by the package.

Consider thus the front projection $\R^2$, endowed with coordinates $(q,z)$. There we consider the union of segments:
    \begin{equation*}
    \begin{split}
        &\bigcup_{i=0}^K \left\{\left(q, -p_-q + i \left(2-2 \eps_2\right) (p_+-p_-)\right)  \,|\, q \in \left[\eps_2,2 - \eps_2 \right] \right\}\\ \cup &\bigcup_{i=1}^K \left\{\left(q, -p_+q + (i \left(2-2 \eps_2\right) + \eps_2)(p_+-p_-)\right)  \,|\, q \in \left[\eps_2,2-\eps_2\right] \right\}\\
        \cup  &\left\{\left(q,-p_- q + K\left(2-2 \eps_2\right) (p_+-p_-)\right) \,|\,q \in \left[2-\eps_2,2\right] \right\}\\
        \cup &\left\{(q,-p_- q) \,|\, q \in [0,\eps_2] \right\},
    \end{split}
    \end{equation*}
which we view as an embedded curve with corners, see Figure \ref{fig:g_tilde}. Observe that its $z$-coordinate satisfies $|z| \leq  2p_-$ by construction.

We modify the corners to cusps so that the slope of the resulting front is still contained in the interval $[p_-,p_+]$ and the $z$-coordinate is contained in $\left[-2p_-,2p_-\right]$. The modification is performed in a small neighbourhood of the corners, at the expense of using a bit of the buffer. We then let the claimed $\widetilde{g}$ be a parametrization of the Legendrian lift of the resulting cuspidal front. We assume that the domain of $\widetilde{g}$ is $[0,2]$. Its left endpoint is $\widetilde{g} = (0,p_-,0)$ and its right endpoint is $\widetilde{g} = (2,p_-,0)$. 

\begin{figure}
    \begin{tikzpicture}
        \node[anchor=south west, inner sep=0] (fig) at (0,0)
            {\includegraphics[width=0.5\textwidth]%
           {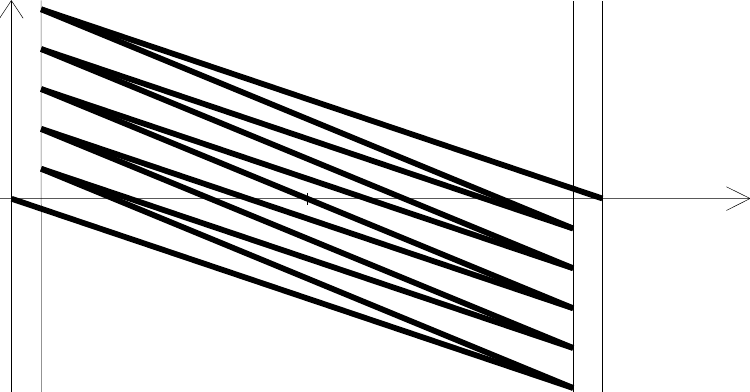}};
        \node at (0.0175\textwidth, - 0.007\textwidth) {\adjustbox{scale={0.53}}{$\underbrace{\hphantom{}}$}};
        \node at (0.019\textwidth, - 0.021\textwidth) {$\eps_2$};
        \node at (0.3925\textwidth, - 0.007\textwidth) {\adjustbox{scale={0.53}}{$\underbrace{\hphantom{}}$}};
        \node at (0.3925\textwidth, - 0.021\textwidth) {$\eps_2$};
        \node at (-0.001\textwidth, 0.117\textwidth) {$0$};
        \node at (-0.001\textwidth, 0.243\textwidth) {$z$};
        \node at (0.415\textwidth, 0.117\textwidth) {$2$};
        \node at (0.475\textwidth, 0.117\textwidth) {$q$};
    \end{tikzpicture}
    \caption{The front of the Legendrian $\widetilde{g}$ introduced in the proof of Lemma \ref{lem:ApproximatePositiveHamiltoniansOnRAux}.}
	\label{fig:g_tilde}
\end{figure}

\textbf{Step 2:} We now construct the claimed stabilized Legendrian $f_0$; it is built by putting together two stabilization packages. Its front is depicted in Figure \ref{fig:f1}.

Given $q_0 \in \R$, consider the contactomorphism that lifts the $q_0$-translation in the base $\R$:
    \begin{equation*}
        \begin{split}
            \tau_{q_0}: J^1\R &\to J^1\R,\\ (q,p,z) &\mapsto (q + q_0,p,z).
        \end{split}
    \end{equation*}
    Similarly, given $\lambda \in (0,1]$, consider the contactomorphism lifting the fibrewise scaling by $\lambda$:
    \begin{equation*}
        \begin{split}
            \eta_\lambda: J^1\R &\to J^1\R,\\ (q,p,z) &\mapsto (q,\lambda p,\lambda z).
        \end{split}
    \end{equation*}
We define Legendrian embeddings $g_\lambda^{q_0}(x) \coloneqq \tau_{q_0} \circ \eta_\lambda \circ \widetilde{g}(x-q_0)$; this means that the domain is $[q_0,q_0+2]$. For $q_0 \in \R$, we also define the Legendrian embeddings
    \begin{alignat*}{2}
        g^{q_0}_- :(-\infty,q_0] \quad&\to\quad J^1\R,\\
        x \quad&\mapsto\quad (x,0,0)\\ 
        g^{q_0}_+ :[q_0,\infty) \quad&\to\quad J^1\R.\\ 
        x \quad&\mapsto\quad (x,0,0)
    \end{alignat*}
    Let $q_0 + 2 = q_1$, and consider the front formed by concatenating $g_{\lambda_0}^{q_0}$ and $g_{\lambda_1}^{q_1}$. This new front has a corner at the concatenation point, which we can smooth out by a small modification in an arbitrarily small neighbourhood $V$. Moreover, we can assume that the slope over $V$ of the smoothed front lies in between 
    $\lambda_- p_-$ and $\lambda_+ p_-$ and that its $z$-coordinate, over $V$, lies in $\pi_z\big(g_{\lambda_0}^{q_0}(V\cap [q_0,q_1]) \cup g_{\lambda_1}^{q_1}(V \cap [q_1,q_1+2])\big)$.

    Let us spell this out. Fix a smooth function $\chi: [-1,1] \to [0,1]$ such that $\chi|_{[-1,-1/2]}=0$, $\chi(x)|_{[1/2,1]}=x$, and $\chi' \in [0,1]$. Then, the family of functions
    \[ \chi_{a,b,C}(x) := ax + (b-a)\cdot C \cdot \chi(x/C) \]
    satisfies the following properties:
    \begin{itemize}
        \item $\chi_{a,b,C}$ interpolates from slope $a$ to slope $b$ in the interval $[-C/2,C/2]$.
        \item $C^0$-converges to the lines of slope $a$ and $b$ as $C$ goes to zero.
        \item The image of $\chi_{a,b,C}|_{[-C/2,C/2]}$ is contained in $[-a \cdot C/2,b \cdot C/2]$.
        \item Over $[-C/2,C/2]$ we have bounds 
        \[ |\partial_a \chi_{a,b,C}| = |x - C \cdot \chi(x/C)| \leq 2C \quad\text{and}\quad|\partial_b \chi_{a,b,C}| = C \cdot\chi(x/C) \leq C. \]
    \end{itemize}
    Then, the union of $g_{\lambda_0}^{q_0}$ and $g_{\lambda_1}^{q_1}$ can be smoothed out by replacing a neighbourhood of the corner by a neighbourhood of $(q_1,0)$ with the image of $\tau_{q_1} \circ j^1\chi_{(\lambda_0 p_-,\lambda_1 p_-,C)}$. We can choose $C < \eps_2$ to guarantee that this happens within the buffer regions of the stabilization packages. We define $g_{\lambda_0}^{q_0} * g_{\lambda_1}^{q_1}$ to be a parametrization of the Legendrian lift of the resulting front. The parametrization agrees with the parametrizations of the pieces away from the smoothing region of the corner.
    
    The exact same construction can be carried out whenever we have two fronts meeting at a corner. For instance, we can define $g_-^{q_0}*g^{q_0}_\lambda$ and $g^{q_0}_\lambda*g_+^{q_1}$ by using the same function $\chi_{(a,b,C)}$ with either $a = 0$ or $b=0$.

    Concatenating and smoothing several fronts yields our starting stabilized Legendrian:
     \begin{equation*}
        f_0 := g_-^{-1}*g^{-1}_1*g^{1}_1*g_+^{3}.
    \end{equation*}

\begin{figure}
\centering
\begin{subfigure}{0.49\textwidth}
    \begin{tikzpicture}
        \node[anchor=south west, inner sep=0] (fig) at (0,0)
            {\includegraphics[width=\textwidth]%
           {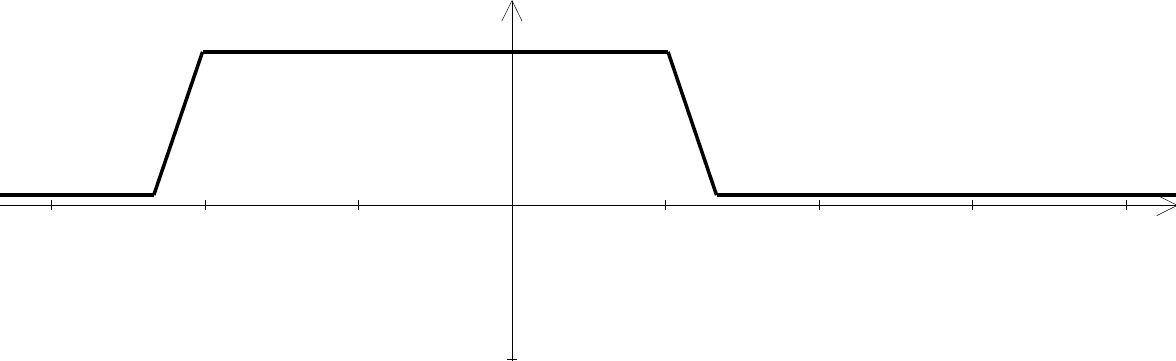}};
        \node at (0.412\textwidth, 0.29\textwidth) {$z$};
        \node at (0.98\textwidth, 0.1\textwidth) {$q$};
    \end{tikzpicture}
    \caption{The graph of $\zeta$.}
    \label{fig:zeta}
\end{subfigure}
\hfill
\begin{subfigure}{0.49\textwidth}
    \begin{tikzpicture}
        \node[anchor=south west, inner sep=0] (fig) at (0,0)
            {\includegraphics[width=\textwidth]%
           {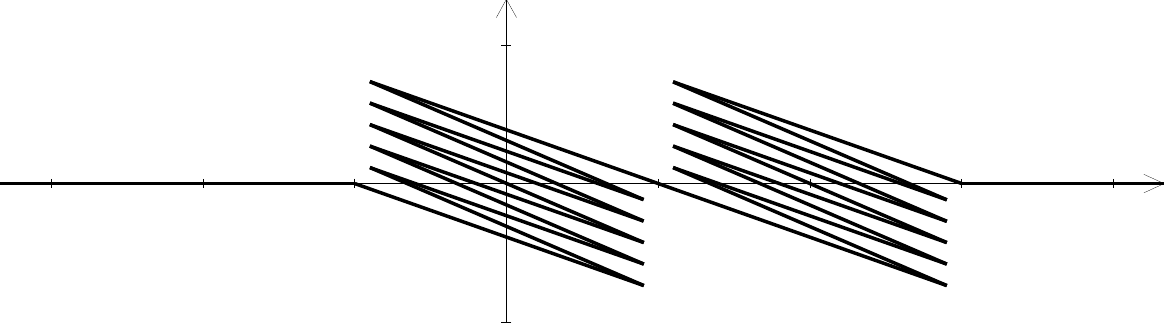}};
        \node at (0.412\textwidth, 0.26\textwidth) {$z$};
        \node at (0.97\textwidth, 0.088\textwidth) {$q$};
    \end{tikzpicture}
    \caption{The front of $f_t$ at time $t = 0$.}
    \label{fig:f1}
\end{subfigure}
\hfill
\begin{subfigure}{0.49\textwidth}
   \begin{tikzpicture}
        \node[anchor=south west, inner sep=0] (fig) at (0,0)
            {\includegraphics[width=\textwidth]%
           {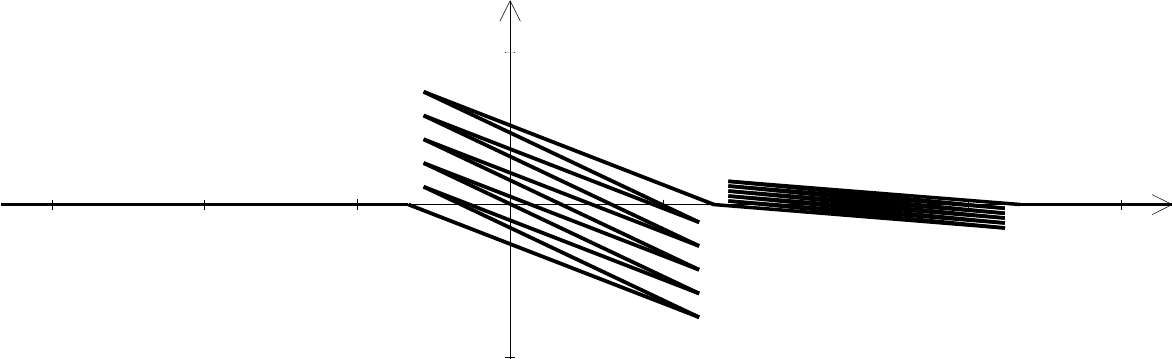}};
        \node at (0.412\textwidth, 0.29\textwidth) {$z$};
        \node at (0.98\textwidth, 0.1\textwidth) {$q$};
    \end{tikzpicture}
    \caption{The front of $f_t$ at time $t= \frac{p_-}{3}$.}
    \label{fig:f2}
\end{subfigure}
\hfill
\begin{subfigure}{0.49\textwidth}
    \begin{tikzpicture}
        \node[anchor=south west, inner sep=0] (fig) at (0,0)
            {\includegraphics[width=\textwidth]%
           {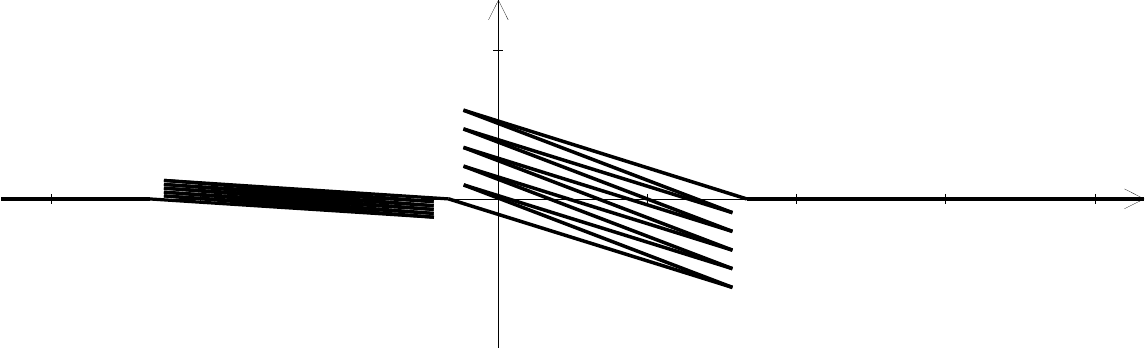}};
        \node at (0.412\textwidth, 0.29\textwidth) {$z$};
        \node at (0.98\textwidth, 0.1\textwidth) {$q$};
    \end{tikzpicture}
    \caption{The front of $f_t$ at time $t= \frac{2p_-}{3}$.}
    \label{fig:f3}
\end{subfigure}
        
\caption{Some of the elements appearing in the proof of Lemma \ref{lem:ApproximatePositiveHamiltoniansOnRAux}. The function $\zeta$ can be obtained by smoothly approximating the piecewise affine function shown in (a). Subfigures (b)-(d) show the front of $f_t$ (before smoothing) for different times of $t$. For $t \in [0,\frac{p_-}{3}]$, the front moves to the right and the right copy of $\widetilde{g}$ is compressed. For $t \in [\frac{p_-}{3},\frac{2p_-}{3}]$, the compressed copy of $\widetilde{g}$ moves along the uncompressed copy to the left, while the uncompressed copy moves to the right. For $t \in [\frac{2p_-}{3},2p_-]$, the front moves to the right while the left copy of $\widetilde{g}$ decompresses.}
\label{fig:twoLegendrianPackages}
\end{figure}

\textbf{Step 3a:} We now define the Legendrian isotopy $f_t$, starting at the given $f_0$. It will be $2 p_-$-periodic in time and will satisfy the required properties with $N = 2K$. We build it in stages, corresponding to different time intervals. As we go along, we also show that the family satisfies the properties claimed in the statement.

Consider $t \in \left[0,\frac{p_-}{3}\right]$. During this time interval, the isotopy will compress the second stabilization package. To do so, what we must do is decrease its slope. This decrease is governed by a constant $\eps_3 > 0$. Moreover, in order to interpolate between the two slopes we use a function $\zeta: \R \to (0,1]$ such that
    \begin{itemize}
        \item[(i)] $\zeta(q) = 1$ for $q \in \left[-2-\eps_3,1+\eps_3\right]$,

        \item[(ii)] $\{(q,p,z) \in J^1\R \,\mid\, |p|, |z| < \max\limits_{q'\in[q-2,q]}2\zeta(q') \delta\} \subseteq W$,

        \item[(iii)] $|\zeta'(q)| < 4$ for all $q \in \R$,

        \item[(iv)] $\zeta(q) < \eps_3$ for $q \in \R \setminus \left(-\frac{7}{3} + \eps_3, \frac{4}{3} - \eps_3 \right)$,
    \end{itemize}
see Figure \ref{fig:zeta}.

Introduce then the quantity $s_t \coloneqq \frac{t}{p_-}$. Its role is twofold: First, it will be used to push the packages to the right by $1/3$, which we will use to approximate the Hamiltonian $1$. Secondly, and crucially for this step, it will compress the height of the second package. Concretely, we set:
    \begin{equation*}
        f_t(x) \coloneqq g_-^{-1+s_t}*g^{-1+s_t}_1*g^{1+s_t}_{\zeta(1+s_t)}*g_+^{3+s_t}(x+s_t).
    \end{equation*}
See Figures \ref{fig:f1} and \ref{fig:f2}.

We now measure the cost of this process. Let $V \subset \R$ be a small open neighbourhood of the concatenation points of $f_t$; note that this does not depend on $t$. We claim that the Hamiltonian generating $f_t$, at points $x \notin V$, reads:
    \begin{equation*}
        \alpha_\std(\partial_t f_t)(x) = \begin{cases} 
            0 &\textnormal{if } x \in (-\infty, -1] \\
            \dfrac{\pi_p\circ f_t(x)}{p_-}  &\textnormal{if } x \in [-1, 1] \\
            \dfrac{\zeta'(1+s_t)}{p_-} \pi_z \circ g^{1+s_t}_{\zeta(1+s_t)}(x+s_t) + \dfrac{\pi_p \circ f_t(x)}{p_-}  &\textnormal{if } x \in [1,3]\\
            0 &\textnormal{if } x \in [3,\infty) \\
        \end{cases}
    \end{equation*}
    Indeed, the cost associated to the leftmost package is due it being translated at speed $1/p_-$, in the direction of $q$. An analogous term appears for the rightmost package, plus an additional term corresponding to the squeezing of the height. For $x \notin V$, it follows that 
    \begin{equation*}
        |\alpha_\std(\partial_t f_t)(x)| \leq \frac{p_+}{p_-} + \frac{4}{p_-} 2 p_- < 1+\eps +8 < 10,
    \end{equation*}
    as long as we assume $\eps < 1$. Here we used that $0 \leq \pi_p(f_t(x)) \leq p_+$, $\frac{p_+}{p_-} < 1 +\eps$, $|\zeta'| < 4$, and $|\pi_z(g^{1+s_t}_{\zeta(1+s_t)}(x+s_t))|\leq 2 p_-$.

    Suppose now that $x \in V$. Over each component of $V$, we have that $f_t = \tau_{s_t} \circ j^1\chi_{(a_t,b_t,C)}$,  for some functions $a_t$, and $b_t$ satisfying:
    \[ |a_t|, |b_t| \leq p_- \quad\text{and}\quad |\partial_t a_t|, |\partial_t b_t| \leq \max(|\zeta'|)\frac{1}{p_-}. \]
    This, together with $\partial_t s_t = 1/p_-$ and assuming that $C$ is sufficiently small, implies that
    \begin{equation*}
        |\alpha_\std(\partial_t f_t)| \leq \frac{p_-}{p_-} + 3 \max(|\zeta'|)\frac{1}{p_-} C < 2,
    \end{equation*}
    where we used the fourth property of $\chi_{a,b,C}$.

    Lastly, we observe that stronger estimates hold over $[-1,1] \setminus V$. Indeed, the computations above particularise to:
    \begin{equation*}
        \alpha_\std(\partial_t f_t)(x) \in \frac{1}{p_-}[p_-,p_+] = [1,1+\eps].
    \end{equation*}
    Note that this is also true near $1 \in [-1,1]$ as long as $s_t < \eps_3$, since then $\zeta(1 + s_t) = 1$ and the union of the fronts $g_1^{-1+s_t}$ and $g^{1+s_t}_{\zeta(1+s_t)}$ is already smooth and equals $g_1^{-1+s_t}*g^{1+s_t}_{\zeta(1+s_t)}$.

\textbf{Step 3b:} We now define $f_t$ for $t \in \left[\frac{p_-}{3},\frac{2p_-}{3}\right]$. During this step, the already-compressed second package gets transported backwards in order to become the ``new first package''.

Observe first that $f_{\frac{p_-}{3}}$ is given by the expression
    \begin{equation*}
        f_{\frac{p_-}{3}}(x) = g_-^{-2/3}*g^{-2/3}_1*g^{4/3}_{\eps_3}*g_+^{10/3}(x+1/3).
    \end{equation*}
In particular, we can think of it as being a stabilization of 
    \begin{equation*}
        h := g_-^{-2/3}*g^{-2/3}_1*g_+^{4/3}.
    \end{equation*}

Let $\widetilde W \subseteq W$ be a strict Weinstein neighbourhood of $h$, of height $\eps_4$ for some $\eps_4>0$. If $\eps_3$ is small enough, then the image of $f_{\frac{p_-}{3}}$ is contained in $\widetilde W$. Moreover, there exists an isotopy $\widetilde f_t$ with $t \in \left[\frac{p_-}{3},\frac{2p_-}{3}\right]$ of Legendrian embeddings from $f_{\frac{p_-}{3}}$ to 
\[ \widetilde f_{\frac{2p_-}{3}}(x) = g_-^{-8/3}*
    g^{-8/{3}}_{\eps_3}*
    g^{-2/{3}}_1*
    g_+^{4/{3}}(x-5/3) \]
that is supported on $\widetilde W$. By Corollary \ref{cor:boudnednessOfStrictWeinsteinNbhdOfHeightC}, we may assume that $|\alpha_\std (\partial_t \widetilde f_t)| < \eps_4$.

We have thus transported the second package backwards, but we still need to introduce some translating to approximate $H=1$. By taking $\widetilde W$ sufficiently small, we can assume that the $p$-coordinate of $\widetilde f_t$ is contained in $[0,p_++\eps_4]$ and, over $q \in [0,1]$, it is contained in $[p_--\eps_4,p_+ +\eps_4]$. We can thus set: 
    \begin{equation*}
        f_t \coloneqq \tau_{\frac{t}{p_-}-\frac{1}{3}} \circ \widetilde f_t
    \end{equation*}
    for $t \in \left[\frac{p_-}{3},\frac{2p_-}{3}\right]$; see Figures \ref{fig:f2} and \ref{fig:f3}.

    Since $\tau_q$ is a strict contactomorphism, it follows that
    \begin{equation*}
        \alpha_\std(\partial_tf_t) = \frac{\pi_p(f_t)}{p_-} + \alpha_\std\left(\partial_t \widetilde{f}_t\right),
    \end{equation*}
    whose absolute value is bounded by $\frac{p_++\eps_4}{p_-}+\eps_4 < 2$ if $\eps_4$ is small enough. If $\pi_q(f_t(x)) \in [0,1]$, then we find that
    \begin{equation*}
        \alpha_\std(\partial_tf_t(x)) \in \left[\frac{p_--\eps_4}{p_-}-\eps_4,\frac{p_++\eps_4}{p_-}+\eps_4\right] \subseteq[1-\eps,1+\eps]
    \end{equation*}
    if $\eps_4$ is small enough.

\textbf{Step 3c.} During the interval $t \in \left[\frac{2p_-}{3},2p_-\right]$ we proceed as in Step 3a. The original second package, which now is on the left, and is still compressed, gets decompressed so that $f_{2p_-} = f_0$; see Figures \ref{fig:f3} and \ref{fig:f1}. The same bounds as in Step 3a hold. We have thus constructed a $2 p_-$-periodic family of Legendrian embeddings. By restricting to $t \in [0,1]$, we find the family claimed in the statement.
\end{proof}

And now the general case:
\begin{lemma}\label{lem:ApproximatePositiveHamiltoniansOnR}
   Given any $C>0$ and $\eps > 0$, there exists a natural number $N = N(C,\eps) \in \N$ with the following property:
   
   Fix an open interval $I \subseteq \R$, a family of functions $H_t:J^1I \to [0,C)$, $t \in [0,1]$, with uniformly compact support, and a neighbourhood $W$ of the zero section in $J^1I$. Then, there exists a compactly supported Legendrian isotopy $f_t:I \to W$ such that $\alpha_\std(\partial_t f_t)$ is $\eps$-close to $H_t \circ f_t$ for all $t \in [0,1]$.
   
   Moreover, each of the Legendrians $f_t$ is an $N$-fold stabilization of the zero section inside of $W$. All stabilizations are of the same sign.
\end{lemma}
\begin{proof}
    After a reparametrization, we may assume that $I = \R$ and that $H_t(q,p,z) = 0$ for $q \not\in [0.1,0.9]$ for all $t \in [0,1]$. We then argue as in the proof of Lemma \ref{lem:ApproximatePositiveHamiltonians}: by a small perturbation we can work with a family $G_t$ of positive functions. Positivity then allows us to change the contact form in order to deal solely with ``normalized'' Hamiltonians.
    
    Let $G_t:\R \to (0,\infty)$ be a positive smooth function such that $G_t(q) = \eps_1$ for all $q \in \Op(\R \setminus [0,1])$ and $|G_t(q) - H_t(q,0,0)| \leq \eps_1$ for all $q \in \R$ for all $t \in [0,1]$ and some small positive number $\eps_1 > 0$. If $\eps_1$ is sufficiently small, then after possibly shrinking $W$ so that the variation of $H_t$ in the fiber direction is small, it will be enough to construct $f_t$ so that $\alpha_\std(\partial_t f_t)$ is $\eps$-close to $G_t \circ \pi_q \circ f_t$ after possibly replacing $\eps$ by a smaller number. Here, $\pi_q : J^1 \R \to \R$ denotes the projection onto the zero section.

    Choose a time-dependent strict Weinstein neighbourhood $W_t \subseteq W$ of the zero-section $\R$ with respect to the time-dependent contact form $\alpha_t = \frac{\alpha_\std}{G_t}$ strictly contactomorphic via $\phi_t: W_t \to \widetilde{W}$ to a neighbourhood $\widetilde{W}$ of the zero section in $(J^1S^1,\alpha_\std)$. According to Proposition \ref{thm:StrictWeinsteinNeighbrouhoods} we may choose $(\phi_t,W_t)$, outside of $(0,1)$, to be equal to 
    \begin{equation*}
    \begin{aligned}
        \phi_t:W_t \cap \pi_q^{-1}(\R \setminus (0,1)) = W \cap \pi_q^{-1}(\R\setminus (0,1)) &\to \Bigg\{\left(q, \frac{p}{\eps_1}, \frac{z}{\eps_1}\right)|\, (q,p,z) \in W, q \in \R \setminus (0,1)\Bigg\} \,.\\
        (q,p,z) &\mapsto \bigg(q,\frac{p}{\eps_1},\frac{z}{\eps_1}\bigg)
    \end{aligned}
    \end{equation*}
    Suppose $f_t$ is a Legendrian isotopy in $W_t$ such that $\alpha_t(\partial_t f_t(x))$ is $\beta(\pi_q(f_t)(x))$-close to 1 for some function $\beta: \R \to (0,\infty)$ and all $x \in \R$. Then $\alpha_\std(\partial_t f_t)(x)$ is $((\beta \cdot G_t) \circ \pi_q \circ f_t)(x)$-close to $(G_t \circ \pi_q \circ f_t)(x)$. Let $\beta$ be such that $\beta(q) = \eps_2$ for $q \in [0,1]$  and some $\eps_2 > 0$, and $\beta \leq 10$ everywhere. Such an $f_t$ exists for any choice of $\eps_2$ according to Lemma \ref{lem:ApproximatePositiveHamiltoniansOnRAux}: if $\widetilde f_t$ denotes the family of Legendrian embeddings in Lemma \ref{lem:ApproximatePositiveHamiltoniansOnRAux} applied to some $\eps_3 < \eps_2$, then we can define $f_t = \phi_t^{-1} \circ \widetilde f_t$. As in the proof of Lemma \ref{lem:ApproximatePositiveHamiltonians}, one computes that $f_t$ has the required properties after possibly shrinking $\widetilde W$.
    
    
    The proof is complete once we set $\eps_1 < \eps/10$ and $\eps_2 := \frac{\eps}{C+\eps_1}$ since $|G_t| \leq C + \eps_1$.
\end{proof}

\begin{remark}\label{rmk:ApproxForBothSingsAndStabTypes}
    These lemmata also hold for non-positive Hamiltonians (by reversing the time parameter in the construction) and for both types of stabilizations (by reversing both the $q$-coordinate and the time coordinate) after fixing an orientation.
\end{remark}

\begin{remark}
    The number $N$ in Lemma \ref{lem:ApproximatePositiveHamiltonians} can be taken to be any integer larger than $\frac{C}{\eps}$. The number $N$ in Lemma \ref{lem:ApproximatePositiveHamiltoniansOnR} can be taken to be any integer larger than $2\frac{C}{\eps}$.
\end{remark}

Suppose we are given a positive isotopy of Legendrians. We can use the previous statements to construct a second isotopy (locally defined in a moving strict Weinstein neighbourhood) so that their costs cancel upon composition. This allows us to produce low cost isotopies:
\begin{lemma}\label{lem:SmallEnergyIsotopiesOfNon-negLegWithManyStabilizations}
    Let $C>0$ and $\eps>0$. There exists a natural number $N=N(C,\eps)$ such that the following statement holds.

    Let $(M,\alpha)$ be a $3$-dimensional strict contact manifold. Fix a compactly supported family of proper and connected Legendrian embeddings $f_t:L \to M$, $t \in [0,1]$, satisfying 
    \[ 0 \leq \alpha(\partial_t f_t) < C. \]
    Let $\psi_t:W \to M$ be a strict Weinstein neighbourhood of $f_t$, of height less than $\eps' > 0$ and parametric in $t$. Let $f^N_0$ and $f^N_1$ be $N$-fold stabilizations of $f_0$ and $f_1$ inside of $\psi_0(W)$ and $\psi_1(W)$. We assume that all stabilizations are of the same type (they can be chosen to be either positive or negative after fixing an orientation of $L$). Then, we have that:
    \begin{equation*}
        d_\alpha(f^N_0,f^N_1) < \eps + 2 \eps'.
    \end{equation*}
\end{lemma}
\begin{proof}
Let $H_t:M \to \R$ be such that $\phi^H_t \circ \psi_0 = \psi_t$. After possibly shrinking $W$, we may assume that $|H_t| < C$. In particular, the conformal factor $\chi^H_t$ of $H_t$ is equal to 1 along $\psi_0(W)$ for all $t \in [0,1]$ by strictness of $\psi_t$. By Lemma \ref{lem:ApproximatePositiveHamiltonians} (if $L = S^1$) or Lemma \ref{lem:ApproximatePositiveHamiltoniansOnR} (if $L=\R$), and by Remark \ref{rmk:ApproxForBothSingsAndStabTypes}, there exists a contact Hamiltonian $G_t:M \to \R$ with compact support in $\psi_0(W)$ and a Legendrian embedding $g:L \to \psi_0(W)$ Legendrian isotopic to $f_0^N$ such that 
    \begin{equation*}
        |G_t \circ \phi^G_t \circ g + H_t \circ \phi^H_t \circ \phi^G_t \circ g| < \eps.
    \end{equation*}
    Indeed, since $ H_t \circ \phi^H_t \circ f_0 = \alpha(\partial_t f_t) \geq 0$ is non-negative, we can apply Lemma \ref{lem:ApproximatePositiveHamiltonians} or Lemma \ref{lem:ApproximatePositiveHamiltoniansOnR} to an auxiliary Hamiltonian $F_t \geq 0$ whose restriction $F_t|_W$ can be assumed to be arbitrarily close to $H_t \circ \phi^H_t|_W$ after possibly shrinking $W$. 
    
    The contact isotopy $\phi^H_t \circ \phi^G_t$ is generated by the Hamiltonian 
    \begin{equation*}
        H_t + \left(\chi^H_t \cdot G_t\right) \circ \left(\phi^H_t\right)^{-1} = H_t + G_t \circ (\phi^H_t)^{-1}.
    \end{equation*}
    In particular, we have that 
    \begin{equation*}
        |\alpha\left(\partial_t\left(\phi^H_t \circ \phi^G_t \circ g\right)\right)| = |H_t \circ \phi^H_t \circ \phi^G_t \circ g + G_t \circ \phi^G_t \circ g| < \eps,
    \end{equation*}
    which implies that 
    \begin{equation*}
        d_\alpha\left(g, \phi^H_1 \circ \phi^G_1 \circ g\right) < \eps.
    \end{equation*}
    By Corollary \ref{cor:boudnednessOfStrictWeinsteinNbhdOfHeightC}, we also have that 
    \begin{equation*}
        d_\alpha\left(f^N_0,g\right), d_\alpha\left(\phi^H_1 \circ \phi^G_1 \circ g, f^N_1\right) < \eps'.
    \end{equation*}
    This finishes the proof by the triangle inequality. 
\end{proof}

\begin{remark}\label{rmk:SmallEnergyIsotopiesOfLegWithManyStabilizationsDisconnected}
    Note that the constructions in the proof of Lemma \ref{lem:SmallEnergyIsotopiesOfNon-negLegWithManyStabilizations} are local around the Legendrian embeddings. This implies that  Lemma \ref{lem:SmallEnergyIsotopiesOfNon-negLegWithManyStabilizations} also holds for proper Legendrian embeddings with possibly disconnected domains if the Weinstein neighbourhoods of the components are taken to be mutually disjoint. In this case, every component of $f$ needs to be stabilized $N$ times. Recall that an $N$-fold stabilization of a disconnected Legendrian is precisely defined in this way. In addition, the type of the stabilization can be chosen for each component of $L$ after fixing an orientation of $L$.
\end{remark}

\subsection{The general case} \label{ssec:generalIsotopies}

The results we have obtained for non-negative isotopies can now be combined with the discretization approach of Subsection \ref{sec:positiveNegative} in order to deal with arbitrary isotopies.

\begin{corollary}\label{cor:SmallEnergyIsotopiesOfLegWithManyStabilizations}
    Fix $C> 0$ and $\eps > 0$. Let $(M,\alpha)$ be a strict contact manifold of dimension 3 and $U \subseteq M$ be a precompact open subset. There exists a natural number $N_0 = N_0(C,\eps,U)$ with the following property.
    
    Fix a family of proper Legendrian embeddings $f_t:L \to M$, $t \in [0,1]$, with compact support in $U$ satisfying
    \[ |\alpha(\partial_t f_t)| < C  .
    \]
    Then, for any $N \geq N_0$ we have that
    \begin{equation*}
        d_\alpha(f_0^N,f_1^N) < \eps.
    \end{equation*}
    Here, $f_i^N$ denotes an $N$-fold stabilization of $f_i$ inside of a sufficiently small strict Weinstein neighbourhood. All stabilizations are of the same type over each component of $L$.
\end{corollary}
\begin{proof}
    In the case that $\alpha(\partial_t f_t)$ is either non-negative or non-positive, Corollary \ref{cor:SmallEnergyIsotopiesOfLegWithManyStabilizations} follows from Lemma \ref{lem:SmallEnergyIsotopiesOfNon-negLegWithManyStabilizations} and Remark \ref{rmk:SmallEnergyIsotopiesOfLegWithManyStabilizationsDisconnected}. We use Lemma \ref{lem:ApproxByPWPosNegHamiltonian} to reduce to this special case.

    By Lemma \ref{lem:ApproxByPWPosNegHamiltonian} applied to a generating Hamiltonian $H_t:M \to (-C,C)$ with $t \in [0,1]$ and support in $U$, there exists a Hamiltonian $G_t$ on $M$ and a natural number $n = n(C,U) \in \N$ such that 
    \begin{itemize}
    \item Over each $[i/n,(i+1)/n] \subset [0,1]$, $G_t \circ \phi^G_t$ is smooth and non-negative or non-positive.
    \item $\phi^G_{2i/n} = \phi^H_{2i/n}$ for all $i$,
    \item $|G| < 2C$.
\end{itemize}
    We apply Lemma \ref{lem:SmallEnergyIsotopiesOfNon-negLegWithManyStabilizations} in the case of a non-negative/non-positive Hamiltonian to each interval $[i/n,(i+1)/n]$ to obtain a number $N_0 = N(2C,\frac{\eps}{n})\in \N$ such that 
    \begin{equation*}
        d_\alpha(\phi^G_{i/n} \circ f_0^{N_0},\phi^G_{(i+1)/n}\circ f_0^{N_0}) < \frac{\eps}{n},
    \end{equation*}
    where we applied Lemma \ref{lem:SmallEnergyIsotopiesOfNon-negLegWithManyStabilizations} to a slightly smaller $\eps$ and a sufficiently small $\eps'$.
    By the triangle inequality, this implies
    \begin{equation*}
        d_\alpha(f_0^{N_0},\phi^G_1 \circ f_0^{N_0}) < \eps.
    \end{equation*}
    We may choose $f_0^{N_0}$ such that $\phi^G_1 \circ f_0^{N_0}$ is an ${N_0}$-fold stabilization of $f_1$ inside of a strict Weinstein neighbourhood of $f_1$ of height less than $\eps_1$ for some $\eps_1 > 0$. If $f_1^{N_0}$ is another such ${N_0}$-fold stabilization, then 
    \begin{equation*}
        d_\alpha(\phi^G_1 \circ f_0^{N_0},f_1^{N_0}) < \eps_1
    \end{equation*} 
    by Corollary \ref{cor:boudnednessOfStrictWeinsteinNbhdOfHeightC}. As $\eps,\eps_1 > 0$ were arbitrary, this finishes the proof for $N = N_0$ by another application of the triangle inequality. The general case $N > N_0$ follows by simply applying the case $N = N_0$ to an $(N - N_0)$-fold stabilization of the Legendrian isotopy $f_t$ inside of a sufficiently small Weinstein neighbourhood.
\end{proof}

\begin{remark}
    The number $N = N(C,\eps,U)$, in fact, only depends on $C$, $\eps$, and a number $s_0 >0$ such that the Reeb flow of $\alpha$ applied to $U$ exists for times in $[0,s_0]$. If all components of $L$ are closed, we can choose $N$ as any natural number larger than $\frac{4C}{\eps}\left\lceil \frac{C}{2s_0}\right\rceil$. Otherwise we can choose $N$ as any natural number larger than $\frac{8C}{\eps}\left\lceil \frac{C}{2s_0}\right\rceil$. In particular, $N$ does not depend on $U$ if the Reeb flow of $(M,\alpha)$ is complete.
\end{remark}

\begin{remark}\label{rmk:SmallEnergyIsotopiesOfStabLegFixedIsotopyClass}
   We can strengthen Corollary \ref{cor:SmallEnergyIsotopiesOfLegWithManyStabilizations} in the following way: there exists a contact isotopy from $f_0^N$ to $f_1^N$ of energy less than $\eps$ which is isotopic with fixed endpoints to $f_t^N$, where $f_t^N$ denotes a Legendrian isotopy which for fixed $t$ is an $N$-fold stabilization of $f_t$ inside of a time-dependent Weinstein neighbourhood of $f_t$. Indeed, it is enough to prove the corresponding generalization of Lemma \ref{lem:SmallEnergyIsotopiesOfNon-negLegWithManyStabilizations}. For this, note that in Corollary \ref{cor:boudnednessOfStrictWeinsteinNbhdOfHeightC}, we may also prescribe the isotopy class of the contact isotopies, see Definition \ref{def:boundedCost}.
\end{remark}


\section{The overtwisted size controls the size of all Legendrian isotopies} \label{sec:energyBoundsLegendrianIsotopies} 

\subsection{The case of a single overtwisted disk}

We now establish Theorem \ref{thm:BoundednessInComplementOfOtDiskOfSizeC}. It states that the overtwisted size provides universal control over all Legendrian isotopies in the complement.

\begin{proof}[Proof of Theorem \ref{thm:BoundednessInComplementOfOtDiskOfSizeC}]
Let $\Delta \subseteq (M,\xi)$ be a hyperdisk of overtwisted size $C$ and let $\eps > 0$. Let $f_0,f_1:L \to M \setminus \Delta$ be two proper Legendrian embeddings that are formally Legendrian isotopic and agree outside of a compact set; this isotopy can be assumed to be disjoint from $\Delta$, since the latter is contractible. Let $V$ be a sufficiently small open neighbourhood of $\Delta$, disjoint from $f_0$ and $f_1$, such that any properly embedded Legendrian in $V \setminus \Delta$ has stabilization cost less than $\frac{C}{2} + \eps$.

Let $(f_t,F_{s,t})$ be a formal Legendrian isotopy from $f_0$ to $f_1$. We may assume by Lemma \ref{lem:FormalLegendrianStandardPosition} that this formal isotopy is a concatenation of the form $(f_t,F_{s,t}) =g^0_t*(g_t,G_{s,t}) *g^1_t$ such that $g^0_t$ and $g^1_t$ are Legendrian isotopies of energy $< \eps$, both $g_t^{-1}(V)$ and $(g_t,G_{s,t})|_{g_t^{-1}(V)}$ are independent of $t$, and every component of $L$ intersects $g_t^{-1}(V)$. In particular, the restriction of $(g_t,G_{s,t})$ to $g_t^{-1}(V)$ is Legendrian, and we may also assume that the restriction of $(g_t,G_{s,t})$ to $L \setminus \overline{g_t^{-1}(V)}$ is a compactly supported formal Legendrian isotopy which is Legendrian outside of a compact set. 

Assume first that $n \geq 2$. Then $g_0$ has loose size $< C + 2\eps$ by Lemma \ref{lem:Nakamura}. Since $d_\alpha(g_0,f_0) < \eps$, it follows that $f_0$ has loose size $< C + 4 \eps$. As $\eps > 0$ was arbitrary, $f_0$ has loose size at most $C$. Similarly, $f_0$ has loose size at most $C$. Theorem \ref{thm:BoundednessInComplementOfOtDiskOfSizeC} then follows from Theorem \ref{thm:LooseLegQuantitativeFlexibility}.

Now we assume that $n=1$. By the h-principle for Legendrian links (Theorem \ref{thm:hprincipleLegLinks}), there exists a natural number $N > 0$ such that an $(N,N)$-fold stabilization of $g_0|_{L \setminus \overline{g_t^{-1}(V)}}$ is Legendrian isotopic to an $(N,N)$-fold stabilization of $g_1|_{L \setminus \overline{g_t^{-1}(V)}}$, both viewed as proper Legendrian embeddings into $M \setminus \overline{V}$. By Corollary \ref{cor:SmallEnergyIsotopiesOfLegWithManyStabilizations}, and after possibly increasing $N$, we may also assume that the energy of such an isotopy is smaller than $\eps$, provided that the stabilizations are contained inside of a sufficiently small Weinstein neighbourhood.

Let $\phi_t$ be a contact isotopy with support in $V$ such that $\Vert \phi_t \Vert_\alpha < \frac{C}{2}+\eps$ and $\phi_1 \circ g_t|_{g_t^{-1}(V)}$ is $(N,N)$-fold stabilized inside of some strict Weinstein neighbourhood of height $< \eps$. By Corollary \ref{cor:boudnednessOfStrictWeinsteinNbhdOfHeightC}, there exist Legendrian isotopies, of energy less than $\eps$, from $\phi_1\circ g_0$ to a Legendrian embedding $h_0$ and from a Legendrian embedding $h_1$ to $\phi_1\circ g_1$, such that the part of $h_i$ that has image in $M \setminus \overline{V}$ is an $(N,N)$-fold stabilization of $g_i|_{L \setminus \overline{g_t^{-1}(V)}}$ inside of an arbitrarily small strict Weinstein neighbourhood, and $h_0$ and $h_1$ coincide on the part that has image in $\overline{V}$. It follows from our choice of $N$ that $d_\alpha(h_0,h_1) < \eps$. Combining all estimates, we get
\begin{equation*}
\begin{split}
    d_\alpha(f_0,f_1) &\leq d_\alpha(f_0,g_0)+ d_\alpha(g_0,\phi_1 \circ g_0) + d_\alpha(\phi_1 \circ g_0,h_0) + d_\alpha(h_0,h_1)\\
    & \quad + d_\alpha(h_1, \phi_1 \circ g_1) + d_\alpha(\phi_1 \circ g_1,g_1) + d_\alpha(g_1,f_1) \\
    & < \eps + \left(\frac{C}{2} + \eps\right) + \eps + \eps + \eps + \left(\frac{C}{2} + \eps\right) + \eps = C + 7 \eps.
\end{split}
\end{equation*}
As $\eps> 0$ was arbitrary, this finishes the proof of Theorem \ref{thm:BoundednessInComplementOfOtDiskOfSizeC}. 
\end{proof}

\begin{remark}\label{rmk:SmallEnergyIsotopiesWithFixedHmtpClass}
In the setting of Theorem \ref{thm:BoundednessInComplementOfOtDiskOfSizeC}, there exists for any $\eps > 0$ a contact isotopy of energy at most $C + \eps$ in any fixed homotopy class of (formal) Legendrian isotopies from $f_0$ to $f_1$. Indeed, in dimension 3, after possibly shrinking the neighbourhood $V$ in the proof of Theorem \ref{thm:BoundednessInComplementOfOtDiskOfSizeC}, we may assume that there exists a slightly larger neighbourhood $\widetilde{V}$ containing $V$ and a contact isotopy $\psi_t$ with support in $\widetilde{V}$ such that $h_i = \psi_1 \circ \phi_1 \circ g_i$. It follows that the formal Legendrian isotopy $(f_t,F_{s,t})$ is isotopic with fixed end-points to $(\psi_{t} \circ \phi_{t} \circ g^0_t)*(\psi_1 \circ \phi_1 \circ (g_t,G_{s,t})) *(\psi_{1-t} \circ \phi_{1-t} \circ g^1_t)$. By the h-principle for Legendrian links in the complement of an overtwisted disk (Theorem \ref{thm:h-principleLooseLegDim3}) and Remark \ref{rmk:SmallEnergyIsotopiesOfStabLegFixedIsotopyClass}, the isotopy from $h_0 = \psi_1 \circ \phi_1 \circ g_0$ to $h_1 = \psi_1 \circ \phi_1 \circ g_1$ of energy at most $\eps$ constructed in the above proof can be chosen to lie in any prescribed (formal) Legendrian isotopy class. 

    In higher dimensions, we may assume by the definition of the overtwisted size that there exists a contact isotopy $\phi_t$ of energy less than $\frac{C}{2} + \eps$ such that $\phi_1 \circ g_0$ and $\phi_1 \circ g_1$ have loose sizes smaller than $\eps$ in the complement of $\Delta$. By Theorem \ref{thm:LooseLegQuantitativeFlexibility}, 
    \begin{equation*}
        d_\alpha(\phi_1 \circ g_0, \phi_1 \circ g_1) < \eps.
    \end{equation*}
    Furthermore, the Legendrian isotopy from $\phi_1 \circ g_0$ to $\phi_1 \circ g_1$ that is constructed using the arguments in \cite{Nak24} is in the formal isotopy class of the original (formal) isotopy. Since $\phi_1 \circ g_0$ and $\phi_1 \circ g_0$ have image in the complement of $\Delta$, the constructed Legendrian isotopy class can be chosen to lie in a prescribed Legendrian isotopy class by the h-principle for Legendrians in the complement of an overtwisted disk (Theorem \ref{thm:h-principleLooseLegInComplementOfOTDisk}).
\end{remark}

\subsection{The case of two overtwisted disks}

Now we proceed to Theorem \ref{thm:BoundMovingOtDisk}. It is the generalisation of Theorem \ref{thm:BoundednessInComplementOfOtDiskOfSizeC} where we consider two overtwisted disks of different costs (and, for each, a Legendrian in its complement).

\begin{proof}[Proof of Theorem \ref{thm:BoundMovingOtDisk}]
In the case that $n \geq 2$, Theorem \ref{thm:BoundMovingOtDisk} follows from Lemma \ref{lem:Nakamura} and Theorem \ref{thm:LooseLegQuantitativeFlexibility}. Let us now assume that $n = 1$.

For $i \in \{0,1\}$, let $\Delta_i \subseteq (M,\xi)$ be a hyperdisk of overtwisted size $C_i$ and let $f_i:L \to M \setminus \Delta_i$ be a proper Legendrian embedding such that $f_0$ and $f_1$ are formally Legendrian isotopic and agree outside of a compact set and there exists a compactly supported contact isotopy $\psi_t$ satisfying $\psi_1(\Delta_0) = \Delta_1$. Fix $\eps >0$. Let $V$ be a sufficiently small open neighbourhood of $\Delta_0$ disjoint from the image of $f_0$ such that any properly embedded Legendrian $f$ in $V \setminus \Delta_0$ has stabilization cost $\mathcal C(f;V,V\setminus \Delta_0) < \frac{C_0}{2} + \eps$ and any properly embedded Legendrian $f$ in $\psi_1(V) \setminus \Delta_1$ has stabilization cost $\mathcal C(f;\psi_1(V),\psi_1(V) \setminus \Delta_1) < \frac{C_1}{2} + \eps$.

Let $(\widetilde f_t,F_{s,t})$ be a formal Legendrian isotopy from $f_0$ to $\psi_1^{-1}\circ f_1$; we can assume it is disjoint $\Delta_0$. We may also assume (Lemma \ref{lem:FormalLegendrianStandardPosition}) that this formal isotopy is a concatenation of the form $(\widetilde f_t,F_{s,t}) =g^0_t*(g_t,G_{s,t}) *g^1_t$ such that $g^0_t$ and $g^1_t$ are Legendrian isotopies of energy less than $\eps$, both $g_t^{-1}(V)$ and $(g_t,G_{s,t})|_{g_t^{-1}(V)}$ are independent of $t$, and every component of $L$ intersects $g_t^{-1}(V)$. In particular, the restriction of $(g_t,G_{s,t})$ to $g_t^{-1}(V)$ is Legendrian, and we may also assume that the restriction of $(g_t,G_{s,t})$ to $L \setminus \overline{g_t^{-1}(V)}$ is a compactly supported formal Legendrian isotopy which is Legendrian outside of a compact set. By applying Lemma \ref{lem:FormalLegendrianStandardPosition} to a smaller $\eps$, we may furthermore assume that also the energy of $\psi_1 \circ g^1_t$ is smaller than $\eps$.

Let $U \subseteq M$ be a connected, open, and pre-compact subset containing the support of $\psi_t$, $\overline V$, and the support of all of the (formal) Legendrian isotopies considered above. Denote by $\chi_t$ and $G_t$ the conformal factor and the generating Hamiltonian of the contact isotopy $\psi_t$. Define 
\begin{equation*}
    C \coloneqq \max_{(t,x) \in [0,1]\times M} |G_t(x)| + \left(C_0 + \eps \right) \max_{(t,x) \in [0,1] \times M} \chi_t(x),
\end{equation*}
and let $N = N(C,\eps,U)$ be the number in Corollary \ref{cor:SmallEnergyIsotopiesOfLegWithManyStabilizations}.

For $i \in \{0,1\}$, there exist by our choice of $V$ $N$-fold $\eps$-stabilizations $h_i:L \to M\setminus \Delta_i$ (with all stabilizations of the same type for each component of $L$) such that
\begin{equation*}
    d_\alpha(h_0,g_0) < \frac{C_0}{2} + \eps \quad\text{and}\quad
    d_\alpha(h_1,\psi_1 \circ g_1) < \frac{C_1}{2}+\eps .
\end{equation*}
Then $h_0$ and $\psi_1^{-1} \circ h_1$ are two formally isotopic Legendrian embeddings into the complement of $\Delta_0$. By Theorem \ref{thm:BoundednessInComplementOfOtDiskOfSizeC} and Lemma \ref{lem:EquiparametrizedHamiltonian}, there exists a Hamiltonian $H_t$ with $|H_t| < C_0+\eps$ such that $\phi^H_1 \circ h_0 = \psi_1^{-1} \circ h_1$. We may choose $H_t$ to be supported in $U$. Then the contact isotopy $\psi_t \circ \phi_t^H$ is generated by the Hamiltonian $G_t + \left(\chi_t H_t \right)\circ \psi_t^{-1}$, which satisfies the bound
\begin{equation*}
    \left|G_t + \left(\chi_t H_t \right)\circ \psi_t^{-1} \right| < \max_{(t,x) \in [0,1]\times M} |G_t(x)| + \left(C_0 + \eps \right) \max_{(t,x) \in [0,1] \times M} \chi_t(x) = C.
\end{equation*}
By Corollary \ref{cor:SmallEnergyIsotopiesOfLegWithManyStabilizations} applied to the Legendrian isotopy $\psi_t \circ \phi^H_t \circ h_0$ from $h_0$ to $h_1$, we obtain that
\begin{equation*}
    d_\alpha(h_0,h_1) < 3 \eps,
\end{equation*}
where we also used the triangle inequality and the observation that after an isotopy of energy less than $\eps$, we may assume by Corollary \ref{cor:boudnednessOfStrictWeinsteinNbhdOfHeightC} that the stabilizations of $h_i$ are contained in an arbitrarily small Weinstein neighbourhood.

Collecting the bounds, we can conclude that
\begin{equation*}
\begin{split}
    d_\alpha(f_0,f_1) &\leq d_\alpha(f_0,g_0) + d_\alpha(g_0,h_0) + d_\alpha(h_0,h_1) + d_\alpha(h_1,\psi_1 \circ g_1) + d_\alpha(\psi_1 \circ g_1, f_1) \\
    & < \eps + \left(\frac{C_0}{2} + \eps\right) + 3 \eps + \left(\frac{C_1}{2} + \eps\right) + \eps = \frac{C_0}{2} + \frac{C_1}{2} + 7 \eps.
\end{split}
\end{equation*}
As $\eps > 0$ was arbitrary, this finishes the proof.
\end{proof}

\begin{remark}\label{rmk:SmallEnergyIsotopiesWithFixedHmtpClassMovingOtDisk}
    By similar arguments as in Remark \ref{rmk:SmallEnergyIsotopiesWithFixedHmtpClass}, we can prescribe in Theorem \ref{thm:BoundMovingOtDisk} the (formal) Legendrian isotopy class of the Legendrian isotopy from $f_0$ to $f_1$. 
\end{remark}

\subsection{Consequences for positive loops and Shelukhin-Chekanov-Hofer distance}

The previous results imply that there are overtwisted contact manifolds in which the Shelukhin-Chekanov-Hofer distance is degenerate for every loose Legendrian isotopy class.
\begin{proof}[Proof of Corollary \ref{cor:LegegendriansWithVanishingd_alpha}]
    We choose a family $(\Delta_\textnormal{ot}^k)_{k \in \N}$ of mutually disjoint overtwisted disks in $(M,\xi)$ so that any compact subset $K \subseteq M$ intersects only finitely many of the disks $\Delta_\textnormal{ot}^k$. We choose a contact form $\alpha$ on $(M,\xi)$ such that $\Delta_\OT^k$ has overtwisted size $1/k$. Let $f_t: L \to M$ with $t \in [0,1]$ be an isotopy of closed Legendrian submanifolds. For any $\eps > 0$, we can find a number $k \in \N$ such that $f_t(L) \cap \Delta_\textnormal{ot}^k = \emptyset$ for all $t$ and $1/k < \eps$. It follows from Theorem \ref{thm:BoundednessInComplementOfOtDiskOfSizeC} that $d_\alpha(f_0,f_1) \leq 1/k < \eps$. As $\eps > 0$ was arbitrary, Corollary \ref{cor:LegegendriansWithVanishingd_alpha} is proven.   
\end{proof}

In a different direction, we can construct positive loops of controlled energy:
\begin{proof}[Proof of Corollary \ref{cor:PosLoops}]
    Let $\Delta \subseteq (M,\alpha)$ be a hyperdisk of size $C$ and fix $\eps > 0$. Let $f:L \to M$ be a closed Legendrian embedding into the complement of $\Delta$. We denote the time-$t$ Reeb flow of $\alpha$ by $\rho_t$. Then the image of the Legendrian embedding $\rho_{-(C+\eps)}\circ f$ is contained in the complement of the hyperdisk $\rho_{-(C+\eps)}(\Delta)$, which has overtwisted size $C$ by strictness of $\rho_t$. By Theorem \ref{thm:BoundMovingOtDisk}, $d_\alpha(f,\rho_{-(C+\eps)}\circ f) \leq C$; so there exists a Legendrian isotopy $(f_t)_{t \in [0,1]}$ from $f$ to $\rho_{-(C+\eps)} \circ f$ with $|\alpha(\partial_t f_t)| < C + \eps$. Then the loop $\left(\widetilde f_t\right)_{t \in [0,1]} = (\rho_{(C+\eps)t} \circ f_t)_{t \in [0,1]}$ satisfies 
    \begin{equation*}
        \alpha\left(\partial_t \widetilde f\right) = C + \eps + \alpha(\partial_t f) \in (0,2C +2 \eps).
    \end{equation*}
    By Remark \ref{rmk:BoundMovingOtDiskFixedHmtpClass}, the Legendrian isotopy $f_t$ can be chosen to be isotopic with fixed end-points to $\rho_{-(C+\eps)t} \circ f$. Then the loop $\widetilde{f}_t$ is contractible.
\end{proof}

\section{The exact size of the standard overtwisted disk} \label{sec:ClosedReebOrbitsInOTCtcMflds}

In this section we establish Theorem \ref{thm:SizeOfOtDisk}, which computes the exact size of the standard overtwisted disk.

The main ingredient behind this result is Theorem \ref{thm:ClosedReebOrbitsInOTCtcMflds}. It tells us that the overtwisted size is related to the Reeb action spectrum. Concretely: assuming that the manifold contains a sufficiently large strict Darboux ball, there are contractible Reeb orbits whose action is bounded above by the overtwisted size.

\subsection{Overtwisted size and Reeb spectrum}

We address the higher dimensional case first. The three dimensional case requires separate handling due to the need to stabilise arbitrarily many times.

\begin{proof}[Proof of Theorem \ref{thm:ClosedReebOrbitsInOTCtcMflds} for the case $\dim M \geq 5$]
    Fix $C > 0$. If $R$ is large enough, then for any sufficiently small $\eps > 0$ there exist two closed Legendrian embeddings $f,g:L\to B(R)$ such that
    \begin{enumerate}
        \item $f$ and $g$ are formally Legendrian isotopic,
        \item the image of $f(L)$ under the Reeb flow is contained in $B(R)$ for all times in  $[-\frac{C}{2} - \frac{\eps}{2},\frac{C}{2} + \frac{\eps}{2}]$,
        \item the minimal action of Reeb chords of $f(L)$ is larger than $C + \eps$,
        \item there are no Reeb chords between $f(L)$ and $g(L)$,
        \item $g$ admits a loose chart of size at most $\frac{\eps}{2}$.
    \end{enumerate}
    
    Now let $(M,\alpha)$ be a strict contact manifold containing a hyperdisk $\Delta$ of overtwisted size $C$, and assume that there exists a strict contact embedding $i: B(R) \to (M \setminus \Delta,\alpha)$. Write $\hbar_c$ for the minimal action among all Reeb chords of $(i \circ f)(L)$ and $\hbar_o$ for the minimal action among all contractible closed Reeb orbits of $(M,\alpha)$. We let $\hbar$ be the minimum of the two. Using items (2) and (3), we see that $\hbar_c > C + \eps$.
    
    According to Lemma \ref{lem:Nakamura}, the composition $i \circ f$ has loose size at most $C$. Theorem \ref{thm:LooseLegQuantitativeFlexibility} then implies:
    \begin{equation*}
        d_\alpha(i\circ f, i \circ g) \leq \frac{C}{2} + \frac{\eps}{4}.
    \end{equation*}
    Fix thus an isotopy of cost less than $\frac{C}{2} + \frac{\eps}{2}$. Write $H_t$ for the Hamiltonian generating it. Recall that $\Vert H_t \Vert_{osc} \leq 2 \Vert \phi^H_t \Vert_\alpha$.
    
    According to the energy-capacity inequality (Theorem \ref{thm:energy-capacity inequality}) one of the following statements must be true:
    \begin{itemize}
        \item $\hbar \leq \Vert H_t \Vert_{osc}$.
        \item $\hbar > \Vert H_t \Vert_{osc}$ and there is a Reeb chord between $(i \circ f)(L)$ and $(i \circ g)(L)$ of action less than $\Vert \phi^H_t \Vert_\alpha < \frac{C}{2} + \frac{\eps}{2}$.
    \end{itemize}
    The second situation cannot happen: According to items (2) and (4), the minimal action of Reeb chords between $(i \circ f)(L)$ and $(i \circ g)(L)$ is larger than $C + \eps$ and thus larger than $\Vert \phi^H_t \Vert_\alpha$. We are thus in the first situation, which implies that $\hbar \leq 2\Vert \phi^H_t \Vert_\alpha =< C+\eps$. Since $\hbar_c > C + \eps$ it follows that $\hbar_o < C + \eps$. The result then follows by noting that $\eps$ was arbitrary.

    The construction shows that $R$ can be chosen to be small if $C$ is sufficiently small.
\end{proof}
    
\begin{proof}[Proof of Theorem \ref{thm:ClosedReebOrbitsInOTCtcMflds} for the case $\dim M = 3$]
    We use the same setup as in higher dimensions. The difference is that it no longer makes sense to demand that $g$ has a small loose chart. Instead, we have to more carefully analyze the proof of Theorem \ref{thm:BoundednessInComplementOfOtDiskOfSizeC}.
    
    Given $C, \eps > 0$, choose $R$, $f$, $g$, and $i:B(R) \to M \setminus \Delta$ as above, satisfying Conditions (1)-(4). We then modify $g$ as follows: We fix a path $\gamma:[0,1] \to M$ from $g(L)$ to $\Delta$, in the complement of $(i \circ f)(L)$ and $(i \circ g)(L)$. We assume that there are no Reeb chords between $\gamma$ and $(i \circ f)(L)$ completely contained in the image of $i$.
    
    The proof of Theorem \ref{thm:BoundednessInComplementOfOtDiskOfSizeC} then implies that $i \circ g$ is Legendrian isotopic to some $\tilde g$ such that the following hold:
    \begin{enumerate}
        \item[(5a) ] $\tilde g(L)$ is contained in an arbitrarily small neighbourhood of $(i\circ g)(L) \cup \gamma([0,1]) \cup \Delta$.
        \item[(5b) ] $d_\alpha(i\circ f,\tilde g) < \frac{C}{2} + \frac{\eps}{4}$.
    \end{enumerate}
    Roughly, $\tilde g$ is obtained from $i \circ g$ by performing an isotopy that takes a point in $i \circ g$ to the vicinity of $\Delta$ and then a further isotopy to add many stabilizations. These are very small, ensuring Condition (5b).
    
    Our assumptions on $g$ and $\gamma$ imply that the minimal action among Reeb chords between $\tilde g(L)$ and $(i\circ f)(L)$ is greater than $C/2+\eps/2$. As before, the minimal action of Reeb chords of $(i \circ f)(L)$ satisfies $\hbar_c > C + \eps$. By the same argument as in the higher-dimensional case, it follows that $\hbar_0 = \hbar \leq \Vert H_t \Vert_{osc} \leq C + \eps$, which implies that $\hbar_0 \leq C$, since $\eps > 0$ was arbitrary.
\end{proof}

\subsubsection{The case of manifolds with boundary}

We now handle Corollary \ref{cor:closedReebOrbit}, the generalisation of Theorem \ref{thm:ClosedReebOrbitsInOTCtcMflds} to the case of manifolds with boundary. We first need the following auxiliary lemma:
\begin{lemma} \label{lem:contactDouble}
Let $(M,\xi)$ be a compact, coorientable contact manifold with boundary. Then, there is a codimension-$0$ embedding of $M$ into a closed, coorientable contact manifold $(N,\xi')$ such that $\xi'|_M = \xi$. In addition, we may assume that the induced map $\pi_1(M) \to \pi_1(N)$ of fundamental groups is injective.
\end{lemma}
\begin{proof}
We first prove the statement for almost contact structures. Extend $M$ a bit beyond its boundary to a contact manifold $M'$ (diffeomorphic to $M$) with contact form $\alpha$. Consider the manifold $W$ obtained from $M' \times [0,1]$ by rounding off the corners. Its boundary $\partial W$ is the double of $M'$. Endow $W$ with the $2$-form $d\alpha + dt \wedge \alpha$, where $t$ is the coordinate in $[0,1]$. This form is non-degenerate, so it restricts to a $2$-form of maximal rank on $\partial W$. This form defines, uniquely up to homotopy, an almost contact structure on $\partial W$ which, by construction, may be assumed to agree with $\xi$ over $M$. This almost contact structure is coorientable since $\partial W$ has trivial normal bundle in $W$.

To conclude, we introduce an overtwisted disk in $\partial W \setminus M$ and invoke the $h$-principle for overtwisted contact structures \cite{BEM}, relative to $M$.

The last claim follows from the fact that the inclusion of a manifold with boundary into its double is always injective on $\pi_1$. Indeed, $N = \partial (M \times [0,1])$ admits a projection onto $M$ which is a left-inverse of the inclusion $M \to N$.
\end{proof}

\begin{proof}[Proof of Corollary \ref{cor:closedReebOrbit}]

    It is enough to consider the case that $M$ is compact with boundary after possibly replacing $M$ by a compact submanifold of $M$ with boundary. The assumptions and Lemma \ref{lem:contactDouble} give us a compact, strict contact manifold $(N,\alpha)$, an equidimensional contact embedding $M \subset N$, a hyperdisk $\Delta \subseteq M$  of overtwisted size $C$, and an open neighbourhood $U \subseteq M$  of $\Delta$ with $\overline{U} \subseteq M^\circ$.
    
    For $x \in N \setminus M^\circ$, we define $\hbar_x$ as the supremum over all $T \geq 0$ such that 
    \begin{align*}
        \left[-\frac{T}{2},\frac{T}{2}\right] &\rightarrow N \setminus U \\
        t &\mapsto \rho_{t}(x) 
    \end{align*}
    is defined and embedded, where $\rho_t(x)$ denotes the Reeb flow in $N \setminus U$, at time $t$, applied to $x$. Let $\eps > 0$ be any number such that
    \begin{equation*}
        C< \hbar := \inf_{x \in N \setminus M^\circ} \hbar_x/\eps.
    \end{equation*}
    After possibly decreasing $\eps$, we may assume that $(N,\frac{1}{\eps}\alpha)$ contains a strict Darboux ball of radius $R > 0$ in the complement of $U$, where $R$ is such that Theorem \ref{thm:ClosedReebOrbitsInOTCtcMflds} applies to $R$ and hyperdisks of overtwisted size at most $C$.
    
    Let $\chi:N \to (0,\infty)$ be a smooth function such that $\chi|_{N \setminus U} \equiv 1/\eps$ and $\chi|_\Delta = 1$. The overtwisted size of $\Delta$ with respect to $\chi \alpha|_\Delta$ is still bounded above by $C$. In particular, Theorem \ref{thm:ClosedReebOrbitsInOTCtcMflds} implies that $(N,\chi\alpha)$ contains a contractible closed Reeb orbit of action at most $C$. Furthermore, by the definition of $\hbar$, any closed Reeb orbit of $\chi \alpha$ that intersects $N \setminus M^\circ$ has action at least $\hbar$, since $\chi|_{N \setminus U} = 1/\eps$. It follows that the contractible closed Reeb orbit has to be contained in $M$. By the last part of Lemma \ref{lem:contactDouble}, we may assume that this Reeb orbit is contractible also in $M$.
\end{proof}

\subsection{Computing the overtwisted size exactly}

\begin{proof}[Proof of Theorem \ref{thm:SizeOfOtDisk}]
We have shown (Proposition \ref{prop:controlledIsotopy}) that the size of the standard overtwisted disk in dimension 3 is bounded from above by $C^*$. We apply Theorem \ref{thm:OtDiskOfSizeCTimesBall} to the standard overtwisted disk $\Delta_\OT^\std \subseteq \R^3$ to conclude that any neighbourhood of
\begin{equation*}
    \Delta_\OT^\std \coloneqq \{z=0, \rho \leq \pi\} \times \left([-1,1] \times \left[-\frac{C^*}{2},\frac{C^*}{2}\right]\right)^{n-1} \subseteq \left(\R^{2n+1}, \alpha^\std_\OT + \sum_{i=1}^{n-1} y_i dx_i \right)
\end{equation*}
is overtwisted, and that also the size of the overtwisted disk $\Delta_\OT^\std$ in higher dimensions is bounded from above by $C^*$.

The lower bound follows from Theorem \ref{thm:ClosedReebOrbitsInOTCtcMflds} as soon as we construct a contact manifold containing $\Delta_\OT^\std$, a sufficiently large Darboux ball, such that the minimal action of closed contractible Reeb orbits is $C^*$. This is the content of Lemma \ref{lem:Closed ctc mfds with std ot disk} below.
\end{proof}

\begin{lemma}\label{lem:Closed ctc mfds with std ot disk}
    For any $n \geq 1$ and any $R > 0$, there exists a closed strict $2n+1$-dimensional contact manifold $(M,\alpha)$ with the following properties:
    \begin{enumerate}
        \item $(M,\alpha)$ contains a strict Darboux ball $B(R)$ of radius $R$,
        \item there exists a strict contact embedding of a neighbourhood of $\Delta^\std_\OT$ into $(M \setminus B(R),\alpha)$,
        \item the minimal action of closed Reeb orbits in $(M,\alpha)$ is given by $C^*$.
    \end{enumerate}
\end{lemma}

We first give the proof in three dimensions, since it is technically easier than the general construction. 
\begin{proof}[Proof of Lemma \ref{lem:Closed ctc mfds with std ot disk} in the $3$-dimensional case]
The idea is the following: We consider the standard contact structure $\alpha_\std$ on the 3-sphere. Its Reeb orbits are the fibres of the Hopf fibration $\pi: S^3 \rightarrow S^2$; they have period $2\pi$. The scaled form $\alpha_A := A\alpha_\std$ has thus orbits of period $2\pi A$. Moreover, $(S^3,\alpha_A)$ contains a strict Darboux ball of radius $\sqrt{A}$, disjoint from a given Hopf fibre. The proof will conclude by adding a Lutz twist along one of the fibres. Doing this carefully will allow us to keep track of all Reeb periods and yield the result.

We remark that the exact same proof would work for any prequantization bundle over a surface.

\textbf{Step 1:} Write $\omega_\std$ for the standard area form on $S^2$ of total area $2\pi$. We think of $(S^3,\alpha_\std)$ as the prequantization associated to the Hopf fibration $\pi: S^3 \to S^2$. The prequantization contact form $\alpha_\std$ satisfies $d \alpha_\std = \pi^* \omega_\std$ and its time-$t$ Reeb flow is given by $t$-translation along the Hopf fibres, which have period $2\pi$.

Let $B(1) \subset S^2$ be one of the hemispheres; it is identified with the unit disk and has area $\pi$. Restricting $\pi$ to $B(1)$ yields an identification
\[ (S^3,\alpha_\std)|_{B(1)} = (C(1) := B(1) \times S^1, \lambda_\std := dz + \frac{1}{2}\rho^2 d\phi); \]
here we use polar coordinates $(\rho,\phi)$ on $B(1)$ and $z \in [0,2\pi]$ is the coordinate on $S^1$.

Fix then a positive integer $A$ and write $\alpha_A := A\alpha_\std$ and $\lambda_A := A\lambda_\std$. We have thus the change of coordinates given by the strict contactomorphism:
    \begin{equation*}
    \begin{split}
        \Psi: (S^3,\alpha_A)|_{B(1)} \cong(C(1), \lambda_A) \quad&\longrightarrow\quad (C(A) :=B(\sqrt{A}) \times S^1_A, \lambda_\std = dz + \frac{1}{2} \rho^2 d\phi)\\
        (\rho,\phi,z) \quad&\mapsto\quad (\sqrt{A} \rho,\phi,Az).
    \end{split}
    \end{equation*}
Here $S^1_A$ is the circle of length $2\pi A$.

\textbf{Step 2:} We now modify the contact form $\lambda_A$ over $C(1)$ by introducing a Lutz twist. Concretely, using $\Psi$-coordinates (and thus working on $C(A)$), we replace $\lambda_\std$ by the contact form
    \begin{equation*}
        \lambda_\OT := f_1(\rho) dz + f_2(\rho) d \phi
    \end{equation*}
for two functions $f_1,f_2:[0,\infty) \to \R$ satisfying the following conditions:
    \begin{enumerate}
        \item $(f_1(\rho),f_2(\rho)) \neq (0,0)$ for all $\rho \geq 0$.
        \item For some $\eps >0$, $(f_1(\rho),f_2(\rho)) = (\cos(\rho),\rho \sin(\rho))$ for $\rho \in [0, \pi + \eps]$.
        \item The angular coordinate of $(f_1,f_2) \in \R^2$ is strictly increasing. Equivalently, $f_1 f_2' - f_1' f_2 > 0$.
        \item $(f_1(\rho), f_2(\rho)) = (1, \frac{1}{2} \rho^2)$ for all $\rho \geq 2 \pi$.
        \item Whenever $f_2'(\rho)=0$ for some $\rho > \pi$, it holds that $|2 \pi f_2(\rho)| > C^*$.
    \end{enumerate}
These functions do not depend on $A$, but we have to assume that $\sqrt{A}>2\pi$.

Such a pair $(f_1,f_2)$ of functions can be constructed as follows. Item (2) defines them for $\rho \in [0, \pi + \eps]$. In particular, $(f_1(\pi),f_2(\pi)) = (-1,0)$ and $(f_1'(\pi),f_2'(\pi)) = (0,-\pi)$. We extend $f_1$ to all of $[0,\infty)$ such that $f_1'> 0$ on $(\pi,2\pi)$ and $f_1 \equiv 1$ on $[2\pi,\infty)$. Then there exists a unique $\rho_0 \in [\pi,2\pi]$ so that $f_1(\rho_0) = 0$. We then extend $f_2$ to all of $[0,\infty)$ in such a way that $f_2' < 0$ on $(\pi,\rho_0)$, $f_2(\rho_0) = -2 < -\frac{C^*}{2\pi}$ is a local minimum of $f_2$, $f_2' > 0$ on $(\rho_0,\infty)$, and $f_2(\rho) = \frac{1}{2}\rho^2$ on $[2\pi,\infty)$.

We define $\alpha_\OT$ to be the contact form obtained from $\alpha_A$ by locally replacing $\lambda_A$ for $\lambda_\OT$.

\textbf{Step 3:} We claim that $(S^3,\alpha:=\alpha_\OT)$ satisfies the properties claimed in the statement. We must compute its Reeb periods.

Within the model $C(1)$, using the identification with $C(A)$ given by $\Psi$, we see that the Reeb vector field of $\lambda_\OT$ reads
    \begin{equation*}
        R = \frac{f_2'}{f_1f_2'-f_1'f_2} \partial_z - \frac{f_1'}{f_1f_2'-f_1'f_2} \partial_\phi.
    \end{equation*}
First, we look at its horizontal orbits. This happens when the $z$-component of $R$ vanishes, i.e. when $f_2'(\rho)=0$. In that case we have that $R = \frac{1}{f_2(\rho)} \partial_\phi$. At such a value of $\rho$ we get a closed Reeb orbit of action $|2 \pi f_2(\rho)|$. If additionally $\rho < \pi$ holds, then $2 \pi f_2(\rho) = C^*$. If instead $\rho \geq \pi$, then $|2\pi f_2(\rho)| > C^*$ by assumption.

At all other $\rho$, the orbits have a non-zero vertical component and must thus complete at least one vertical loop. The vertical displacement is given by the vector $\frac{f_2'}{f_1f_2'-f_1'f_2} \partial_z$; its length over $C(2\pi)$ is bounded from above, by compactness. Since the vertical has length $2\pi A$, it follows that the minimal period among all non-horizontal closed orbits grows linearly with $A$. We then take $A$ large enough to guarantee their periods are greater than $C^*$.

Outside of $C(2\pi) \subset C(A)$ (and in particular in the rest of $S^3$, away from the model) all Reeb orbits have not been modified, so they have period $2\pi A$. This is greater than $C^*$ if $A$ is large enough.

The argument is complete by noting that the cylinder in $S^3$ lying over the southern hemisphere of $S^2$ contains a strict Darboux ball of radius $\sqrt{A}$.
\end{proof}

\begin{proof}[Proof of Lemma \ref{lem:Closed ctc mfds with std ot disk} in arbitrary dimension]
The proof will establish that any closed overtwisted contact manifold $(M,\xi)$ admits a contact form $\alpha$ with the properties claimed. To prove this, it is sufficient to construct a family of local models $(N,\beta_D)$, $D \in [1,\infty)$, with the following properties:
\begin{itemize}
    \item[A. ] Close to $\partial N$, $\beta_D = D\beta_1$.
    \item[B. ] The minimal Reeb period within $N$ is $C^*$.
    \item[C. ] The model is closed under the Reeb flow and foliated by Reeb orbits.
    \item[D. ] Any overtwisted contact manifold admits a contact embedding of the model.
    \item[E. ] There exists a strict contact embedding of a neighbourhood of $\Delta^{\textnormal{std}}_\textnormal{ot}$ into $(N,\beta_D)$ for all $D$.
\end{itemize}
Indeed, once we have this, we define $(M,\alpha_1)$ to be an arbitrary extension of $(N,\beta_1)$ to the rest of $M$. Then we can define $\alpha_D$ to be $D\alpha_1$ in the complement of $N$ and $\beta_D$ over $N$. If we let $\hbar_c$ be the minimal Reeb period of $\alpha_1$, we have that the minimal Reeb period of $D\alpha_1$ is $D\hbar_c$, which can be assumed to be greater than $C^*$ if $D$ is large enough. This means that the minimal Reeb period of $\alpha_D$ is $C^*$ and the proof is complete.

The rest of the proof amounts to constructing $(N,\beta_D)$.

\textbf{Step 1:} Our construction of the models involves the constant $D\geq 1$, as well as auxiliary constants $\delta>0$, $R>0$, $\eps>0$, $L>0$. We will choose them as we go along. It may be helpful to keep in mind that $\delta$ and $\eps$ need to be sufficiently small and $R$ and $L$ need to be sufficiently large. Moreover, $R$ will eventually be chosen (Equation \ref{eq:choiceOfR}) to be a function of $\delta$. None of these will depend on $D$.

Consider the cylinder
\begin{equation*}
N' := \R^2 \times S^1_L \times \R^{2n-2} 
\end{equation*}
endowed with coordinates $(\rho,\theta,z,x_i,y_i)$, such that $(\rho,\theta)$ are polar coordinates in the plane, $(x,y)$ are standard symplectic coordinates in $\R^{2n-2}$, and $z$ is the coordinate in the circle $S^1_L$ of length $L$. We denote $s := \rho^2+\sum_{i=1}^{n-1}(x_i^2+y_i^2)$ and $\lambda:=\sum_{i=1}^{n-1} y_i\,dx_i$ for the standard Liouville form on $\R^{2n-2}$. We then set $N(R) := \{s \leq R^2\} \subset N'$.

We impose $R>2\pi$. Consider then the $1$-forms:
\[ \alpha = \mu + \lambda,\qquad \beta := h(s)\alpha, \qquad \mu := a(\rho)dz + b(\rho)d\theta. \]
Here $a$, $b$, and $h$ are functions to be fixed as we go along. In the case of $a$ and $b$, we just need that they satisfy the following properties:
    \begin{enumerate}
        \item $(a(\rho),b(\rho)) \neq (0,0)$ for all $\rho \geq 0$.
        \item For some $\eps >0$, $(a(\rho),b(\rho)) = (\cos(\rho),\rho \sin(\rho))$ for $\rho \in [0, \pi + \eps]$.
        \item The angular coordinate of $(a,b) \in \R^2$ is strictly increasing. Equivalently, $a b' - a' b > 0$.
        \item $(a(\rho), b(\rho)) = (1,\rho)$ for all $\rho \geq 2 \pi$.
        \item Whenever $b'(\rho)=0$ for some $\rho > \pi$, it holds that $|2 \pi b(\rho)| > C^*$.
    \end{enumerate}
These properties guarantee, in particular, that $\alpha$ is an overtwisted contact form as long as $R$ is large enough. Indeed, if that is the case, we have an obvious strict contact embedding $\Delta_\std^\OT \longrightarrow (N(R),\alpha)$, i.e.\,Condition (E) holds.

Moreover, $(N(R),\alpha)$ is formally homotopic to the Weinstein neighbourhood around a transverse curve. This implies that it can be formally embedded into any formal contact manifold. As a consequence of the $h$-principle, the model admits a contact embedding into any overtwisted $(M,\xi)$, i.e.\,the model satisfies Condition (D).

\textbf{Step 2:} We now compute the Reeb flow of $\beta$ and make a choice for the function $h$. For notational convenience we introduce the auxiliary functions
\[ W(\rho):=a(\rho)b'(\rho)-a'(\rho)b(\rho),
\qquad \gamma(s):=\frac{h'(s)}{h(s)}, 
\qquad s_y:=\sum_{i=1}^{n-1} y_i^2. \]
Consider then the vector fields
\[ R_0 := \frac{b'(\rho)}{W(\rho)}\partial_z - \frac{a'(\rho)}{W(\rho)}\partial_\theta, \quad\textrm{and}\quad J:=\sum_{i=1}^n \left(-y_i\partial_{x_i} + x_i\partial_{y_i}\right). \]
The former is the Reeb vector field of $\mu$. The latter is Hamiltonian for $d\lambda$ with Hamiltonian function $\frac{1}{2}\sum_i x_i^2 + y_i^2$. Then, the Reeb of $\beta$ reads
\[ R_\beta = \frac{1}{h(s)}
\left[ \left(1+2\gamma(s)s_y\right)R_0 + 
\frac{2\rho\,\gamma(s)}{W(\rho)}\left(b(\rho)\partial_z - a(\rho)\partial_\theta\right)\right] +
\frac{2\gamma(s)}{h(s)}J. \]
We write this as $R_\beta = R_\beta' + R_\beta''$, with the second summand the term proportional to $J$. We see then that $\rho$ and $x^2+y^2$ (and thus $s$) are conserved quantities. In particular, the dynamics of $R_\beta$ are a lift of the dynamics of $R_\beta''$. Since $s$ is a conserved quantity, Condition (C) holds for the model.

Given any $\delta>0$ small, we claim that there is a radius $R = R(\delta)>0$ large enough such that the following holds: Given any constant $D>1$ we can find $h$ satisfying:
\begin{enumerate}
    \item $h$ is identically $1$ in $[0,R^2/2]$.
    \item $h$ is identically $D$ in $[R^2,\infty)$.
    \item $h'>0$ in $(R^2/2,R^2)$.
    \item $|\gamma(s)/h(s)| \leq \delta$ for all $s$.
\end{enumerate}
Note that Item (2) implies Condition (A) for the model. An explicit formula for $R$ in terms of $\delta$ is given below in Equation \ref{eq:choiceOfR}.

Now, consider the auxiliary function $h_0(s): [R^2/2,R^2] \rightarrow [1,D]$ given by the expression 
\[ h_0(s) := \frac{DR^2}{DR^2+2(1-D)(s-R^2/2)}. \]
It satisfies $h_0(R^2/2) = 1$ and $h_0(R^2) = D$. Its derivative $h_0'$ is increasing and takes values in 
\[ \left[\dfrac{2(D-1)}{DR^2},\dfrac{2D(D-1)}{R^2}\right]. \]
The function $h_0$ is therefore increasing itself. Moreover, we have
\[ \frac{\gamma_0(s)}{h_0(s)} = \frac{2(D-1)}{DR^2}, \]
which is thus constant and can be assumed to be smaller than $\delta/2$, independently of $D$, by taking 
\begin{align} \label{eq:choiceOfR}
R(\delta) := 2\sqrt{1/\delta}.  
\end{align}
We extend $h_0$ to $[0,\infty)$, as a piecewise function, by letting it be $1$ over $[0,R^2/2]$ and $D$ over $[R^2,\infty)$.

We then let $h$ be a smoothing of $h_0$ satisfying Items (1), (2), and (3). We must show that this can be done while maintaining Item (4). To do so, we replace the derivative $h_0'$ by some $h'$, which we then integrate to yield $h$. The smoothing is done region by region. 

We let $\eps>0$ be a constant that we will take to be small. Over $[R^2/2,R^2/2+\eps]$ we let $h'$ be a function that interpolates from $0$ to $h_0'$, is positive if $s>R^2/2$, and is bounded above by $h_0'$. By construction, the integrals $h|_{[R^2/2,R^2/2+\eps]}$ and $h_0|_{[R^2/2,R^2/2+\eps]}$ differ by some quantity $d_1$ that may be assumed to be arbitrarily small by choosing $\eps$ small.

Over $[R^2/2+\eps,R^2-\eps]$ we set $h' := h_0'$. It follows that $h|_{[R^2/2,R^2-\eps]}$ is still below $h_0$ and their difference is still bounded by $d_1$.

Over $[R^2-\eps,R^2]$ we require that $0 \leq h' \leq 2h_0'$ and that $h'(s)=0$ only at $s=R^2$, where it is flat. We claim that we can adjust the area under the curve $h'|_{[R^2-\eps,R^2]}$ so that, upon integration, we get $h(R^2) = D$. Indeed, consider a $1$-parametric family of functions satisfying the properties above. Let one of them be bounded above by $h'_0$ and another one have value $2h_0'$ over most of the interval; the area under the latter is thus greater than $d_1$ plus the area under $h_0'|_{[R^2-\eps,R^2]}$. Using the intermediate value theorem, it follows that one of them yields the desired value at the endpoint. Our assumption $h' <2h_0'$ implies that $h$ satisfies Item (4) if $\eps$ is sufficiently small.

\textbf{Step 3:} It remains to verify that the model $(N(R+1),\beta)$ satisfies Condition (B), as long as $\delta$ is sufficiently small (and therefore $R(\delta)$ is sufficiently large).

There are two situations. First, it could be that $\gamma(s)=0$. In that case, the Reeb vector field reads $R_\alpha = R_0/h(s)$. This means that the shortest closed orbits are those with $h(s)=1$. The computation is now like in the $3$-dimensional case: the shortest orbit will have period $C^*$ as soon as we take $L$, the length of the $z$-coordinate, large enough (independently of $D$).

Suppose otherwise that $\gamma(s) \neq 0$. This means $s \in (R^2/2,R^2)$. We then separate into two cases. The first is that $\rho \in [2\pi,R]$. Then, the formula for the Reeb vector field particularises to:
\[ R_\alpha = \frac{1}{h(s)}
\left[ (1+2\gamma(s)(\rho^2+s_y))\partial_z
- 2\gamma(s)\rho\,\partial_\theta + 2\gamma(s)J\right]. \]
Since $\gamma(s) \neq 0$, the motion in $\partial_\theta$ is non-trivial. Its angular speed is bounded above by $2\delta\rho \leq 2\delta R$. This means, according to Equation \ref{eq:choiceOfR}, that the period of any closed Reeb orbit is larger than 
\[ \dfrac{2\pi}{2\delta R} = \dfrac{\pi}{2\sqrt{\delta}} > C^*, \] 
as long as $\delta$ is small enough.

The second possibility is that $\rho < 2\pi$, in which case $x^2+y^2 > R^2/2 - 4\pi^2$, which we can assume is larger than $R^2/4$ if $R$ is large enough. Then, the motion in $J$ is non-trivial and its speed is bounded above by $2\delta R$. Note that the motion in $J$ is circular and its trajectories have euclidean length at least $\pi R$. This implies that their periods are at least $\pi/{2\delta}$, which can be assumed to be larger than $C^*$ if $\delta$ is small enough. This concludes the proof.
\end{proof}

\section{Products} \label{sec:products}

\subsection{Products with Liouville manifolds}

We will now prove Corollary \ref{lem:OTdiskInLiouvilleProd}, which states that products with Liouville manifolds can only decrease the overtwisted size. We first need the following lemma:
\begin{lemma} \label{lem:bigLiouvilleBall}
Let $(N,\lambda)$ be a complete $2k$-dimensional Liouville manifold and $C,D>0$ constants. Then, for any sufficiently large $a > 0$ there exists an embedding
\[ \left(([-1,1] \times [-C,C])^{k-1} \times ([-1,1] \times [a,a+D]) \subset \R^{2k}, \sum_{i=1}^k y_idx_i\right) \quad\longrightarrow\quad (N,\lambda) \]
that preserves Liouville forms.
\end{lemma}
\begin{proof}
The idea is to start with a small hypercube. Since the Liouville vector field is expanding, this small cube will become arbitrarily big if we flow for sufficiently long times.

Let $Y \subset N$ a hypersurface transverse to the Liouville vector field $X_\lambda$. Such a $Y$ can be found by considering a point $q \in N$ in which $X_\lambda(q) \neq 0$ and letting $Y$ be a sufficiently small transverse hyperdisk passing via $q$. According to Darboux's theorem, there exists a strict contact embedding 
    \begin{equation*}
    \left([-\eps,\eps]_z \times ([-\eps,\eps]_{x} \times [-\eps,\eps]_y)^{k-1}, dz + \sum_{i=1}^{k-1} y_i d x_i\right) \quad\longrightarrow\quad (Y,\lambda|_Y)
    \end{equation*}
for some $\eps > 0$.

We can then apply the Liouville flow for an amount of time $\ln(1/\eps^2)$ to $Y$; we still denote it as $Y$. In doing so we change the contact form to $\frac{1}{\eps^2}\left(dz + \sum_{i=1}^{k-1} y_i d x_i\right)$. We will assume that $\eps <1$. Consider then the strict contact embedding:
    \begin{equation*}
    \begin{split}
        \left[-1,1\right] \times \left([-1,1] \times \left[-1,1\right]\right)^{k-1} \quad&\longrightarrow\quad [-\eps,\eps] \times ([-\eps,\eps] \times [-\eps,\eps])^{k-1}, \\
        (z,x_i,y_i) \quad&\mapsto\quad (\eps^2 z, \eps x_i, \eps y_i)
    \end{split}
    \end{equation*}
where the contact form on the target is $\frac{1}{\eps^2}\left(dz + \sum_{i=1}^{k-1} y_i d x_i\right)$ and the contact form on the domain is $dz + \sum_{i=1}^{k-1} y_i d x_i$. In this manner we see the domain as a subset of $(Y,\lambda|_Y)$.

Now we choose $\delta>0$ small so that there is a strict Liouville embedding $(Y \times [1,1+\delta], s \lambda|_Y)$ into $(N,\lambda)$. Recall that the Liouville flow acts by multiplication on the $s$-coordinate, i.e. the time-$t$ flow sends $s$ to $se^t$. Moreover, the Liouville flow expands the Liouville form. Using that the Liouville flow is complete, it follows that we can obtain a strict Liouville embedding of 
\begin{equation*}
    \left(\left[-1,1\right] \times \left([-1,1] \times \left[-1,1\right]\right)^{k-1} \times [b,b(1+\delta)], s\left(dz + \sum_{i=1}^{k-1} y_i d x_i\right)\right)
\end{equation*}
into $(N,\lambda)$, for any $b=e^t > 1$.

Assume now that $b\eps \geq 1$. Then, the following embedding is well-defined:
    \begin{equation*}
    \begin{split}
        [-1,1] \times \left([-1,1] \times \left[-\frac{1}{\eps},\frac{1}{\eps}\right]\right)^{k-1} \times [b\eps,b\eps(1+\delta)] & \to \left[-1,1\right] \times \left([-1,1] \times \left[-1,1\right]\right)^{k-1} \times [b,b(1+\delta)]\\
        (z,x_i,y_i,s) &\mapsto \left(\eps z,x_i,\eps\frac{y_i}{s},\frac{s}{\eps}\right),
    \end{split}
    \end{equation*}
Moreover, it pulls back the Liouville form $s\left(dz + \sum_{i=1}^{k-1} y_i d x_i\right)$ to $s dz + \sum_{i=1}^{k-1} y_i d x_i$. Now we impose that $\eps$ is sufficiently small so that  $C < \frac{1}{\eps}$. We also require $b$ to be large enough to guarantee $D < b\eps\delta$. Lastly, we set $a = b\eps$.
\end{proof}

\begin{proof}[Proof of Corollary \ref{lem:OTdiskInLiouvilleProd}]
Fix a constant $C > 0$ and an open set $U \subseteq M$. We claim that there exists a strict contact embedding
\[ (U \times \left([-1,1] \times [-C,C]\right)^{k}, \alpha + \sum_i y_i dx_i) \longrightarrow (M \times N, \alpha+\lambda). \]
Corollary \ref{lem:OTdiskInLiouvilleProd} then follows from the claim as an application of Theorem \ref{thm:OtDiskOfSizeCTimesBall}.

Let us prove the claim. Denote by $V \subseteq M$ the trace of $U$ under the Reeb flow $\rho_t$, for all $t \in \R$. Fix a constant $\gamma > 0$. Computing on Darboux charts we see then that the map
    \begin{equation*}
        \begin{split}
            U \times ([-1,1] \times [-C,C])^k &\to V \times ([-1,1] \times [-C,C])^{k-1} \times [-1,1] \times [-C+\gamma, C + \gamma]\\
            (p,x_i,y_i) &\mapsto \left(\rho_{-\gamma x_k}(p),x_1,y_1,...,x_{k-1},y_{k-1},x_k,y_k + \gamma\right)
        \end{split}
    \end{equation*}
is a strict embedding with respect to the contact form $\alpha + \sum_{i = 1}^k y_i d x_i$ both on the target and the domain. We then apply the auxiliary Lemma \ref{lem:bigLiouvilleBall}. We choose the numbers $a$ and $\gamma$ in such a way that $-C + \gamma > a$, and $D > C + \gamma$. The claim follows and the proof is complete.
\end{proof}

\begin{remark}\label{rmk:OTdiskInSymplectizationProd}
Consider the special case of the symplectization
\[ (N,\lambda) := (SY = Y \times (0,\infty), s \alpha_Y) \]
of a strict contact manifold $(Y,\alpha_Y)$. Then, it follows from the proof of Corollary \ref{lem:OTdiskInLiouvilleProd} that $(M \times SY, \alpha + s \alpha_Y)$ contains an overtwisted disk of size $\leq C$ inside of $M \times Y \times [R,\infty)$ for any $R>0$ by choosing the variable $a$ in the proof to be sufficiently large provided that $(M,\alpha)$ admits an overtwisted disk of size $C$.
\end{remark}

\subsection{Contact products}

\begin{proof}[Proof of Theorem \ref{thm:vanihsingOTsizeInCtcProd}]
    Let $(M_1,\alpha_1)$ and $(M_2,\alpha_2)$ be two strict contact manifolds, and let $\Delta_\textnormal{ot} \subseteq M_1$ have overtwisted size $C$. Observe that $(M_1 \times M_2 \times (0,\infty), s\alpha_2 - \alpha_1)$ is the product of $(M_1,-\alpha_1)$ with the symplectization $(SM_2, s \alpha_2)$. As such, Corollary \ref{lem:OTdiskInLiouvilleProd} and Remark \ref{rmk:OTdiskInSymplectizationProd} apply to conclude that for any $R>0$, there exists a disk of overtwisted size $\leq C$ inside of $(M_1 \times M_2 \times [R,\infty),s\alpha_2 - \alpha_1)$. The contactomorphism
    \begin{equation*}
         \phi:(M_1 \times M_2 \times (0,1/R], \alpha_2 - s\alpha_1) \to (M_1 \times M_2 \times [R,\infty),s\alpha_2 - \alpha_1)
    \end{equation*}
    induced by the map $s \mapsto 1/s$ has conformal factor $1/s \geq R$, so $\phi^{-1}(\Delta_\textnormal{ot})$ is a disk of overtwisted size $\leq \frac{C}{R}$. This completes the proof, since $R>0$ can be chosen to be arbitrarily large.
\end{proof}


\bibliographystyle{amsalpha}
\bibliography{references}

@article{AA25a ,
	AUTHOR = {Allais, Simon and Arlove, Pierre-Alexandre},
	TITLE = {{$C^1$-Local Flatness and Geodesics of the Legendrian Spectral Distance}},
	YEAR = {2025},
	journal = {Proc. Amer. Math. Soc.},
	volume = {153},
	pages = {2139-2152}
}

@article{AA25,
author = {Allais, Simon and Arlove, Pierre-Alexandre},
title = {Spectral selectors and contact orderability},
journal = {Communications in Contemporary Mathematics},
volume = {28},
number = {6},
year = {2026},
doi = {10.1142/S0219199725500713}
}

@article{Can23,
	author = {Cant, Dylan},
	year = {2023},
	month = {08},
	pages = {},
	title = {Remarks on the oscillation energy of {L}egendrian isotopies},
	volume = {217},
	journal = {Geometriae Dedicata},
}

@article{CN25,
	title={{An overtwisted convex hypersurface in higher dimensions}},
	author={Chiang, River and Niederkrüger-Eid, Klaus},
	journal={Algebraic \& Geometric Topology},
	year={2025},
	volume={25},
	number={7},
	pages = {3813-3831}
}

@article{DRS20,
	author = {Dimitroglou Rizell, Georgios and Sullivan, Michael},
	year = {2020},
	month = {11},
	pages = {},
	title = {The persistence of the {C}hekanov-{E}liashberg algebra},
	volume = {26},
	journal = {Selecta Mathematica}
}

@article{DRS24,
	author = {Dimitroglou Rizell, Georgios and Sullivan, Michael},
	year = {2024},
	month = {08},
	pages = {2145--2206},
	title = {The persistence of a relative {R}abinowitz-{F}loer complex},
	volume = {28},
	journal = {{G}eometry \& {T}opology},
}

@article{Dy,
author = {K. {Dymara}},
    title = {{Legendrian knots in overtwisted contact structures on $S^3$}},
		journal = {Ann. Global Anal. Geom.},
		year = 2001,
		VOLUME = {19},
		NUMBER = {3},
		PAGES = {293--305}
}

@article{EliFras,
author = {Y. {Eliashberg} and M. {Fraser}},
    title = {{Classification of topologically trivial Legendrian knots}},
		journal = {CRM Proc. Lecture Notes},
		year = 1998,
		VOLUME = {15},
		PAGES = {17--51}
}

@article {etn13,
	AUTHOR = {John B. {Etnyre}},
	TITLE = {On knots in overtwisted contact structures},
	JOURNAL = {Quantum Topol.},
	VOLUME = {4},
	Number = {3},
	YEAR = {2013},
	PAGES = {229–264}
}

@article{El89,
	author = {Y. {Eliashberg}},
	title = {Classification of overtwisted contact structures on 3-manifolds},
	journal = {Invent. Math.},
	volume = {98},
	year = {1989},
	number = {3},
	pages = {623--637}
}

@article{BEM,
	author = {M.S. {Borman} and Y. {Eliashberg} and E. {Murphy}},
	title = {Existence and classification of overtwisted contact structures in all dimensions},
	journal = {Acta Math.},
	volume = {215},
	year = {2015},
	number = {2},
	pages = {281--361}
}

@article{Mur,
	author = {E. {Murphy}},
	title = {Loose Legendrian Embeddings in High Dimensional Contact Manifolds},
	year = {2012},
	journal = {arXiv e-prints},
	archivePrefix = "arXiv",
	eprint = {1201.2245}
}

@book{Gei2,
	author = {H. {Geiges}},
	title = {An Introduction to Contact Topology},
	series = {Cambr. Studies in Adv. Math.},
	volume = {109},
	publisher = {Cambr. Univ. Press},
	year = {2008}
}

@article{ah09,
	author = {Albers, Peter and Hofer, Helmut},
	year = {2009},
	pages = {429–436},
	title = {{On the Weinstein conjecture in higher dimensions}},
	volume = {84},
	journal = {Commentarii Mathematici Helvetici}
}

@article{Cha22,
	title = {Links in overtwisted contact manifolds},
	journal = {Expositiones Mathematicae},
	volume = {40},
	number = {2},
	pages = {231-248},
	year = {2022},
	author = {Chatterjee, Rima}
}

@article{Gir20,
	title = {{On some examples and constructions of contact manifolds}},
	journal = {Math. Ann.},
	volume = {376},
	pages = {957–1008},
	year = {2020},
	author = {Gironella, Fabio}
}

@article{HCMMP20,
	author = {Hern\'andez-Corbato, Luis and Mart\'in-Merch\'an, Luc\'ia and Presas, Francisco},
	title = {{Tight neighborhoods of contact
 submanifolds}},
	journal = {Journal of symplectic geometry},
	volume = {18},
	number = {6},
	pages = {1629-1646},
	year = {2020}
}

@article{Ho,
	author = {H. {Hofer}},
	title = {Pseudoholomorphic curves in symplectizations with applications to the {Weinstein} conjecture in dimension three},
	journal = {Invent. math.},
	volume = {114},
	pages = {515--563},
	year = {1993}
}

@article{NiederkrugerPresas2010,
  author    = {Klaus Niederkrüger and Francisco Presas},
  title     = {Some remarks on the size of tubular neighborhoods in contact topology and fillability},
  journal   = {Geometry \& Topology},
  volume    = {14},
  number    = {2},
  pages     = {719--754},
  year      = {2010},
  doi       = {10.2140/gt.2010.14.719},
  eprint    = {0812.2108},
  archivePrefix = {arXiv},
  primaryClass  = {math.SG}
}

@unpublished{Mur13,
	author = {Murphy, Emmy},
	NOTE = {preprint, arXiv:1304.6620},
	year = {2013},
	title = {Closed exact {L}agrangians in the symplectization of contact manifolds}
}

@article{Nak24,
	author = {Nakamura, Lukas},
	title = {{Small-energy isotopies of loose Legendrian submanifolds}},
	journal = {Geom. \& Top.},
	volume = {28},
    issue = {9},
	pages = {4233–4255},
	year = {2024}
}

@unpublished{Nak23,
	AUTHOR = {Nakamura, Lukas},
	NOTE = {preprint,  	arXiv:2307.10905},
	TITLE = {A new metric on the contactomorphism group of orderable contact manifolds},
	YEAR = {2023}
}

@article{Nak24a,
	author = {Nakamura, Lukas},
	year = {2024},
	pages = {147–151},
	title = {{Legendrians with vanishing
	Shelukhin-Chekanov-Hofer metric}},
	volume = {62},
	journal = {Ark. Mat.}
}

@article{Nie,
	author = {K. {Niederkr\"uger}},
	title = {The plastikstufe: A generalization of the overtwisted disk to higher dimensions},
	journal = {Algebr. Geom. Topol.},
	volume = {6},
	pages = {2473--2508},
	year = {2006}
}

@article{NR,
	author = {K. {Niederkr\"uger} and A. {Rechtman}},
	title = {The {Weinstein} conjecture in the presence of submanifolds having a Legendrian foliation},
	journal = {J. Topol. Anal.},
	volume = {3},
	year = {2011},
	number = {4},
	pages = {405--421}
}

@article{FT,
    author = {Fuchs, Dmitry and Tabachnikov, Serge},
    year = {1997},
    month = {9},
    pages = {1025-1053},
    title = {{Invariants of Legendrian and transverse knots in the standard contact space}},
    volume = {36},
    issue = {5},
    journal = {Topology}
}

@article{shelukhin2017hofer,
  author       = {Shelukhin, Egor},
  title        = {{The Hofer norm of a contactomorphism}},
  journal      = {Journal of Symplectic Geometry},
  volume       = {15},
  number       = {4},
  pages        = {1173--1208},
  year         = {2017},
  doi          = {10.4310/JSG.2017.v15.n4.a7},
}

@article{liu2020positive,
  author       = {Liu, Guogang},
  title        = {Positive loops of loose Legendrian embeddings and applications},
  journal      = {Journal of Symplectic Geometry},
  volume       = {18},
  number       = {3},
  pages        = {867--887},
  year         = {2020},
  doi          = {10.4310/JSG.2020.v18.n3.a7},
}

@article{DRSC0Limits,
	author = {Dimitroglou Rizell, Georgios and Sullivan, Michael},
	year = {2024},
	pages = {2904-2915},
	title = {${C}^0$-limits of {L}egendrians and positive loops},
	volume = {160},
	issue = {12},
	journal = {Compositio Mathematica}
}

@article{RoZhChekanovsDichotomy,
	title={Chekanov's dichotomy in contact topology},
	author={Daniel Rosen and Jun Zhang},
	journal={Mathematical Research Letters},
	year={2020},
	volume={27},
	number={4},
	pages={1165-1193}
}

@article{cf24,
	author = {Cardona, Robert and Presas, Francisco},
	title = {An h-principle for embeddings transverse to a contact structure},
	journal = {Journal of Topology},
	volume = {17},
	number = {1},
	pages = {e12326},
	year = {2024}
}

@unpublished{HS24,
	author = {Hedicke, Jakob and Shelukhin, Egor},
	title={{Non-orderability and the contact Hofer norm}},
	note={preprint, 	arXiv:2411.19887},
	year={2024}
}

@article{PP15,
	author = {A. {del Pino} and F. {Presas}},
	title = {The foliated {Weinstein} conjecture},
	journal = {Int. Math. Res. Not.},
	volume = {16},
	pages = {5148--5177},
	year = 2018,
}

@article{MNPS,
    author = {E. {Murphy} and K. {Niederkr{\"u}ger} and O. {Plamenevskaya} and A. {Stipsicz}},
    title = {{Loose Legendrians and the plastikstufe}},
    volume = {17},
    journal = {Geometry \& Topology},
    number = {3},
    publisher = {MSP},
    pages = {1791 -- 1814},
    year = {2013}
}

@article{Hu17,
author = {Y. {Huang}},
title = {On plastikstufe, bordered Legendrian open book and overtwisted contact structures},
journal = {Journal of Topology},
volume = {10},
number = {3},
pages = {720-743},
year = {2017}
}

@article{Hu20,
author = {Y. {Huang}},
title = {Existence of coLegendrians in contact 5-manifolds},
journal = {arXiv:2006.11844},
year = {2020}
}

@article{CMP,
author = {R. {Casals} and E. {Murphy} and F. {Presas}},
title = {Geometric Criteria for Overtwistedness},
journal = {J. Amer. Math. Soc.},
volume = {32},
pages = {563-604},
year = {2019}
}

@article{Mon,
	author = {R. {Montgomery}},
	title = {{Engel} deformations and contact structures},
	series = {Amer. Math. Soc. Transl. Ser.},
	PUBLISHER = {American Mathematical Society},
  ADDRESS = {Providence, RI},
	volume = {196},
	year = {1999},
	pages = {103--117}
}

@article{Vo3,
	author = {T. {Vogel}},
	title = "Non-loose unknots, overtwisted discs and the contact mapping class group of $S^3$",
	journal = {Geom. Funct. Anal.},
    volume = {28},
	year = {2018},
    pages = {228–288}
}

@article{CP,
	author = {R. {Casals} and A. {del Pino}},
	title = "Classification of {Engel} knots",
  journal = {Math. Ann.},
	volume = {371},
	year = {2018},
	pages = {391--404}
}

\end{document}